\documentclass[12pt,centertags,oneside,reqno]{amsart}
\usepackage{amsmath,amsthm,amssymb,bm}

\usepackage{amscd,amsxtra,calc}
\usepackage{cmmib57}
\usepackage{url}

\usepackage{amscd}
\usepackage{pstricks}
\usepackage{color}
\usepackage[a4paper,width=16.2cm,top=3cm,bottom=3cm]{geometry}

\numberwithin{equation}{section}

\theoremstyle{plain}

\newtheorem{thm}{Theorem}[section]
\newtheorem{theorem}[thm]{Theorem}
\newtheorem*{thma}{Theorem A}
\newtheorem*{thmb}{Theorem B}

\newtheorem{lemma}[thm]{Lemma}

\newtheorem{proposition}[thm]{Proposition}

\theoremstyle{definition}

\newtheorem{remark}[thm]{Remark}

\newtheorem{definition}[thm]{Definition}

\newtheorem{example}[thm]{Example}

\numberwithin{equation}{section}

\newtheorem{report}{report}
\newtheorem{exercise}{excercise}[section]

\renewcommand{\labelenumi}{\rm{(}\arabic{enumi}\rm{)}}

\usepackage{mleftright}
\usepackage{mdwlist}

\usepackage{etoolbox}

\makeatletter
\def\subsection{\@startsection{subsection}{1}%
  \z@{.5\linespacing\@plus.7\linespacing}{-.5em}%
  {\normalfont\itshape}}
\patchcmd{\@sect}
  {\@addpunct.}
  {}
  {}
  {}
\makeatother

\makeatletter
\def\subsection{\@startsection{subsection}{2}%
  \z@{.5\linespacing\@plus.7\linespacing}{.3\linespacing}%
  {\normalfont\bfseries}}
\makeatother

\usepackage{caption}

\usepackage{amsmath}

\usepackage{mathtools}
\DeclarePairedDelimiter{\abs}{\lvert}{\rvert}

\newcommand{\id}{\mathrm{id}}

\usepackage[pdftex]{graphicx}

\usepackage{threeparttable}

\usepackage{appendix}

\title[]{The minimum value of the first dynamical degrees of automorphisms of complex simple abelian varieties with fixing dimensions}

\author{Yutaro Sugimoto
}
\begin{document}
\maketitle

\begin{abstract}
 We consider the minimum value of the first dynamical degrees, which are larger than $1$, of automorphisms for prime dimensional complex simple abelian varieties.
 Also, we calculate the minimum value of the first dynamical degrees, which are larger than $1$, of automorphisms of complex simple abelian varieties with fixing the dimensions from $2$ to $10$.
\end{abstract}

\section{Introduction}
 Let $X$ be a compact K\"{a}hler manifold and $f\colon X\dashrightarrow X$ be its dominant meromorphic map.
 The first dynamical degree of $f$ is defined by
 \begin{align*}
  \lambda_1(f)=\lim_{n\to+\infty}||(f^n)^{*}:\mathrm{H}^{1,1}(X)\rightarrow\mathrm{H}^{1,1}(X)||^{\frac{1}{n}}
 \end{align*}
 where $||\cdot||$ is a norm and $\mathrm{H}^{1,1}(X)$ is the $(1,1)$-Dolbeault cohomology.
 Especially if $X$ is a complex simple abelian variety and $f$ is an automorphism, 
 \begin{align*}
  \lambda_1(f)&=\lim_{n\to+\infty}||(f^n)^{*}:\mathrm{H}^{1,1}(X)\rightarrow\mathrm{H}^{1,1}(X)||^{\frac{1}{n}}\\
              &=\lim_{n\to+\infty}||(f^{*})^n:\mathrm{H}^{1,1}(X)\rightarrow\mathrm{H}^{1,1}(X)||^{\frac{1}{n}}\\
              &=\rho(f^{*}:\mathrm{H}^{1,1}(X)\rightarrow\mathrm{H}^{1,1}(X))
 \end{align*}
 and we only consider this case in this paper.
 It is known that $\lambda_1(f)\geq1$ (cf.\ \cite{DS05} or \cite{DS17}) and if $g=1$, then $\lambda_1(f)=1$ (cf.\ \cite[Definition 4.1]{DS17}).\par
 In our past study, we prove the next theorem.
 \begin{theorem}[{cf.\ \cite[Theorem 7.2]{Sug23}}]
  For an integer $g\geq2$, the set
  \begin{align*}
   S_g:=\left\{
         \begin{array}{l}
          \text{the first dynamical degrees of automorphisms}\\
          \text{of simple abelian varieties over $\mathbb{C}$ whose dimension is $g$}
         \end{array}
        \right\}\setminus\{1\}
  \end{align*}
  has the minimum value.
 \end{theorem}
 Denote the minimum value of $S_g$ by $m(g)$ and in this paper, we study $m(g)$ for some $g$ in detail.\par
 In Section \ref{The minimum value of the first dynamical degrees for simple abelian varieties with prime dimensions}, we consider the minimum value for prime dimensional simple abelian varieties.
 \begin{thma}
  For a prime number $p$, we get the next result about $m(p)$.
  \begin{align*}
   \begin{cases}
    m(p)=4\mathrm{cos}^2\left(\frac{\pi}{2p+1}\right) & \text{$($if $2p+1$ is also a prime$)$} \\
    4+4^{-2p-3}<m(p)<4.3264=(2.08)^2 & \text{$($otherwise$)$}
   \end{cases}
  \end{align*}
 \end{thma}
 \begin{remark}
  $p$ is called a Sophie Germain prime if $p$ and $2p+1$ are both prime numbers, and it is conjectured that there are infinitely many Sophie Germain primes (cf.\ \cite[Chapter 5.5.5]{Sho09}).
  Conversely, since there are infinitely many prime numbers $p$ such that $p\equiv1\ (\mathrm{mod}\ 3)$, there are infinitely many prime numbers $p$ which are not Sophie Germain primes.
 \end{remark}
 In Section \ref{The minimum value of the first dynamical degrees for simple abelian varieties with lower dimensions}, we consider the minimum value for lower dimensions.
 \begin{thmb}
  For $g=2$ to $10$, the value $m(g)$ is the following.
  \begin{table}[hbtp]
   \label{table}
   \begin{tabular}{|c|c|}
    \hline 
    dimension $g$ & minimum value $m(g)$ \\
    \hline
    $1$ & $(\text{all }1)$\\
    $2$ & $4\mathrm{cos}^2(\pi/5)=2.6180\cdots$ \\
    $3$ & $4\mathrm{cos}^2(\pi/7)=3.2469\cdots$ \\
    $4$ & $2\mathrm{cos}(\pi/5)=1.6180\cdots$ \\
    $5$ & $4\mathrm{cos}^2(\pi/11)=3.6825\cdots$ \\
    $6$ & $2\mathrm{cos}(\pi/7)=1.8019\cdots$ \\
    $7$ & $4.0333\cdots$ \\
    $8$ & $2\mathrm{cos}(\pi/5)=1.6180\cdots$ \\
    $9$ & $1.1509\cdots^2=1.3247\cdots$ \\
    $10$& $1.1762\cdots^2=1.3836\cdots$ \\
    \hline
   \end{tabular}
  \end{table}
 \end{thmb}
 \medskip\noindent
 {\bf Acknowledgements.} The author thanks Professor Keiji Oguiso for suggesting the direction of this paper.

\section{Algebraic preliminaries}
 This section is devoted to the algebraic preparations for the following sections.
 Some of the results are also used in our past study (cf. \cite{Sug23}).
 \begin{definition}[{cf.\ \cite[Chapter 1.2]{Lan83}}]\label{equivalent condition for CM-field}
  A CM-field is a number field $K$ which satisfies the following equivalent conditions.
  \begin{enumerate}
   \item $K$ is a totally complex field, and is a quadratic extension of some totally real number field.
   \item $K$ is closed under the complex conjugation and all $\mathbb{Q}$-embeddings $K\hookrightarrow\mathbb{C}$ commute with the complex conjugation, and $K$ is not real.
  \end{enumerate}
 \end{definition}
 For a number field $K$, denote its ring of integer by $\mathcal{O}_K$ and define
 \begin{align*}
  U_K:=\{x\in\mathcal{O}_K\setminus\{0\}\mid x^{-1}\in\mathcal{O}_K\}
 \end{align*}
 as the group of algebraic units.\par
 The next lemma follows from the definition of CM-fields. 
 \begin{lemma}[{cf.\ \cite[Chapter 1.2]{Lan83}}]\label{unit element}
  Let $L$ be a totally real number field.
  \begin{enumerate}
   \item If $L'$ is a totally complex quadratic extension of $L$, then for $x\in U_{L'}$, $\abs{x}^2$ is an element of $U_L$.
   \item Fix $a\in U_L\setminus\{0\}$.
         Then the next two conditions are equivalent.
         \begin{enumerate}
          \item There exist a totally complex quadratic extension $L'/L$ and $x\in U_{L'}$ which satisfies $\abs{x}^2=a$.
	  \item The conjugate elements of $a\in L$ are all positive (i.e., $a$ is totally positive).
         \end{enumerate}
	 Moreover, these conditions are equivalent to the totally realness of $L(\sqrt{a})$.
  \end{enumerate}
 \end{lemma}
 \begin{proof}
  We start with the proof of (1).
  By the assumption, $x, \frac{1}{x}\in \mathcal{O}_{L'}$.
  Since $[L':L]=2$ and $L$ is totally real, $x\overline{x}, \frac{1}{x\overline{x}}\in\mathcal{O}_{L}$.
  Thus, $\abs{x}^2=x\overline{x}\in U_L$.\par
  Next, consider the proof of (2). 
  If all conjugate elements of $a\in L$ are positive, then by defining $L':=L(\sqrt{-a})$, $L'$ is a totally complex number field.
  This implies (ii)$\Rightarrow$(i).\par
  Assume the condition (i).
  Fix a $\mathbb{Q}$-embedding $\sigma':L\hookrightarrow\mathbb{C}$ and take a $\mathbb{Q}$-embedding $\sigma:L'\hookrightarrow\mathbb{C}$ with $\sigma'=\sigma\circ i$ where $i:L\hookrightarrow L'$ is the natural inclusion.
  By Definition \ref{equivalent condition for CM-field}, the complex conjugation commutes with $\sigma'$ and so
  \begin{align*}
   \sigma'(a)=\sigma'(\abs{x}^2)=\sigma'(x\overline{x})=\sigma'(x)\sigma'(\overline{x})=\sigma'(x)\overline{\sigma'(x)}=\abs{\sigma'(x)}^2.
  \end{align*}
  Thus, all the conjugates of $a$ is positive and this implies (i)$\Rightarrow$(ii) and the proof is concluded.
 \end{proof}
 In this paper, we denote $\zeta_n=\mathrm{cos}\left(\frac{2\pi}{n}\right)+i\cdot\mathrm{sin}\left(\frac{2\pi}{n}\right)$ for a positive integer $n$.
 Denote the minimal polynomial of $\zeta_n$ over $\mathbb{Q}$ by $\Phi_n(x)\in\mathbb{Z}[x]$ and called a cyclotomic polynomial.
 Now the degree of $\Phi_n(x)$ is Euler's totient function $\varphi(n)$.
 Let $\Psi_n(x)$ be the minimal polynomial of $2\mathrm{cos}\left(\frac{2\pi}{n}\right)$ over $\mathbb{Q}$ (define $\Psi_4(x)=x$ for $n=4$).
 The constant term of $\Psi_n(x)$ is calculated in \cite{ACR16} as below.
 \begin{proposition}[{cf.\ \cite{ACR16}}]\label{constant term of minimal polynomial}
  The absolute value of the constant term of $\Psi_n(x)$ is equal to $1$ except the following cases.
  \begin{enumerate}
  \renewcommand{\labelenumi}{\rm{(\roman{enumi})}}
   \item $\abs{\Psi_n(0)}=0$ for $n=4$
   \item $\abs{\Psi_n(0)}=2$ for $n=2^m$, with $m\in\mathbb{Z}_{\geq0}\setminus\{2\}$
   \item $\abs{\Psi_n(0)}=p$ for $n=4p^k$, with $k\in\mathbb{Z}_{>0},\ p\text{ is an odd prime number}$
  \end{enumerate}
 \end{proposition}
 Also for $n\geq3$, the polynomial $\Psi_n(x)$ satisfies
 \begin{align*}
  \Phi_n(x)=x^{\frac{\varphi(n)}{2}}\Psi_n\left(x+\frac{1}{x}\right)
 \end{align*}
 and $\Psi_n(x)$ is of degree $\frac{\varphi(n)}{2}$.\par
 The next lemma, which follows from Kronecker's theorem, is useful to determine the minimum value of the first dynamical degrees.
 \begin{lemma}\label{cyclotomic}
  Let $P(x)=x^n+a_1x^{n-1}+\cdots+a_n\in\mathbb{Z}[x]$ be an irreducible monic polynomial whose roots are all real. 
  If the maximal absolute value of the roots of $P(x)$ is less than $2$, then $P(x)$ is a polynomial which satisfies $x^nP\left(x+\frac{1}{x}\right)=\Phi_N(x)$ for some $N\in\mathbb{Z}_{\geq3}$.
  This implies $P(x)=\Psi_N(x)$.
  Also, the roots of $P(x)$ can be written as $\zeta_N^m+\frac{1}{\zeta_N^m}=2\mathrm{cos}(\frac{2m\pi}{N})$ for some $m\in\mathbb{Z}_{>0}$ with $m<N$, $\mathrm{gcd}(N,m)=1$.
  Moreover, the maximal absolute value of the roots of $P(x)$ can be written as
  \begin{align*}
   \left\{
    \begin{array}{ll}
     2\mathrm{cos}\left(\frac{2\pi}{N}\right) & (\text{if }N\text{ is even})\\ [+5pt]
     2\mathrm{cos}\left(\frac{\pi}{N}\right) & (\text{if }N\text{ is odd})
    \end{array}.
   \right.
  \end{align*}
 \end{lemma}
 Let $L/K$ be a field extension of number fields with $[L:K]=n$ and $\mathcal{O}_L$ and $\mathcal{O}_K$ be their ring of integers.
 Then $\mathcal{O}_L$ and $\mathcal{O}_K$ are Dedekind domains.
 For a non-zero prime ideal $\mathfrak{p}\subset\mathcal{O}_K$, there is the prime ideal factorization $\mathfrak{p}\mathcal{O}_L=\mathfrak{P}_1^{e_1}\cdots\mathfrak{P}_g^{e_g}$ since $\mathcal{O}_L$ is a Dedekind domain.
 We say that $\mathfrak{P}_i$ is over $\mathfrak{p}$.\par
 Now $L/K$ is a separable extension of degree $n$ and so
 \begin{align*}
  n=\sum_{i=1}^g e_if_i.
 \end{align*}
 Moreover, if $L/K$ is a Galois extension, by Hilbert's ramification theory (cf.\ \cite[Chapter I, \S9]{Neu99}), 
 \begin{align*}
  e_1=\cdots=e_g=e, f_1=\cdots=f_g=f
 \end{align*}
 hold and this implies $n=efg$.
 \begin{remark}[{cf.\ \cite[Chapter I, Proposition 8.3]{Neu99}}]\label{construction of prime ideal factorization}
  Let $L/K$ be a field extension of number fields and let $\theta\in\mathcal{O}_L$ be a primitive element for $L/K$ (i.e., $L=K(\theta)$).\par
  Denote the conductor of $\mathcal{O}_K[\theta]$ by $\mathfrak{J}\subset\mathcal{O}_L$, in other words,
  \begin{align*}
   \mathfrak{J}=\{\alpha\in\mathcal{O}_L\mid\alpha\mathcal{O}_L\subset\mathcal{O}_K[\theta]\}.
  \end{align*}
  Under this condition, let $\mathfrak{p}\subset\mathcal{O}_K$ be a prime ideal which is relatively prime to $\mathfrak{J}$ (i.e., $\mathfrak{p}\mathcal{O}_L+\mathfrak{J}=\mathcal{O}_L$).
  Denote the minimal polynomial of $\theta$ over $\mathcal{O}_K$ by $p(X)\in\mathcal{O}_K[X]$ and denote the image of $p(X)$ via the map $\mathcal{O}_K[X]\rightarrow(\mathcal{O}_K/\mathfrak{p})[X]$ by $\overline{p}(X)$.
  $\overline{p}(X)$ is decomposed as 
  \begin{align*}
   \overline{p}(X)=\overline{p_1}(X)^{e_1}\cdots\overline{p_g}(X)^{e_g}
  \end{align*}
  in $(\mathcal{O}_K/\mathfrak{p})[X]$ for some $p_i(X)\in\mathcal{O}_K[X]$.
  Then, by defining
  \begin{align*}
   \mathfrak{P}_i:=\mathfrak{p}\mathcal{O}_L+p_i(\theta)\mathcal{O}_L,
  \end{align*}
  $\mathfrak{P}_i\subset\mathcal{O}_L$ are all different prime ideals over $\mathfrak{p}$ and
  \begin{align*}
   \mathfrak{p}\mathcal{O}_L=\mathfrak{P}_1^{e_1}\cdots\mathfrak{P}_g^{e_g}
  \end{align*}
  holds and this is the prime ideal factorization of $\mathfrak{p}$ in $L$.
  Also, $f_i=[\mathcal{O}_L/\mathfrak{P}_i:\mathcal{O}_K/\mathfrak{p}]=\mathrm{deg}(p_i(X))$ holds.
 \end{remark}
 Let $L/K$ be an extension of number fields with $[L:K]=n$ again.
 Let $\{\sigma_i\}_{1\leq i\leq n}$ be the set of all $K$-embeddings $L\hookrightarrow\mathbb{C}$.\par
 For $a_1,\ldots,a_n\in L$, the discriminant over $L/K$ is defined as
 \begin{align*}
  \Delta_{L/K}(a_1,\ldots,a_n):=\mathrm{det}\begin{pmatrix}
                                             \sigma_1(a_1) & \sigma_1(a_2) & \dots  & \sigma_1(a_n) \\
                                             \sigma_2(a_1) & \sigma_2(a_2) & \dots  & \sigma_2(a_n) \\
                                             \vdots        & \vdots        & \ddots & \vdots \\
                                             \sigma_n(a_1) & \sigma_n(a_2) & \dots  & \sigma_n(a_n)
                                            \end{pmatrix}^2.
 \end{align*}
 Then, $\Delta_{L/K}(a_1,\ldots,a_n)\in K$ and $\Delta_{L/K}(a_1,\ldots,a_n)\neq0$ if and only if $a_1,\ldots,a_n\in L$ is linearly independent over $K$.
 Now, the next theorem holds.
 \begin{lemma}[{cf.\ \cite[Chapter I, Lemma 2.9]{Neu99}}]\label{discriminant}
  Let $L/K$ be an extension of number fields with $[L:K]=n$ and $a_1,\ldots,a_n\in\mathcal{O}_L$ be a basis of $L/K$.
  Define $\delta:=\Delta_{L/K}(a_1,\ldots,a_n)$ and then
  \begin{align*}
   \mathcal{O}_L\subset\frac{a_1}{\delta}\mathcal{O}_K\oplus\cdots\oplus\frac{a_n}{\delta}\mathcal{O}_K.
  \end{align*}
 \end{lemma}
 \begin{remark}\label{conductor and discriminant}
  The conductor defined in Remark \ref{construction of prime ideal factorization} contains $\delta=\Delta_{L/K}(1,\theta,\ldots,\theta^{n-1})$, where $n=[L:K]$, by Lemma \ref{discriminant}.
 \end{remark}
 \begin{definition}[{cf.\ \cite[Definition 3.2.1, Definition 8.4.1]{Voi21}}]
  Let $B$ be a finite-dimensional  algebra over $\mathbb{Q}$.
  An anti-involution is a $\mathbb{Q}$-linear map $\phi:B\rightarrow B$ such that
  \begin{enumerate}
   \item $\phi(1)=1$
   \item $\phi(\phi(x))=x$\quad($x\in B$)
   \item $\phi(xy)=\phi(y)\phi(x)$\quad($x,y\in B$)
  \end{enumerate}
  An anti-involution $\phi:B\rightarrow B$ is called positive anti-involution over $\mathbb{Q}$ if and only if $\mathrm{Tr}_{B/\mathbb{Q}}(\phi(x)x)>0$ for all $x\in B$.
 \end{definition}
 Also, in this paper, we adopt the next definition of Salem numbers.
 \begin{definition}[{\cite[Chapter 5.2]{BDGPS92}}]
  A Salem number is a real algebraic integer $\lambda$ greater than $1$ whose other conjugates have modulus at most equal to $1$, at least one having a modulus equal to $1$.
 \end{definition}
 This implies that all Salem numbers have degrees at least $4$.

\section{Endomorphism algebras of simple abelian varieties}\label{Endomorphism algebras of simple abelian varieties}
 Let $X$ be a $g$-dimensional simple abelian variety and then its endomorphism algebra $B:=\mathrm{End}_{\mathbb{Q}}(X)=\mathrm{End}(X)\otimes_\mathbb{Z}\mathbb{Q}$ is a division ring.
 $B$ is a finitely generated algebra over $\mathbb{Q}$ with a center $K$, which is a field.
 Also, $B$ admits the Rosati involution and it is a positive anti-involution $\phi:B\rightarrow B$ over $\mathbb{Q}$.
 Now $\phi$ can be restricted to $K$ and define $K_0:=\{x\in K\mid \phi(x)=x\}$.
 Then $K_0$ is a totally real number field and $[K:K_0]$ is $1$ or $2$.
 This deduction gives the classification of simple abelian varieties by their endomorphism algebras as below (cf.\ \cite[Chapter 5.5]{BL04}).
 Denote $[B:K]=d^2$, $[K:\mathbb{Q}]=e$ and $[K_0:\mathbb{Q}]=e_0$.
 \begin{table}[h]
  \caption{Classification of $X$ by its endomorphism algebra}
  \label{table1}
  \begin{tabular}{|c|c|c|c|c|c|}
   \hline
   $X$ & $B=\mathrm{End}_\mathbb{Q}(X)$ & $K$ & $d$ & $e_0$ & restriction \\
   \hline
   Type 1 & $K$                                             & totally real & $1$ & $e$           & $e\mid g$\\
   Type 2 & totally indefinite quaternion algebra over $K$  & totally real & $2$ & $e$           & $2e\mid g$\\
   Type 3 & totally definite quaternion algebra over $K$    & totally real & $2$ & $e$           & $2e\mid g$\\
   Type 4 & division ring with center $K$                   & CM-field     & $d$ & $\frac{e}{2}$ & $\frac{d^2 e}{2}\mid g$\\
   \hline
  \end{tabular}
 \end{table}\par
 An endomorphism $f:X\rightarrow X$ can be identified as an element $\alpha\in B$ with the minimal polynomial $p(x)\in\mathcal{O}_K[x]$ over $K$ and the minimal polynomial $P(x)\in\mathbb{Z}[x]$ over $\mathbb{Q}$.
 If $f$ is an automorphism, the constant term of $p(x)$ is in $U_K$ and the constant term of $P(x)$ is in $\{\pm1\}$.
 The degree of $p(x)$ is at most $d$ and the degree of $P(x)$ is at most $de\leq 2g$.
 Also, the first dynamical degree of the endomorphism $f$ is the square of the maximal absolute value of the roots of $P(x)\in\mathbb{Z}[x]$ (cf.\ \cite[Section 3]{Sug23}).\par
 Conversely, for a division algebra in Table \ref{table1} with some property, we have a simple abelian variety whose endomorphism algebra is the same with the division algebra by the following propositions (cf.\ \cite[Chapter 9]{BL04}).
 The same methods are also used in \cite{Sug23}.
 \begin{proposition}[Type 1]\label{construction for Type 1}
  Let $K$ be a totally real number field with $[K:\mathbb{Q}]=e$ and fix a $\mathbb{Z}$-order $\mathcal{O}$.
  Then, for any ineger $m\in\mathbb{Z}_{>0}$, there exists an $em$-dimensional simple abelian variety $X$ with an isomorphism $K\stackrel{\simeq}{\longrightarrow}\mathrm{End}_\mathbb{Q}(X)$ which induces an injective ring homomorphism $\mathcal{O}\hookrightarrow\mathrm{End}(X)$.
 \end{proposition}
 \begin{proposition}[Type 2]\label{construction for Type 2}
  Let $B$ be a totally indefinite quaternion algebra over a totally real number field $K$ with $[K:\mathbb{Q}]=e$ and fix a $\mathbb{Z}$-order $\mathcal{O}$.
  Assume there exists a positive anti-involution $\phi:B\rightarrow B$ over $\mathbb{Q}$ and fix a positive integer $m\in\mathbb{Z}_{>0}$.
  Then there exists a $2em$-dimensional simple abelian variety $X$ with an isomorphism $B\stackrel{\simeq}{\longrightarrow}\mathrm{End}_\mathbb{Q}(X)$ which induces an injective ring homomorphism $\mathcal{O}\hookrightarrow\mathrm{End}(X)$.
 \end{proposition}
 \begin{proposition}[Type 4]\label{construction for Type 4}
  Let $B$ be a central simple division algebra over a CM-field $K$ with $[B:K]=d^2$, $[K:\mathbb{Q}]=e=2e_0$ and fix a $\mathbb{Z}$-order $\mathcal{O}$.
  Assume there exists\nolinebreak\ a positive anti-involution $\phi:B\rightarrow B$ over $\mathbb{Q}$ of the second kind (i.e., $\phi\lvert_K\neq\id_K$) and fix a positive integer $m\in\mathbb{Z}_{>0}$.
  Then there exists a $d^2e_0m$-dimensional abelian variety $X$ with an injective ring homomorphism $B\hookrightarrow\mathrm{End}_\mathbb{Q}(X)$ which induces an injective ring homomorphism $\mathcal{O}\hookrightarrow\mathrm{End}(X)$.\par
  In addition, assume one of the next conditions.
  \begin{enumerate}
   \item $dm\geq3$
   \item $dm=2$ and $e_0\geq2$
  \end{enumerate}
  Then the above $X$ can be taken as a simple abelian variety and $B\simeq\mathrm{End}_\mathbb{Q}(X)$.
 \end{proposition}
 \begin{remark}
  The construction of simple abelian varieties of Type 3 is also written in \cite[Chapter 9]{BL04}, but it is not used in this paper, so we omit here.
 \end{remark}
 The construction of a division ring is due to the next theorems.
 \begin{theorem}[{\cite[Theorem 2.5]{DH22}}]\label{divisional}
  Let $F$ be a totally real number field and $K=F(\sqrt{a})$ a quadratic extension for $a\in\mathcal{O}_F$.
  Then there exists a prime number $p$ such that the quaternion algebra $B=\left(\frac{a,p}{F}\right)$ is divisional.
 \end{theorem}
 \begin{theorem}[{cf.\ \cite[Chapter 15.1, Corollary d]{Pie82}}]\label{construct division algebra}
  Let $E/F$ be a cyclic extension of fields with $[E:F]=n$.
  Define $G=\mathrm{Gal}(E/F)$ and denote its generator by $\sigma$.
  Fix $u\in F^{\times}$ and take a symbol $v$ as $v^n=u$.
  Define $B=\oplus_{i=0}^{n-1} v^{i}E$ with the multiplication on $B$ as $e\cdot v=v\cdot\sigma(e)$ for $e\in E$.
  Assume that the order of $u$ in $F^{\times}/\mathrm{N}_{E/F}(E^{\times})$ is exactly $n$.\par
  Then $B$ is a division ring and also a central simple algebra over $F$ with $[B:F]=n^2$.
 \end{theorem}
 The next lemma is used for searching $u\in F^{\times}$ with $u\notin\mathrm{N}_{E/F}(E^{\times})$.
 \begin{lemma}[{\cite[Lemma 8.9]{Sug23}}]\label{multiplicity}
  Let $E/F$ be a Galois extension of number fields with $[E:F]=n$ and let $\mathfrak{q}\subset\mathcal{O}_E$ be a prime ideal and take $\alpha\in E^{\times}\setminus U_E$.\par
  Assume $\sigma(\mathfrak{q})=\mathfrak{q}$ for all $\sigma\in\mathrm{Gal}(E/F)$.
  Then, the multiplicity of $\mathfrak{q}$ of the prime ideal factorization of the ideal generated by $\mathrm{N}_{E/F}(\alpha)$ in $\mathcal{O}_E$ is a multiple of $n$.
 \end{lemma}
 Also, the construction of a positive anti-involution is reduced to the following theorems in this paper.
 \begin{theorem}[{cf.\ \cite[Theorem 5.5.3]{BL04}}]\label{construction of positive anti-involution}
  Let $B$ be a totally indefinite quaternion algebra of finite dimension over $\mathbb{Q}$ with center a totally real number field $K$.
  Assume that $B$ is divisional.
  Then, $\mathbb{Q}$-linear map $\phi:B\rightarrow B$ is a positive anti-involution over $\mathbb{Q}$ if and only if it can be written as
  \begin{align*}
   \phi(x)=c^{-1}\overline{x}c  
  \end{align*}
  where $c\in B\setminus K$ with $c^2\in K$ totally negative and $x\mapsto\overline{x}$ is the quaternion conjugation.
 \end{theorem}
 \begin{theorem}[{cf.\ \cite[Chapter 21]{Mum70}, \cite[Theorem 5.5.6]{BL04}}]\label{existence of positive anti-involution}
  Let $B$ be a division algebra of finite dimension over $\mathbb{Q}$ with center a CM-field $K$.
  Assume that there exists an anti-involution $\phi:B\rightarrow B$ of the second kind.
  Then there exists a positive anti-involution $\phi':B\rightarrow B$ over $\mathbb{Q}$ of the second kind.
 \end{theorem}

\section{The minimum value of the first dynamical degrees of automorphisms}\label{The minimum value of the first dynamical degrees of automorphisms}
 First we consider the general results for considering the minimum value of the first dynamical degrees of automorphisms.
 The flow is repeatedly used in the following sections.
 Let $X$ be a $g$-dimensional simple abelian variety and consider the automorphisms and its first dynamical degrees for each type in Table \ref{table1}.
 Define $B$, $K$, $K_0$, $d$, $e$ and $e_0$ as in Section \ref{Endomorphism algebras of simple abelian varieties}.
 \begin{flushleft}{\bf{Type 1}}\end{flushleft}\par
  Since $B=K$ is a totally real number field, so an automorphism $f:X\rightarrow X$ can be identified as an element $\alpha\in U_K$ and its conjugates are all real.
  The degree of $\alpha$ divides $e$ and by the restricton in Table \ref{table1}, it also divides $g=\mathrm{dim}(X)$.
  Assuming that the first dynamical degree is less than $4$, then the conjugates of $\alpha$ are all in the interval $(-2,2)$.
  Thus, by Lemma \ref{cyclotomic}, it is reduced to the deduction of cyclotomic polynomials.\par
  Conversely, let $K$ be a totally real number field with $[K:\mathbb{Q}]=e$ and take $\alpha\in U_K$ and $m\in\mathbb{Z}_{>0}$.
  Then by Proposition \ref{construction for Type 1}, there exists an $em$-dimensional simple abelian variety $X$ with the endomorphism algebra $K=\mathrm{End}_\mathbb{Q}(X)$ with $\mathcal{O}_K\subset\mathrm{End}(X)$.
  Thus for the endomorphism $f:X\rightarrow X$ which corresponds to $\alpha$, the condition $\alpha\in U_K$ implies $\frac{1}{\alpha}\in\mathcal{O}_K$ and so $f:X\rightarrow X$ is an automorphism.\par
  Therefore, the minimum value of the first dynamical degrees for this type is caluculated by
  \begin{align*}
   \mathrm{min}\left\{
                \begin{array}{ll}
                 \text{the square of the maximal absolute value of the conjugates}\\
                 \text{of an algebraic unit, whose conjugates are all real, of degree divides $g$}
                \end{array}
               \right\}\setminus\{1\}
  \end{align*}
 \begin{flushleft}{\bf{Type 2, Type 3}}\end{flushleft}\par
  An automorphism $f$ corresponds to an element $\alpha\in B$ and since $B$ is a quaternion algebra over a totally real number field $K$, $\alpha$ has the minimal polynomial $p(x)\in\mathcal{O}_K[x]$ over $K$ of degree $1$ or $2$.\par
  If the degree is $1$, then $\alpha\in U_K$ and so the deduction about the minimum value is reduced to the case of Type 1.\par
  If the degree is $2$, then it can be written as $p(x)=x^2+ax+b\in\mathcal{O}_K[x]$ with $b\in U_K$.
  Now the minimal poynomial of $\alpha$ over $\mathbb{Q}$ satisfies
  \begin{align*}
   {P(x)}^m=\prod_{i=1}^e(x^2+\sigma_i(a)x+\sigma_i(b))\in\mathbb{Z}[x]
  \end{align*}
  for some $m\in\mathbb{Z}_{>0}$, where $\{\sigma_i\}$ is the set of all $\mathbb{Q}$-embeddings $K\hookrightarrow\mathbb{C}$.\par
  The first dynamical degree of $f$ is the square of the maximal absolute value of the roots of $P(x)$ and we consider the following three cases.\\
  $\bullet$ there exists $1\leq i\leq e$ such that $\abs{\sigma_i(b)}>1$\par
   If $\abs{\sigma_i(b)}>1$, one of the roots of $x^2+\sigma_i(a)x+\sigma_i(b)$ has absolute value not less than $\sqrt{\abs{\sigma_i(b)}}$ and so the first dynamical degree is not less than $\abs{\sigma_i(b)}$.
   $\sigma_i(b)$ are all real and so by assuming that the first dynamical degree is less than $2$, $\sigma_i(b)$ are all in the interval $(-2,2)$ and the remaining is same as the case of Type 1.\\
  $\bullet$ $\sigma_i(b)=1$ for all $1\leq i\leq e$\par
   The roots of $x^2+\sigma_i(a)x+1$ have modulus $1$ if and only if $-2\leq\sigma_i(a)\leq2$.
   Thus, for small first dynamical degrees except $1$, $\mathrm{max}\{\abs{\sigma_i(a)}\}$ must be larger than $2$ and close to $2$.\\
  $\bullet$ $\sigma_i(b)=-1$ for all $1\leq i\leq e$\par
   The roots of $x^2+\sigma_i(a)x-1$ have modulus $1$ if and only if $\sigma_i(a)=0$.
   Thus, for the small first dynamical degrees except $1$, $\mathrm{max}\{\abs{\sigma_i(a)}\}$ must be larger than $0$ and close to $0$.
   Since $\prod_{i=1}^e \sigma_i(a)\in\mathbb{Z}$, $\mathrm{max}\{\abs{\sigma_i(a)}\}\geq1$ for $a\neq0$ and the equality is always achieved by $a=1$.
   For this case, the minimum value of the maximal absolute value of the conjugates is $\frac{1+\sqrt{5}}{2}$.
   Thus, $\frac{3+\sqrt{5}}{2}$ is a lower bound of the first dynamical degrees for this case.\par
  The realizability would be considered independently in the following sections.
 \begin{flushleft}{\bf{Type 4}}\end{flushleft}\par
  For the case $d=1$, $B=K$ is a CM-field and an automorphism of $X$ corresponds to $\alpha\in U_{K}$.\par
  If $\alpha\in K_0$, the deduction is reduced to the case of Type 1.\par
  If $\alpha\in K\setminus K_0$, by Lemma \ref{unit element}, for denoting $a=\abs{\alpha}^2\in U_{K_0}$, $K_0(\sqrt{a})$ is a totally real number field.
  The first dynamical degree of $\alpha$ is the square of the maximal absolute value of the conjugate elements of $\alpha$, and so this is equal to the maximal absolute value of the conjugates of $a\in K_0$.
  In order to search the small first dynamical degree, assume the maximal absolute value of the conjugate elements of $\sqrt{a}$ is less than $2$.
  If $a=1$, then the first dynamical degree would be $1$ and so we assume $a\neq1$.
  Then, as in Lemma \ref{cyclotomic}, it can be written as $\sqrt{a}=\zeta_N^m+\frac{1}{\zeta_N^m}$ ($m,N\in\mathbb{Z}_{>0}$) with $m<N$, $\mathrm{gcd}(N,m)=1$ and we can assume $N\neq1,2,3,4,6$.
  Thus, $a=\left(\zeta_N^m+\frac{1}{\zeta_N^m}\right)^2=\zeta_N^{2m}+\frac{1}{\zeta_N^{2m}}+2$ and so
  \begin{align*}
   [\mathbb{Q}(a):\mathbb{Q}]=\left\{
                               \begin{array}{ll}
                                \frac{1}{2}\varphi(\frac{N}{2}) & (N:\text{even})\\ [+5pt]
                                \frac{1}{2}\varphi(N) & (N:\text{odd})
                               \end{array}.
                              \right.
  \end{align*} 
  Now $K_0\supset\mathbb{Q}(a)\supset\mathbb{Q}$ and so we get the candidates of $N$ from this condition.\par
  Conversely, let $K$ be a CM-field with $[K:\mathbb{Q}]=e=2e_0$ and take $\alpha\in U_K$ and fix $m\in\mathbb{Z}_{>0}$.
  Now $K$ admits the complex conjugate as a positive anti-involution of the second kind.
  If $m\in\mathbb{Z}_{\geq3}$ (or $e\geq4$ and $m\in\mathbb{Z}_{\geq2}$) holds, then by Proposition \ref{construction for Type 4}, there exist an $e_0m$-dimensional simple abelian variety $X$ and an automorphism $f:X\rightarrow X$ which corresponds to $\alpha$.\par
  The case $d\geq2$ would be considered independently in the following sections.

\section{The minimum value of the first dynamical degrees for simple abelian varieties with prime dimensions}\label{The minimum value of the first dynamical degrees for simple abelian varieties with prime dimensions}
 Define
 \begin{align*}
   m(g):=\mathrm{min}\left\{
                      \begin{array}{l}
                       \text{the first dynamical degrees of automorphisms}\\
                       \text{of simple abelian varieties over $\mathbb{C}$ whose dimension is $g$}
                      \end{array}
                     \right\}\setminus\{1\}.
  \end{align*}\par
 We prove the next result in this section.
 \begin{theorem}\label{Result1}
  For a prime number $p$, we get the next result about $m(p)$.
  \begin{align*}
   \begin{cases}
    m(p)=4\mathrm{cos}^2\left(\frac{\pi}{2p+1}\right) & \text{$(2p+1$ is also a prime$)$} \\
    4+4^{-2p-3}<m(p)<(2.08)^2 & \text{$($otherwise$)$}
   \end{cases}
  \end{align*}
 \end{theorem}
 \begin{remark}
  Especially, $(2.08)^2$ is an upper bound of $m(p)$ for any prime number $p$.
 \end{remark}
 First, we consider the case $p=2$.
 In other words, we prove the next theorem.
 \begin{proposition}\label{case=2}
  \begin{align*}
   m(2)=4\mathrm{cos}^2\left(\frac{\pi}{5}\right)
  \end{align*}
 \end{proposition}
 \begin{proof}
  We consider by dividing into the type of simple abelian varieties in Table \ref{table1}.\par
  Define $B$, $K$, $K_0$, $d$, $e$ and $e_0$ as in Section \ref{Endomorphism algebras of simple abelian varieties} for each type of $X$.
  \begin{flushleft}{\bf{Type 1}}\end{flushleft}\par
   Let $X$ be a $2$-dimensional simple abelian variety of Type 1 and $f\in\mathrm{End}(X)$ be an automorphism of $X$ which corresponds to $\alpha\in U_K$ as in Section \ref{The minimum value of the first dynamical degrees of automorphisms}.
   Now $[K:\mathbb{Q}]=1$ or $2$ by the restriction in Table \ref{table1}.\par
   Assuming that the first dynamical degree is less than $4$, the conjugates of $\alpha$ are all in the interval $(-2,2)$.
   If the degree of $\alpha$ is $1$, then $\alpha=\pm1$ and so the first dynamical degree is $1$ and so we assume the degree of $\alpha$ is $2$.
   Denote the minimal polynomial of $\alpha$ over $\mathbb{Q}$ by $P(x)$, whose degree is $2$.
   By Lemma \ref{cyclotomic}, $x^2P(x+\frac{1}{x})$ is an irreducible cyclotomic polynomial and its degree is $4$.
   Cyclotomic polynomials $\Phi_N(x)$ have degree $4$ only for $N=5,8,10,12$.
   By comparing with Proposition \ref{constant term of minimal polynomial}, since $P(x)$ has constant term $\pm1$, $N=5,10$ are only possible.
   Thus, the minimum value of the maximal absolute value of the roots of $P(x)$ is $2\mathrm{cos}\left(\frac{\pi}{5}\right)$ except $1$.
   The realizability is concerned in Section \ref{The minimum value of the first dynamical degrees of automorphisms}.
   Thus, the minimum value of the first dynamical degrees is $4\mathrm{cos}^2\left(\frac{\pi}{5}\right)=\left(\frac{1+\sqrt{5}}{2}\right)^2$ except $1$ for this type.
  \begin{flushleft}{\bf{Type 2, Type 3}}\end{flushleft}\par
   Let $X$ be a $2$-dimensional simple abelian variety of Type 2 or 3 and $f\in\mathrm{End}(X)$ be an automorphism of $X$ which corresponds to $\alpha\in B$.
   Now $[K:\mathbb{Q}]=1$ by the restriction in Table \ref{table1} and so $K=\mathbb{Q}$.\par
   Denote the minimal polynomial of $\alpha$ over $\mathbb{Q}$ by $P(x)$.
   If the degree of $P(x)$ is $1$, then $\alpha\in U_\mathbb{Q}=\{\pm1\}$ and so the first dynamical degree is $1$.\par
   If the degree of $P(x)$ is $2$, then it can be written as $P(x)=x^2+ax+b\in\mathbb{Z}[x](b\in\{\pm1\})$ and consider the following 3 cases as in Section \ref{The minimum value of the first dynamical degrees of automorphisms}.
   Let $\sigma_1$ be the $\mathbb{Q}$-embedding $\mathbb{Q}\hookrightarrow\mathbb{C}$.\\
   $\bullet$ there exists $1\leq i\leq e$ such that $\abs{\sigma_i(b)}>1$\par
    This case would not be occured for this condition.\\
   $\bullet$ $\sigma_i(b)=1$ for all $1\leq i\leq e$\par
    Since $\mathrm{max}\{\abs{\sigma_i(a)}\}$ must be larger than $2$ and close to $2$, $a=3$ provides the minimum value of the maximal absolute value of the conjugates of $\alpha$.
    Thus, $\left(\frac{3+\sqrt{5}}{2}\right)^2$ is a lower bound except $1$ for this case.\\
   $\bullet$ $\sigma_i(b)=-1$ for all $1\leq i\leq e$\par
    $a=1$ provides the minimum value of the first dynamical degrees as in Section \ref{The minimum value of the first dynamical degrees of automorphisms}.
    Thus, $\left(\frac{1+\sqrt{5}}{2}\right)^2$ is a lower bound except $1$ for this case.\par
   Thus, $\left(\frac{1+\sqrt{5}}{2}\right)^2$ is a lower bound of the first dynamical degrees except $1$ for these types.
  \begin{flushleft}{\bf{Type 4}}\end{flushleft}\par
   Let $X$ be a $2$-dimensional simple abelian variety of Type 4 and $f\in\mathrm{End}(X)$ be an automorphism of $X$ which corresponds to $\alpha\in U_K$.
   Now $d=1$ and $[K:\mathbb{Q}]=2$ or $4$ by the restriction in Table \ref{table1} and then $[K_0:\mathbb{Q}]=1$ or $2$, respectively.
   Denote the minimal polynomial of $\alpha$ over $\mathbb{Q}$ by $P(x)$.\par
   If $\alpha\in K_0$, the minimum value of the square of the maximal absolute value of the roots of $P(x)$ is $\left(\frac{1+\sqrt{5}}{2}\right)^2$ except $1$, by the deduction in Type 1.\par
   If $\alpha\in K\setminus K_0$, by denoting $a=\abs{\alpha}^2\in U_{K_0}$, as in Section \ref{The minimum value of the first dynamical degrees of automorphisms}, $\mathbb{Q}(\sqrt{a})$ is a totally real number field.
   If $a=1$, then the first dynamical degree would be $1$ and so we assume $a\neq1$.
   Also, by assuming that the maximal absolute value of the conjugates of $\sqrt{a}$ is less than $2$, it can be written as $\sqrt{a}=\zeta_N^m+\frac{1}{\zeta_N^m}$ ($m,N\in\mathbb{Z}_{>0}$) with $m<N$, $\mathrm{gcd}(N,m)=1$ and we can assume $N\neq1,2,3,4,6$.
   Thus, $a=\left(\zeta_N^m+\frac{1}{\zeta_N^m}\right)^2=\zeta_N^{2m}+\frac{1}{\zeta_N^{2m}}+2$ and so 
   \begin{align*}
    [\mathbb{Q}(a):\mathbb{Q}]=\left\{
                                \begin{array}{ll}
                                 \frac{1}{2}\varphi(\frac{N}{2}) & (N:\text{even})\\ [+5pt]
                                 \frac{1}{2}\varphi(N) & (N:\text{odd})
                                \end{array}.
                               \right.
   \end{align*} 
   Now $K_0\supset\mathbb{Q}(a)\supset\mathbb{Q}$ and so $[\mathbb{Q}(a):\mathbb{Q}]=1$ or $2$, and this implies $N=5,8,10,12,16,20,24$.
   By comparing with Proposition \ref{constant term of minimal polynomial}, since $P(x)$ has constant term $\pm1$, $N=5,10,24$ are only possible.
   Thus, the minimum value of the maximal absolute value of the roots of $P(x)$ is $2\mathrm{cos}\left(\frac{\pi}{5}\right)$ except $1$.
   Thus, $4\mathrm{cos}^2\left(\frac{\pi}{5}\right)=\left(\frac{1+\sqrt{5}}{2}\right)^2$ is a lower bound of the first dynamical degrees except $1$ for this type.\par
   Therefore, the minimum value of the first dynamical degrees is $\frac{3+\sqrt{5}}{2}=4\mathrm{cos}^2(\frac{\pi}{5})$ and realized on Type 1.
 \end{proof}
 Next, we consider the case that $p$ is a Sophie Germain prime except for $p=2$.
 \begin{proposition}\label{case for Sophie Germain prime}
  For a prime number $p\neq2$ such that $2p+1$ is also a prime,
  \begin{align*}
   m(p)=4\mathrm{cos}^2\left(\frac{\pi}{2p+1}\right)
  \end{align*}
 \end{proposition}
 \begin{proof}
  By the restriction in Table \ref{table1}, it suffices to consider only for Type 1 and 4.
  Define $B$, $K$, $K_0$, $d$, $e$ and $e_0$ as in Section \ref{Endomorphism algebras of simple abelian varieties} for each type of $X$.
  \begin{flushleft}{\bf{Type 1}}\end{flushleft}\par
   Let $X$ be a $p$-dimensional simple abelian variety of Type 1 and $f\in\mathrm{End}(X)$ be an automorphism of $X$ which corresponds to $\alpha\in U_K$ as in Section \ref{The minimum value of the first dynamical degrees of automorphisms}.
   Now $[K:\mathbb{Q}]=1$ or $p$ by the restriction in Table \ref{table1}.\par
   Assuming that the first dynamical degree is less than $4$, the conjugates of $\alpha$ are all in the interval $(-2,2)$.
   If the degree of $\alpha$ is $1$, then $\alpha=\pm1$ and so the first dynamical degree is $1$ and so we assume the degree of $\alpha$ is $1$.
   Denote the minimal polynomial of $\alpha$ over $\mathbb{Q}$ by $P(x)$, whose degree is $p$.
   By Lemma \ref{cyclotomic}, $x^pP(x+\frac{1}{x})$ is an irreducible cyclotomic polynomial and its degree is $2p$.
   Cyclotomic polynomials $\Phi_N(x)$ have degree $2p$ only for $N=2p+1,4p+2$ (but for $p=3$, $N=7,9,14,18$).
   By comparing with Proposition \ref{constant term of minimal polynomial}, since $P(x)$ has constant term $\pm1$, $N=2p+1,4p+2$ are possible (but for $p=3$, $N=7,9,14,18$ are possible).
   Thus, the minimum value of the maximal absolute value of the roots of $P(x)$ is $2\mathrm{cos}\left(\frac{\pi}{2p+1}\right)$ except $1$.
   The realizability is concerned in Section \ref{The minimum value of the first dynamical degrees of automorphisms}.
   Thus, the minimum value of the first dynamical degrees is $4\mathrm{cos}^2\left(\frac{\pi}{2p+1}\right)$ except $1$ for this type.
  \begin{flushleft}{\bf{Type 4}}\end{flushleft}\par
   Let $X$ be a $p$-dimensional simple abelian variety of Type 4 and $f\in\mathrm{End}(X)$ be an automorphism of $X$ which corresponds to $\alpha\in U_K$.
   Now $d=1$ and $[K:\mathbb{Q}]=2$ or $2p$ by the restriction in Table \ref{table1} and then $[K_0:\mathbb{Q}]=1$ or $p$, respectively.
   Denote the minimal polynomial of $\alpha$ over $\mathbb{Q}$ by $P(x)$.\par
   If $\alpha\in K_0$, the minimum value of the square of the maximal absolute value of the roots of $P(x)$ is $4\mathrm{cos}^2\left(\frac{\pi}{2p+1}\right)$ except $1$, by the deduction in Type 1.\par
   If $\alpha\in K\setminus K_0$, by denoting $a=\abs{\alpha}^2\in U_{K_0}$, as in Section \ref{The minimum value of the first dynamical degrees of automorphisms}, $\mathbb{Q}(\sqrt{a})$ is a totally real number field.
   If $a=1$, then the first dynamical degree would be $1$ and so we assume $a\neq1$.
   Also, by assuming that the maximal absolute value of the conjugates of $\sqrt{a}$ is less than $2$, it can be written as $\sqrt{a}=\zeta_N^m+\frac{1}{\zeta_N^m}$ ($m,N\in\mathbb{Z}_{>0}$) with $m<N$, $\mathrm{gcd}(N,m)=1$ and we can assume $N\neq1,2,3,4,6$.
   Thus, $a=\left(\zeta_N^m+\frac{1}{\zeta_N^m}\right)^2=\zeta_N^{2m}+\frac{1}{\zeta_N^{2m}}+2$ and so 
   \begin{align*}
    [\mathbb{Q}(a):\mathbb{Q}]=\left\{
                                \begin{array}{ll}
                                 \frac{1}{2}\varphi(\frac{N}{2}) & (N:\text{even})\\ [+5pt]
                                 \frac{1}{2}\varphi(N) & (N:\text{odd})
                                \end{array}.
                               \right.
   \end{align*} 
   Now $K_0\supset\mathbb{Q}(a)\supset\mathbb{Q}$ and so $[\mathbb{Q}(a):\mathbb{Q}]=1$ or $p$, and this implies $N=8,12,2p+1,4p+2,8p+4$ (but for $p=3$, $N=7,8,9,12,14,18,28,36$).
   By comparing with Proposition \ref{constant term of minimal polynomial}, since $P(x)$ has constant term $\pm1$, $N=2p+1,4p+2$ are only possible (but for $p=3$, $N=7,9,14,18$ are possible).
   Thus, the minimum value of the maximal absolute value of the roots of $P(x)$ is $2\mathrm{cos}\left(\frac{\pi}{2p+1}\right)$ except $1$.
   Thus, $4\mathrm{cos}^2\left(\frac{\pi}{2p+1}\right)$ is a lower bound of the first dynamical degrees except $1$ for this type.\par
  Therefore, the minimum value of the first dynamical degrees is $4\mathrm{cos}^2(\frac{\pi}{2p+1})$ and realized on Type 1.
 \end{proof}
 Next, we consider the case that $p$ is not a Sophie Germain prime.
 First, we prove the next lemma.
 \begin{lemma}\label{unrealizable first dynamical degrees}
  The first dynamical degree of an automorphism of a simple abelian variety is not equal to $n^2$ for an integer $n\in\mathbb{Z}_{\geq2}$.
  Especially, The first dynamical degree of an automorphism is not equal to $4$.
 \end{lemma}
 \begin{proof}
  Suppose that there is an automorphism $f:X\rightarrow X$ of a simple abelian variety $X$ which has the first dynamical degree $n^2$.
  Let $B$ be the endomorphism algebra of $X$.
  Let $\alpha\in B$ be the element corresponding to $f:X\rightarrow X$ and denote the minimal polynomial of $\alpha$ over $\mathbb{Q}$ by $P(x)\in\mathbb{Z}[x]$.
  Since $f$ is an automorphism, the constant term of $P(x)$ is $\pm1$.\par
  By the assumption that the first dynamical degree of $f$ is $n^2$, the maximal absolute value of the roots of $P(x)$ is $n$, and denote this root by $z$.
  Now $\abs{z}=n$, $\abs{\frac{1}{z}}=\frac{1}{n}$ and $\frac{1}{z}$ is an algebraic integer, and so $\frac{1}{z}\cdot\frac{1}{\bar{z}}=\frac{1}{n^2}$ is also an algebraic integer, a contradiction.
 \end{proof}
 \begin{proposition}\label{case otherwise}
  For a prime number $p$ such that $2p+1$ is not a prime,
  \begin{align*}
   m(p)=\mathrm{min}\left\{
                     \begin{array}{ll}
                      \text{the maximal absolute value of the conjugates of an}\\
                      \text{algebraic unit, whose conjugates are all real and positive, of degree $p$}
                     \end{array}
                    \right\}
  \end{align*}
  and also $m(p)$ is larger than $4+4^{-2p-3}$.
 \end{proposition}
 \begin{remark}
  The minimum value of the right side exists since the coefficients of the minimal polynomial is restricted by the maximal absolute value of its roots by using the triangle inequality.
  As a result, this equation holds for all prime numbers.
 \end{remark}
 \begin{proof}[Proof of Proposition \ref{case otherwise}]
  By the restriction in Table \ref{table1}, it suffices to consider only for Type 1 and 4.
  Define $B$, $K$, $K_0$, $d$, $e$ and $e_0$ as in Section \ref{Endomorphism algebras of simple abelian varieties} for each type of $X$.
  \begin{flushleft}{\bf{Type 1}}\end{flushleft}\par
   Let $X$ be a $p$-dimensional simple abelian variety of Type 1 and $f\in\mathrm{End}(X)$ be an automorphism of $X$ which corresponds to $\alpha\in U_K$ as in Section \ref{The minimum value of the first dynamical degrees of automorphisms}.
   Now $[K:\mathbb{Q}]=1$ or $p$ by the restriction in Table \ref{table1}.\par
   If $\alpha$ has degree $1$, then $\alpha=\pm1$ and the first dynamical degree is $1$ and so we assume $\alpha$ has degree $p$.
   Thus, by the deduction in Section \ref{The minimum value of the first dynamical degrees of automorphisms}, the set of the first dynamical degrees except $1$ for this type is
   \begin{align*}
    \mathcal{A}:=\left\{
     \begin{array}{ll}
      \text{the square of the maximal absolute value of the conjugates}\\
      \text{of an algebraic unit, whose conjugates are all real, of degree $p$}
     \end{array}
    \right\}.
   \end{align*}
  \begin{flushleft}{\bf{Type 4}}\end{flushleft}\par
   Let $X$ be a $p$-dimensional simple abelian variety of Type 4 and $f\in\mathrm{End}(X)$ be an automorphism of $X$ which corresponds to $\alpha\in U_K$.
   Now $d=1$ and $[K:\mathbb{Q}]=2$ or $2p$ by the restriction in Table \ref{table1} and then $[K_0:\mathbb{Q}]=1$ or $p$, respectively.
   Denote the minimal polynomial of $\alpha$ over $\mathbb{Q}$ by $P(x)$.\par
   If $\alpha\in K_0$, the set of the first dynamical degrees except $1$ is in
   \begin{align*}
   \mathcal{A}:=\left\{
    \begin{array}{ll}
     \text{the square of the maximal absolute value of the conjugates}\\
     \text{of an algebraic unit, whose conjugates are all real, of degree $p$}
    \end{array}
   \right\},
   \end{align*}
   by the deduction in Type 1.\par
   If $\alpha\in K\setminus K_0$, by denoting $a=\abs{\alpha}^2\in U_{K_0}$, as in Section \ref{The minimum value of the first dynamical degrees of automorphisms}, $\mathbb{Q}(\sqrt{a})$ is a totally real number field.
   If $a$ has degree $1$, then $a=1$ and so the first dynamical degree would be $1$ and so we assume $a$ has degree $p$.
   The first dynamical degree of $f$ is the square of the maximal absolute value of the conjugates of $\alpha$ and it is equal to the maximal absolute value of the conjugates of $a$.
   Now $a$ is a totally positive algebraic unit of degree $p$, the first dynamical degrees except $1$ are contained in
   \begin{align*}
    \mathcal{A}':=\left\{
     \begin{array}{ll}
      \text{the maximal absolute value of the conjugates of an}\\
      \text{algebraic unit, whose conjugates are all real and positive, of degree $p$}
     \end{array}
    \right\}.
   \end{align*}\par
   Let $a$ be an element of $\mathcal{A}'$.
   If $a$ is an element of $\mathcal{A}$, $a$ is realizable as the first dynamical degree of an automorphism of a simple abelian variety of Type 1.
   If $a$ is not an element of $\mathcal{A}$, then $\sqrt{a}$ is of degree $2p$.
   $a$ is a totally positive algebraic unit of degree $p$ and let $K=\mathbb{Q}(\sqrt{-a})$, $K_0=\mathbb{Q}(a)$, $d=1$, $e=2p$, $e_0=p$ and $m=1$.
   Now $K$ is a CM-field by Lemma \ref{unit element}.
   $K$ admits the complex conjugate as a positive anti-involution of the second kind, so by Proposition \ref{construction for Type 4}, there exist a $p$-dimensional abelian variety $X$ and an automorphism $f:X\rightarrow X$ which corresponds to $\sqrt{-a}$.
   Then, $a$ is the first dynamical degree of $f$ and by Lemma \ref{prove simplicity}, $X$ must be a simple abelian variety.
   Therefore, every element of $\mathcal{A}'$ is realizable as the first dynamical degree of an automorphism of some $p$-dimensional simple abelian variety.\par
   Let $r$ be an algebraic unit, whose conjugates are all real, of degree $p$.
   We can assume that the maximal absolute value of the conjugates of $r$ is $\abs{r}$.
   Then, $r^2$ is an algebraic unit, whose conjugates are all real and positive, and the maximal absolute value of the conjugate of $r^2$ is $r^2=\abs{r}^2$.
   Since $\mathbb{Q}(r)\supset\mathbb{Q}(r^2)\supset\mathbb{Q}$ and $[\mathbb{Q}(r):\mathbb{Q}(r^2)]=1$ or $2$, the degree of $r^2$ is $p$.
   Thus, $\mathcal{A}\subset\mathcal{A}'$ and so
   \begin{align*}
    m(p)=\mathrm{min}(\mathcal{A}').
   \end{align*}\par
   It remains to show $m(p)>4+4^{-2p-3}$.
   If $m(p)<4$, then there exists a totally real algebraic unit $r$ of degree $p$, whose conjugates are all inside the interval $(0,4)$.
   Thus, $r-2$ is also a totally real algebraic integer and so by Lemma \ref{cyclotomic}, the minimal polynomial of $r-2$ is written as $\Psi_N(x)$ for some $N\geq3$.
   The degree of this polynomial is $p=\frac{1}{2}\varphi(N)$, but this would not be occured by the assumption that $2p+1$ is not a prime.
   Therefore, by composing with Lemma \ref{unrealizable first dynamical degrees}, $m(p)>4$.
   Let $r'$ be an algebraic unit of degree $p$ whose conjugates are all real and positive, and at most one conjugate of $r'$ is larger than $4$.
   Then $r'-2$ is a totally real algebraic integer and its maximal absolute value of conjugates is larger than $2$.
   Thus, by \cite[Theorem 2]{SZ65}, the maximal absolute value of the conjugates of $r'-2$ is larger than $2+4^{-2p-3}$, and it concludes that at least one conjugate of $r'$ is larger than $4+4^{-2p-3}$.
   Therefore, $m(p)>4+4^{-2p-3}$.
  \end{proof}
  \begin{lemma}\label{prove simplicity}
   Let $p$ be an odd prime number, $a$ be a totally positive algebraic unit of degree $p$ and $X$ be a $p$-dimensional abelian variety.
   Assume $\sqrt{a}$ is of degree $2p$ and denote $\alpha=\sqrt{-a}$.
   If $\alpha\in\mathrm{End}_\mathbb{Q}(X)$, then $X$ is a simple abelian variety.
  \end{lemma}
  \begin{proof}
   Let $f(x)\in\mathbb{Z}[x]$ be the minimal polynomial of $a$ and so its constant term is $\pm1$.
   Then the minimal polynomial of $\alpha$ over $\mathbb{Q}$ is $g(x)=-f(-x^2)\in\mathbb{Z}[x]$ of degree $2p$.\par
   Assume $X$ is not a simple abelian variety.
   By Poincare's complete reducibility theorem (\cite[Theorem 5.3.7]{BL04}), there is an isogeny
   \begin{align*}
    X\rightarrow X_1^{n_1}\times\cdots\times X_r^{n_r},
   \end{align*}
   where $X_i$ are simple abelian varieties and not isogenious each other.
   Denote $\mathrm{dim}(X_1)=g_1$ and then $p=\sum_{i=1}^r g_in_i$.
   Moreover, by its corollary (cf.\ \cite[Corollary 5.3.8]{BL04}), 
   \begin{align*}
    \mathrm{End}_\mathbb{Q}(X)\simeq\mathrm{M}_{n_1}(\mathrm{End}_\mathbb{Q}(X_1))\oplus\cdots\oplus \mathrm{M}_{n_r}(\mathrm{End}_\mathbb{Q}(X_r)),
   \end{align*}
   where $\mathrm{End}_\mathbb{Q}(X_i)$ are division rings.\par
   $\alpha\in\mathrm{End}_\mathbb{Q}(X)$ is decomposed as $\alpha=\alpha_1+\cdots+\alpha_r$ with $\alpha_i\in \mathrm{M}_{n_i}(\mathrm{End}_\mathbb{Q}(X_i))$.
   Each $\alpha_i$ has the minimal polynomial $p_i(x)$ over $\mathbb{Q}$ and its degree is at most $2g_in_i$.
   The minimal polynomial $g(x)\in\mathbb{Z}[x]$ of $\alpha$ is a factor of $p_1(x)\cdots p_r(x)$ and by the inequality
   \begin{align*}
    2p=\mathrm{deg}(g(x))\leq\sum_{i=1}^r\mathrm{deg}(p_i(x))=\sum_{i=1}^r 2g_in_i=2p,
   \end{align*}
   this implies $r=1$.
   Since $X$ is not simple, the case $g_1=1$, $n_1=p$ is only possible.
   For this case, the minimal polynomial of $\alpha_1$ is $g(x)$.\par
   Now
   \begin{align*}
    \mathrm{End}_\mathbb{Q}(X)=\mathrm{M}_{p}(\mathrm{End}_\mathbb{Q}(X_1))
   \end{align*}
   and since $\alpha_1\in \mathrm{M}_{p}(\mathrm{End}_\mathbb{Q}(X_1))$ has the minimal polynomial of degree $2p$ over $\mathbb{Q}$, $\mathrm{End}_\mathbb{Q}(X_1)\neq\mathbb{Q}$.
   By Table \ref{table1}, $X_1$ is of Type 4 and $\mathrm{End}_\mathbb{Q}(X_1)$ is a totally complex quadratic algebra over $\mathbb{Q}$.
   Denote $\mathrm{End}_\mathbb{Q}(X)=\mathbb{Q}(\sqrt{-D})$ for some $D\in\mathbb{Z}_{>0}$ a square-free integer.
   The polynomial $g(x)$ is written as
   \footnotesize
   \begin{align*}
    g(x)&=(x^p+(a_1+b_1\sqrt{-D})x^{p-1}+\cdots+(a_p+b_p\sqrt{-D}))(x^p+(a_1-b_1\sqrt{-D})x^{p-1}+\cdots+(a_p-b_p\sqrt{-D}))\tag*{($\ast$)}\\
        &=g_1(x)g_2(x)
   \end{align*}
   \normalsize
   with $a_i+b_i\sqrt{-D}\in\mathcal{O}_{\mathbb{Q}(\sqrt{-D})}$.
   If $D\equiv1,2\ (\mathrm{mod}\ 4)$, then $a_i,b_i\in\mathbb{Z}$ and if $D\equiv3\ (\mathrm{mod}\ 4)$, then $a_i,b_i\in\frac{1}{2}\mathbb{Z}$.
   Now the roots of $g(x)$ are in $\mathbb{R}i$ and this implies $g_1(x)=(-1)^pg_2(-x)$.
   Thus, $a_p+b_p\sqrt{-D}=-a_p+b_p\sqrt{-D}$ and so $a_p=0$.
   By considering the constant term of the equation ($\ast$),
   \begin{align*}
    \pm1=a_p^2+b_p^2D=b_p^2D
   \end{align*}
   and so $D=1$.\par
   Thus, it can be said that $\mathbb{Q}(\alpha)\supset\mathbb{Q}(i)$ and therefore $\mathbb{Q}(\alpha)\supset\mathbb{Q}(\sqrt{a})\supset\mathbb{Q}(a)$.
   By the assumption, $[\mathbb{Q}(\sqrt{a}):\mathbb{Q}(a)]=2$ and this implies $\mathbb{Q}(\alpha)=\mathbb{Q}(\sqrt{a})$.
   This contradicts to $\mathbb{Q}(\alpha)\not\subset\mathbb{R}$ and $\mathbb{Q}(\sqrt{a})\subset\mathbb{R}$, and so $X$ is a simple abelian variety.
  \end{proof}
 Next, for proving Theorem \ref{Result1}, we would like to prove $m(p)<(2.08)^2$ for a prime number $p$ which is not a Sophie Germain prime.
 This result follows from the next proposition.
 \begin{proposition}\label{Main Proposition}
  For any prime number $p>3$, there exists a monic irreducible polynomial $T(x)\in\mathbb{Z}[x]$ of degree $p$, whose constant term is $\pm1$ and whose roots are all real and inside the interval $(-2.08,2.08)$.
 \end{proposition}
 \begin{remark}
  The proposition holds also for $p=2,3$, but for simplifying the proof, we assume $p>3$.
 \end{remark}
 For proving this proposition, we use the next lemmas.
 \begin{lemma}\label{make Salem polynomials}
  Define the polynomial
  \begin{align*}
  P_n(x)= \frac{x^{n-2}(x^3-x-1)+(x^3+x^2-1)}{x-1},\ Q_m(x)=x^{m-3}(x^3-x-1)-(x^3+x^2-1)
  \end{align*}
  for $n\geq3$ and $m\geq4$,
  The next properties hold.
  \begin{enumerate}
   \item For any $n\geq3$, there is at most one root of mudulus larger than $1$, counted with multiplicity, among the roots of $P_n(x)$, and it is a real root.
   \item For any $m\geq4$, there is at most one root of mudulus larger than $1$, counted with multiplicity, among the roots of $Q_m(x)$, and it is a real root.
  \end{enumerate}
 \end{lemma}
 \begin{proof}
  $x^3-x-1$ is the minimal polynomial of (the smallest) Pisot number and by the proof of \cite[Theorem 6.4.1]{BDGPS92}, the lemma holds.
 \end{proof}
 \begin{lemma}\label{cyclotomic factor}
  The cyclotomic polynomials which divide $Q_m(x)\in\mathbb{Z}[x]$ for some $m\geq4$ are only $\Phi_2(x), \Phi_8(x), \Phi_{12}(x), \Phi_{18}(x), \Phi_{30}(x)$.
 \end{lemma}
 \begin{proof}
  For $N\geq1$, if $\Phi_N(x)$ divides $Q_m(x)$ for some $m\geq4$, then $Q_m(\zeta_N)=0$ for some $m\geq4$.
  Thus,
  \begin{align*}
   \zeta_N^{3-m}=\frac{\zeta_N^3-\zeta_N-1}{\zeta_N^3+\zeta_N^2-1}=-\zeta_N^{-3}\frac{1+\zeta_N-\zeta_N^3}{1+\zeta_N^{-1}-\zeta_N^{-3}}\tag*{($\ast\ast$)}
  \end{align*}
  and for $N=1$ to $8$, left side calculated as
  \begin{align*}
   &N=1:\frac{\zeta_N^3-\zeta_N-1}{\zeta_N^3+\zeta_N^2-1}=-1,\quad                       &N=2&:\frac{\zeta_N^3-\zeta_N-1}{\zeta_N^3+\zeta_N^2-1}=1,\\
   &N=3:\frac{\zeta_N^3-\zeta_N-1}{\zeta_N^3+\zeta_N^2-1}=-\zeta_3^2,\quad               &N=4&:\frac{\zeta_N^3-\zeta_N-1}{\zeta_N^3+\zeta_N^2-1}=\frac{4+3i}{5},\\
   &N=5:\frac{\zeta_N^3-\zeta_N-1}{\zeta_N^3+\zeta_N^2-1}=-\zeta_5^3,\quad               &N=6&:\frac{\zeta_N^3-\zeta_N-1}{\zeta_N^3+\zeta_N^2-1}=\frac{11+5\sqrt{3}i}{14},\\
   &N=7:\frac{\zeta_N^3-\zeta_N-1}{\zeta_N^3+\zeta_N^2-1}=\frac{3+\sqrt{7}i}{4},\quad    &N=8&:\frac{\zeta_N^3-\zeta_N-1}{\zeta_N^3+\zeta_N^2-1}=\zeta_8.
  \end{align*}
  Thus, $N=2,8$ are only possible.
  For $N>8$, the equation ($\ast\ast$) implies
  \begin{align*}
   \mathrm{arg}(1+\zeta_N-\zeta_N^3)=-\frac{\pi}{2N}, -\frac{\pi}{N}, -\frac{3\pi}{2N}, \ldots
  \end{align*}
  and this is equivalent to
  \begin{align*}
   \angle{PAQ}=\frac{\pi}{2N}, \frac{\pi}{N}, \frac{3\pi}{2N}, \ldots
  \end{align*}
  in Figure \ref{figure1}, Figure \ref{figure2} or Figure \ref{figure3}.
  By some coordinate calculation, it is also equivalent to
  \begin{align*}
   \frac{2\mathrm{sin}\left(\frac{2\pi}{N}\right)\mathrm{cos}\left(\frac{4\pi}{N}\right)}{1+2\mathrm{sin}\left(\frac{2\pi}{N}\right)\mathrm{sin}\left(\frac{4\pi}{N}\right)}=\mathrm{tan}\left(\frac{\pi}{2N}\right), \mathrm{tan}\left(\frac{\pi}{N}\right), \mathrm{tan}\left(\frac{3\pi}{2N}\right), \ldots.
  \end{align*}
  \begin{enumerate}
   \item{$9\leq N<18$}\\
    See Figure \ref{figure1}.
    For this case, point A is outside of the unit circle, so by the inscribed\\ angle theorem, $\angle{PAQ}<\frac{2\pi}{N}$ and so $\angle{PAQ}=\frac{\pi}{2N},\frac{\pi}{N},\frac{3\pi}{2N}$.
    By observing the equations
    \begin{equation*}
     \frac{2\mathrm{sin}(\theta)\mathrm{cos}(2\theta)}{1+2\mathrm{sin}(\theta)\mathrm{sin}(2\theta)}=\mathrm{tan}\left(\frac{\theta}{4}\right), \mathrm{tan}\left(\frac{\theta}{2}\right), \mathrm{tan}\left(\frac{3\theta}{4}\right)
    \end{equation*}
    for $\frac{\pi}{9}<\theta\leq\frac{2\pi}{9}$, then $\theta=\frac{\pi}{6}$ and this correeponds to the case $N=12$.
   \item{$N=18$}\\
    See Figure \ref{figure2}.
    For this case, point A is on the unit circle, so by the inscribed angle theorem, $\angle{PAQ}=\frac{2\pi}{N}$, and this is the case.
   \item{$18<N$}\\
    See Figure \ref{figure3}.
    For this case, point A is inside of the unit circle, so by the inscribed angle theorem, $\angle{PAQ}>\frac{2\pi}{N}$.
    Also, by $\angle{PBA}=\frac{2\pi}{N}$ and $AB<1<AP$, $\angle{PAQ}=\angle{PBA}+\angle{APB}<2\angle{PBA}=\frac{4\pi}{N}$ and so $\angle{PAQ}=\frac{5\pi}{2N},\frac{3\pi}{N},\frac{7\pi}{2N}$.
    By observing the equations
    \begin{equation*}
     \frac{2\mathrm{sin}(\theta)\mathrm{cos}(2\theta)}{1+2\mathrm{sin}(\theta)\mathrm{sin}(2\theta)}=\mathrm{tan}\left(\frac{5\theta}{4}\right), \mathrm{tan}\left(\frac{3\theta}{2}\right), \mathrm{tan}\left(\frac{7\theta}{4}\right)
    \end{equation*}
    for $0<\theta<\frac{\pi}{9}$, then $\theta=\frac{\pi}{15}$ and this correeponds to the case $N=30$.
  \end{enumerate}
  \begin{figure}
  \begin{tabular}{cc}
  \begin{minipage}{0.45\hsize}
   \includegraphics[scale=0.5]{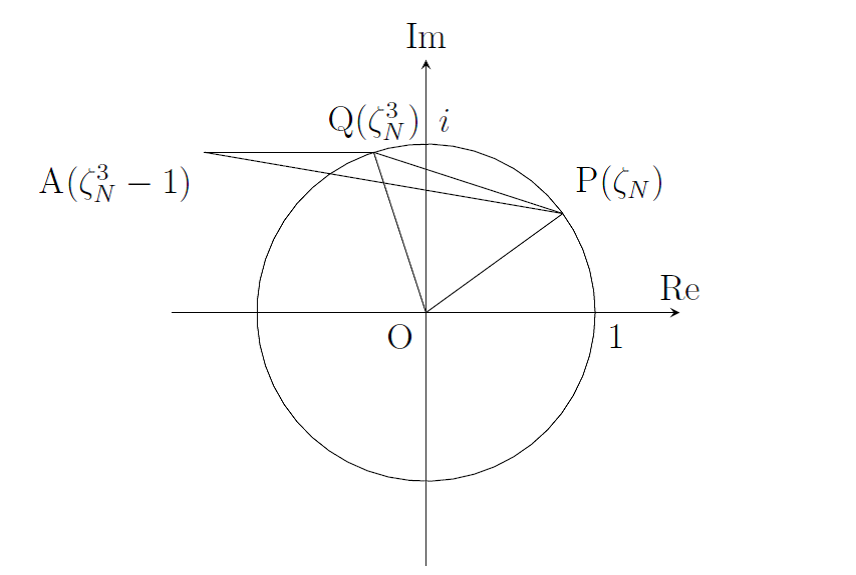}
   \caption{$9<N<18$}\label{figure1}
  \end{minipage} &
  \begin{minipage}{0.45\hsize}
   \includegraphics[scale=0.5]{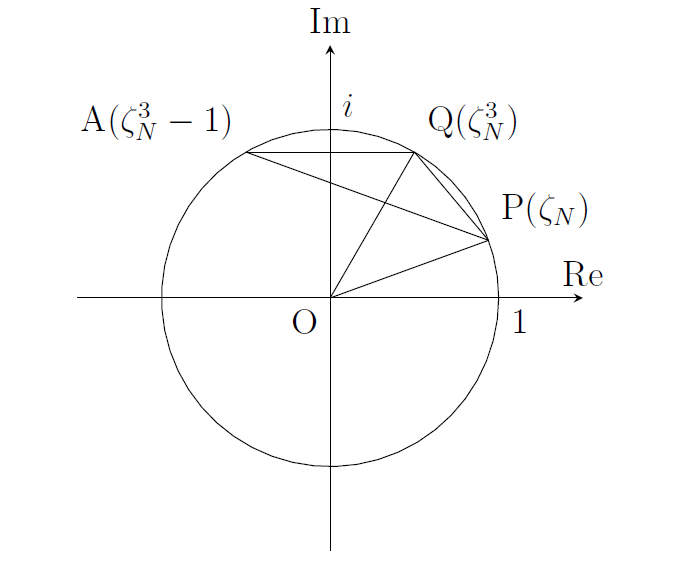}
   \caption{$N=18$}
   \label{figure2}
  \end{minipage}
  \end{tabular}
  \end{figure}
  \begin{figure}
   \includegraphics[scale=0.5]{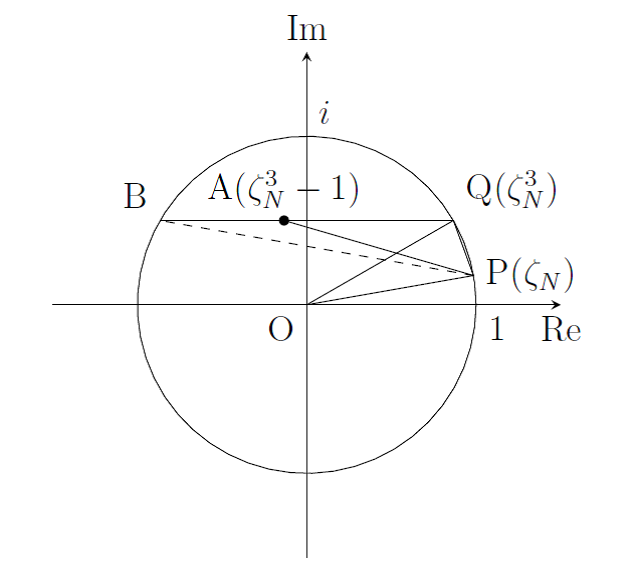}
   \caption{$18<N$}
   \label{figure3}
  \end{figure}
 \end{proof}
 \begin{remark}\label{Remark}
  For $0<t<2$, the equation
  \begin{align*}
   \frac{2\mathrm{sin}(\theta)\mathrm{cos}(2\theta)}{1+2\mathrm{sin}(\theta)\mathrm{sin}(2\theta)}=\mathrm{tan}(t\theta)
  \end{align*}
  has at most one root in the interval $\left(0, \frac{\pi}{4}\right)$ (see Appendix \ref{appendix b}) and so the way for searching the suitable $N\in\mathbb{Z}_{>0}$ is returned to substituting integers into the equations in the proof.
 \end{remark}
 \begin{remark}\label{equivalent condition2}
  By using the equation ($\ast\ast$) in the proof of the lemma, it can be said that
  \begin{align*}
   &Q_m(x)\text{ is divided by }\Phi_2(x)=x+1\quad\Longleftrightarrow\quad m\equiv1\ (\mathrm{mod}\ 2)\\
   &Q_m(x)\text{ is divided by }\Phi_8(x)=x^4+1\quad\Longleftrightarrow\quad m\equiv2\ (\mathrm{mod}\ 8)\\
   &Q_m(x)\text{ is divided by }\Phi_{12}(x)=x^4-x^2+1\quad\Longleftrightarrow\quad m\equiv1\ (\mathrm{mod}\ 12)\\
   &Q_m(x)\text{ is divided by }\Phi_{18}(x)=x^6-x^3+1\quad\Longleftrightarrow\quad m\equiv17\ (\mathrm{mod}\ 18)\\
   &Q_m(x)\text{ is divided by }\Phi_{30}(x)=x^8+x^7-x^5-x^4-x^3+x+1\quad\Longleftrightarrow\quad m\equiv24\ (\mathrm{mod}\ 30)
  \end{align*}
 \end{remark}
 \begin{lemma}\label{cyclotomic factor for P_n}
  The cyclotomic polynomials which divide $P_n(x)\in\mathbb{Z}[x]$ for some $n\geq10$ are only $\Phi_2(x), \Phi_3(x), \Phi_5(x), \Phi_8(x), \Phi_{12}(x), \Phi_{18}(x), \Phi_{30}(x)$.
 \end{lemma}
 \begin{proof}
  The proof is proceeding as the proof of Lemma \ref{cyclotomic factor}.
  For $1\leq N\leq8$, $N=2,3,5,8$ are only possible.
  Also for $N>8$, the only different point is the candidates for $\angle{PAQ}$ are changed to $\angle{PAQ}=\frac{\pi}{N}, \frac{2\pi}{N}, \frac{3\pi}{N}$ and $N=12,18,30$ are possible.
 \end{proof}
 \begin{remark}\label{equivalent condition1}
  As in Remark \ref{equivalent condition2}, it can be said that
  \begin{align*}
   &P_n(x)\text{ is divided by }\Phi_2(x)=x+1\quad\Longleftrightarrow\quad n\equiv1\ (\mathrm{mod}\ 2)\\
   &P_n(x)\text{ is divided by }\Phi_3(x)=x^2+x+1\quad\Longleftrightarrow\quad n\equiv0\ (\mathrm{mod}\ 3)\\
   &P_n(x)\text{ is divided by }\Phi_5(x)=x^4+x^3+x^2+x+1\quad\Longleftrightarrow\quad n\equiv4\ (\mathrm{mod}\ 5)\\
   &P_n(x)\text{ is divided by }\Phi_8(x)=x^4+1\quad\Longleftrightarrow\quad n\equiv5\ (\mathrm{mod}\ 8)\\
   &P_n(x)\text{ is divided by }\Phi_{12}(x)=x^4-x^2+1\quad\Longleftrightarrow\quad n\equiv6\ (\mathrm{mod}\ 12)\\
   &P_n(x)\text{ is divided by }\Phi_{18}(x)=x^6-x^3+1\quad\Longleftrightarrow\quad n\equiv7\ (\mathrm{mod}\ 18)\\
   &P_n(x)\text{ is divided by }\Phi_{30}(x)=x^8+x^7-x^5-x^4-x^3+x+1\quad\Longleftrightarrow\quad n\equiv8\ (\mathrm{mod}\ 30)
  \end{align*}
 \end{remark}
 \begin{remark}
  Similar results for Lemma \ref{cyclotomic factor for P_n} and Remark \ref{equivalent condition1} are stated in \cite[Theorem 1.1]{GHM09}.
 \end{remark}
 \begin{lemma}\label{simpliity of roots}
  For each $n\geq10$, the roots of $P_n(x)$ are all simple.
  Also, for each $m\geq4$, the roots of $Q_m(x)$ are all simple.
 \end{lemma}
 \begin{proof}
  Every $P_n(x)$ and $Q_m(x)$ has at most one root of modulus larger than $1$, so this root and its conjugates are all simple.
  The simplicity of the other roots can be checked by using Remark \ref{equivalent condition1} and Remark \ref{equivalent condition2} for each cyclotomic polynomial.
 \end{proof}
 \begin{remark}
  Assume $n\geq10$ and $m\geq4$.
  Dividing $P_n(x)$ (resp. $Q_m(x)$) by its cyclotomic factor as many as possible, the remaining polynomial is constant or the polynomial which has only one root of modulus larger than $1$ by Lemma \ref{make Salem polynomials}.\par
  By Lemma \ref{simpliity of roots}, the dividing operation is ended at most $7$ times (resp. $5$ times) and the degree decreases at most $29$ (resp. $23$).
  Thus, by Table \ref{appendix 1} and Table \ref{appendix 2} in Appendix \ref{appendix a}, the remaining polynomials have degree at least $4$ and so these are Salem polynomials.
  Write $P_n(x)=C_n(x)S_n(x)$ (resp. $Q_m(x)=C'_m(x)S'_m(x)$) where $C_n(x)$ (resp. $C'_m(x)$) is a product of cyclotomic polynomials and $S_n(x)$ (resp. $S'_m(x)$) is a Salem polynomial. 
 \end{remark}
 \begin{lemma}\label{convergence lemma}
  Assume $n\geq10$ and $m\geq4$.
  Let $s_n$ be a Salem number whose minimal polynomial is $S_n(x)$ and let $s'_m$ be a Salem number whose minimal polynomial is $S'_m(x)$.
  Then, the sequence $\{s_n\}_{n\geq10}$ strictly increasing and the sequence $\{s'_m\}_{m\geq4}$ strictly decreasing.
  Also, the limit of these sequence is the same and it is $1.3247\cdots$, the smallest Pisot number.
 \end{lemma}
 \begin{proof}
  $1.3247\cdots$ is the real root of $x^3-x-1\in\mathbb{Z}[x]$ and this result is proved as in the proof of \cite[Theorem 6.4.1]{BDGPS92}.
 \end{proof}
 For proceeding the proof, we need the next notation.
 \begin{definition}
  For a Salem polynomial $S(x)$ of degree $2d$, the \textit{trace polynomial} of $S(x)$ is a polynomial $T(x)$ of degree $d$ such that $S(x)=x^dT(x+\frac{1}{x})$. 
 \end{definition}
 \begin{remark}
  By the definition of Salem numbers, the roots of $T(x)$ are all real and it has only one root outside of the interval $(-2,2)$.
  Thus, for proving Proposition \ref{Main Proposition}, we need a Salem polynomial $S(x)$ of degree $2p$ and the associated trace polynomial with constant term $\pm1$.
 \end{remark}
 \begin{lemma}\label{constant term}
  For a Salem polynomial $S(x)$ and its associated trace polynomial $T(x)$, the constant term of $T(x)$ is $\pm1$ if and only if $\abs{S(i)}=1$.
 \end{lemma}
 \begin{proof}
  By the relation $S(x)=x^dT(x+\frac{1}{x})$, it is immediate that
  \begin{align*}
   \text{the constant term of $T(x)$ is $\pm1$}\ \Longleftrightarrow\ T(0)=\pm1\Longleftrightarrow\ \abs{S(i)}=1
  \end{align*}
 \end{proof}
 For $n\geq10$ and $m\geq4$, let $T_n(x)$ be the trace polynomial of $S_n(x)$ and let $T'_m(x)$ be the trace polynomial of $S'_m(x)$.
 By using Remark \ref{equivalent condition1} and Remark \ref{equivalent condition2}, we get the condition for that when the constant term of $T_n(x)$ or $T'_m(x)$ is $\pm1$.\\
 \begin{lemma}\label{condition for constant term}
  Assume $n\geq10$ and $m\geq4$.
  Then,
  \begin{align*}
   \text{the constant term of $T_n(x)$ is $\pm1$}\quad\Longleftrightarrow\quad n\equiv0,3,4,5,6,&7,8,11,12,13,15,\\
                                                                                                &16,18,19,20,21,23\ (\mathrm{mod}\ 24)
  \end{align*}
  \begin{align*}
   \text{the constant term of $T'_m(x)$ is $\pm1$}\quad\Longleftrightarrow\quad m\equiv1,2,3,7,10,11,13,15,18,19,23\ (\mathrm{mod}\ 24)
  \end{align*}
 \end{lemma}
 \begin{proof}
  By calculating,
  \begin{align*}
   \abs{\Phi_2(i)}=\sqrt{2}, \abs{\Phi_3(i)}=1, \abs{\Phi_5(i)}=1, \abs{\Phi_8(i)}=2, \abs{\Phi_{12}(i)}=3, \abs{\Phi_{18}(i)}=1, \abs{\Phi_{30}(i)}=1, 
  \end{align*}
  \begin{align*}
   \abs{P_n(i)}=\left\{
                 \begin{array}{ll}
                  1 & (n\equiv0\ (\mathrm{mod}\ 4))\\
                  2\sqrt{2} & (n\equiv1\ (\mathrm{mod}\ 4))\\
                  3 & (n\equiv2\ (\mathrm{mod}\ 4))\\
                  \sqrt{2} & (n\equiv3\ (\mathrm{mod}\ 4))
                 \end{array},
                \right.
   \abs{Q_m(i)}=\left\{
                 \begin{array}{ll}
                  4 & (n\equiv0\ (\mathrm{mod}\ 4))\\
                  3\sqrt{2} & (n\equiv1\ (\mathrm{mod}\ 4))\\
                  2 & (n\equiv2\ (\mathrm{mod}\ 4))\\
                  \sqrt{2} & (n\equiv3\ (\mathrm{mod}\ 4))
                 \end{array}.
                \right.
  \end{align*}
  By Lemma \ref{constant term}, it is enough to calculate $\abs{S(i)}$ and so by using Remark \ref{equivalent condition1} and Remark \ref{equivalent condition2}, cyclicity of order $24$ can be found and the proof is done. 
 \end{proof}
 In Appendix \ref{appendix a}, Table \ref{appendix 3} is enumerating the degree of $S_n(x)$ for $n$, which satisfies the condition in Lemma \ref{condition for constant term}.
 Table \ref{appendix 4} is enumerating the degree of $S'_m(x)$ for $m$, which satisfies the condition in Lemma \ref{condition for constant term}.
 Both have cyclicity of order $360$ and so only $10\leq n\leq373$ and $7\leq m\leq378$ are filled in.
 \begin{proof}[Proof of Proposition \ref{Main Proposition}]
  Now all the double of an odd integer, larger than $3$, would be appeared as the degree of $S_n(x)$ in Table \ref{appendix 3} or $S'_m(x)$ in Table \ref{appendix 4} by the cyclicity.
  Thus, for a prime number $p>3$, $2p$ is appeared as the degree of $S_n(x)$ or $S'_m(x)$ with the conditions in Lemma \ref{condition for constant term}.\par
  If $2p$ is appeared as the degree of $S_n(x)$ for some $n$, then the associated trace polynomial $T_n(x)$ has degree $p$ and its roots are all totally real algebraic intregers.\par
  By Lemma \ref{convergence lemma}, the maximal absolute value of the roots of $S_n(x)$ is less than $1.3247\cdots$ and so the maximal absolute value of the roots of $T_n(x)$ is less than $1.3247\cdots+\frac{1}{1.3247\cdots}=2.0795\cdots<2.08$.\par
  The first $2p$, which is not appeared as the degree of $S_n(x)$ is $26$ and it is appeared as the degree of $S'_{27}(x)$ and the associated trace polynomial $T'_{27}(x)$ has degree 13.
  By Lemma \ref{convergence lemma} again, the maximal absolute value of $S'_m(x)$ is strictly decreasing, and so it suffices to show that the maximal absolute value of $T'_{27}(x)$ is less than $2.08$.
  The Salem number for $S'_{27}(x)$ is $1.3255\cdots$ and so the maximal absolute value of $T'_{27}(x)$ is $1.3255\cdots+\frac{1}{1.3255\cdots}=2.0799\cdots<2.08$.
  Thus, the proof is concluded.
 \end{proof}
 \begin{proof}[Proof of Theorem \ref{Result1}]
  Theorem \ref{Result1} is concluded by composing Proposition \ref{case=2}, Proposition \ref{case for Sophie Germain prime}, Proposition \ref{case otherwise} and Proposition \ref{Main Proposition}.
 \end{proof}

\section{The minimum value of the first dynamical degrees for simple abelian varieties with lower dimensions}\label{The minimum value of the first dynamical degrees for simple abelian varieties with lower dimensions}
 In this section, we calculate the minimum value of the first dynamical degrees except $1$ with fixing the dimension of simple abelian varieties $g=2$ to $10$.
 The result is the following.
 \begin{table}[hbtp]
   \label{table}
   \begin{tabular}{|c|c|}
    \hline 
    dimension $g$ & minimum value \\
    \hline
    $1$ & $(\text{all }1)$\\
    $2$ & $4\mathrm{cos}^2(\pi/5)=2.6180\cdots$ \\
    $3$ & $4\mathrm{cos}^2(\pi/7)=3.2469\cdots$ \\
    $4$ & $2\mathrm{cos}(\pi/5)=1.6180\cdots$ \\
    $5$ & $4\mathrm{cos}^2(\pi/11)=3.6825\cdots$ \\
    $6$ & $2\mathrm{cos}(\pi/7)=1.8019\cdots$ \\
    $7$ & $4.0333\cdots$\footnotemark[1] \\
    $8$ & $2\mathrm{cos}(\pi/5)=1.6180\cdots$ \\
    $9$ & $1.1509\cdots^2=1.3247\cdots$\footnotemark[2] \\
    $10$& $1.1762\cdots^2=1.3836\cdots$\footnotemark[3] \\
    \hline
   \end{tabular}
  \end{table}
 \footnotetext[1]{$4.0333\cdots$ is the maximal absolute value of the roots of $x^7-14x^6+77x^5-211x^4+301x^3-210x^2+56x-1$}
 \footnotetext[2]{$1.1509\cdots$ is the absolute value of a complex root of $x^3-x^2+1$}
 \footnotetext[3]{$1.1762\cdots$ is the Lehmer number}
 \begin{remark}
  If $g=1$, then the first dynamical degree of an automorphism is the last dynamical degree and it is always $1$.\par
  The cases $g=2,3,5$ are already done and the result is compatible with Theorem \ref{Result1}.\par
  We use a computer algebra for $g=6$ to $10$.
  The calculation would be proceeding as dividing into the type of simple abelian varieties.
 \end{remark}

 \subsection{dimension $g=4$}
  Let $X$ be a $4$-dimensional simple abelian variety and define $B$, $K$, $K_0$, $d$, $e$ and $e_0$ as in Section \ref{Endomorphism algebras of simple abelian varieties}.
  \begin{flushleft}{\bf{Type 1}}\end{flushleft}\par
   Let $X$ be a $4$-dimensional simple abelian variety of Type 1 and $f\in\mathrm{End}(X)$ be an automorphism of $X$ which corresponds to $\alpha\in U_K$ as in Section \ref{The minimum value of the first dynamical degrees of automorphisms}.
   Now $[K:\mathbb{Q}]=1$ or $2$ or $4$ by the restriction in Table \ref{table1}.\par
   Assuming that the first dynamical degree is less than $4$, the conjugates of $\alpha$ are all in the interval $(-2,2)$.
   If the degree of $\alpha$ is $1$, then $\alpha=\pm1$ and so the first dynamical degree is $1$ and so we assume the degree of $\alpha$ is $2$ or $4$.
   Denote the minimal polynomial of $\alpha$ over $\mathbb{Q}$ by $P(x)$, whose degree is $r=2$ or $4$.
   By Lemma \ref{cyclotomic}, $x^rP(x+\frac{1}{x})$ is an irreducible cyclotomic polynomial and its degree is $2r=4$ or $8$.
   Cyclotomic polynomials $\Phi_N(x)$ have degree $4$ or $8$ only for $N=5,8,10,12,15,16,20,24,30$.
   By comparing with Proposition \ref{constant term of minimal polynomial}, since $P(x)$ has constant term $\pm1$, $N=5,10,15,24,30$ are only possible.
   Thus, the minimum value of the maximal absolute value of the roots of $P(x)$ is $2\mathrm{cos}\left(\frac{\pi}{5}\right)$ except $1$.
   The realizability is concerned in Section \ref{The minimum value of the first dynamical degrees of automorphisms}.
   Thus, the minimum value of the first dynamical degrees is $4\mathrm{cos}^2\left(\frac{\pi}{5}\right)=\left(\frac{1+\sqrt{5}}{2}\right)^2$ except $1$ for this type.
  \begin{flushleft}{\bf{Type 2, Type 3}}\end{flushleft}\par
   Let $X$ be a $4$-dimensional simple abelian variety of Type 2 or 3 and $f\in\mathrm{End}(X)$ be an automorphism of $X$ which corresponds to $\alpha\in B$.
   Now $[K:\mathbb{Q}]=1$ or $2$ by the restriction in Table \ref{table1}.\par
   Denote the minimal polynomial of $\alpha$ over $K$ by $p(x)$ and the minimal polynomial of $\alpha$ over $\mathbb{Q}$ by $P(x)$.
   If the degree of $p(x)$ is $1$, then $\alpha\in U_K$ and so it is reduced to Type 1 of $2$-dimensional case.
   Thus, $\frac{3+\sqrt{5}}{2}$ is a lower bound except $1$ for this case.\par
   If the degree of $p(x)$ is $2$, then it is written as $p(x)=x^2+ax+b\in\mathcal{O}_K[x](b\in U_K)$ and consider the following 3 cases as in Section \ref{The minimum value of the first dynamical degrees of automorphisms}.
   Let $\{\sigma_i\}$ be the set of all $\mathbb{Q}$-embeddings $K\hookrightarrow\mathbb{C}$.\\
   $\bullet$ there exists $1\leq i\leq e$ such that $\abs{\sigma_i(b)}>1$\par
    The maximal absolute value of all conjugates of $b\in U_K$ is at least $\frac{1+\sqrt{5}}{2}$ and so the maximal absolute value of the roots of $P(x)$ is at least $\sqrt{\frac{1+\sqrt{5}}{2}}$.
    This is achived by $p(x)=x^2+\frac{1+\sqrt{5}}{2}$ and $P(x)=x^4+x^2-1$.
    Thus, $\frac{1+\sqrt{5}}{2}$ is a lower bound except $1$ for this case.\\
   $\bullet$ $\sigma_i(b)=1$ for all $1\leq i\leq e$\par
    Since $\mathrm{max}\{\abs{\sigma_i(a)}\}$ must be larger than $2$ and close to $2$, $a=\sqrt{5}$ provides the minimum value of the maximal absolute value by some calculation.
    For this case, the maximal absolute value of the roots of $P(x)=x^4-3x^2+1$ is $\frac{1+\sqrt{5}}{2}$ and $\left(\frac{1+\sqrt{5}}{2}\right)^2$ is a lower bound except $1$ for this case.\\
   $\bullet$ $\sigma_i(b)=-1$ for all $1\leq i\leq e$\par
    $a=1$ provides the minimum value of the first dynamical degrees as in Section \ref{The minimum value of the first dynamical degrees of automorphisms}.
    Thus, $\left(\frac{1+\sqrt{5}}{2}\right)^2$ is a lower bound except $1$ for this case.\par
   Thus, $\frac{1+\sqrt{5}}{2}$ is the minimum value of the first dynamical degrees except $1$ for these types by the next theorem.
   \begin{theorem}\label{4-dimensional construction}
    Define $\alpha=\sqrt{-\frac{1+\sqrt{5}}{2}}$.
    There exists a prime number $p$ where $B:=\left(\frac{\alpha^2, p}{\mathbb{Q}(\sqrt{5})}\right)$ is a divisional totally indefinite quaternion algebra by Theorem \ref{divisional}.
    On this condition, there is a positive anti-involotion on $B$ over $\mathbb{Q}$.
    Moreover, on the case \textnormal{Type 2 ($d=2$, $e=2$, $m=1$)}, there is an automorphism of a $4$-dimensional simple abelian vaeriety which correponds to $\alpha=\sqrt{-\frac{1+\sqrt{5}}{2}}$.
   \end{theorem}
   \begin{proof}
    By Theorem \ref{construction of positive anti-involution}, it is enough to find $c\in B\setminus\mathbb{Q}(\sqrt{5})$ such that $c^2\in\mathbb{Q}(\sqrt{5})$ is totally negative.\par
    Denote $c=xi+yij$ ($x,y\in\mathbb{Q}(\sqrt{5})$) where $1,i,j,ij$ is an $\mathbb{Q}(\sqrt{5})$-basis of $B$ with $i^2=\alpha^2$ and $j^2=p$.
    Then $c^2=x^2\cdot\left(-\frac{1+\sqrt{5}}{2}\right)-y^2\cdot\left(-\frac{1+\sqrt{5}}{2}\right)\cdot p=\left(\frac{1+\sqrt{5}}{2}\right)(y^2p-x^2)$ by calculation.
    Take a pair $(u,w)$ of integers which are larger than $p$, which satisfies a Pell's equation $u^2-5w^2=1$.
    By substituting $x=u+w\sqrt{5}, y=1$, then,
    \begin{align*}
     c^2=\left(\frac{1+\sqrt{5}}{2}\right)(p-(u+w\sqrt{5})^2)<0.
    \end{align*}
    Now $1=u^2-5w^2=(u+w\sqrt{5})(u-w\sqrt{5})$ and so $0<(u-w\sqrt{5})<1$.\par
    Thus, the another conjugate of $c^2\in\mathbb{Q}(\sqrt{5})$ is
    \begin{align*}
     \left(\frac{1-\sqrt{5}}{2}\right)(p-(u-w\sqrt{5})^2)<0.
    \end{align*}
    Therefore, $c^2\in\mathbb{Q}(\sqrt{5})$ is totally negative and so the division ring $B$ has a positive anti-involution over $\mathbb{Q}$.\par
    Now there is the order $\mathcal{O}=\mathcal{O}_{\mathbb{Q}(\sqrt{5})}\oplus\mathcal{O}_{\mathbb{Q}(\sqrt{5})}i\oplus\mathcal{O}_{\mathbb{Q}(\sqrt{5})}j\oplus\mathcal{O}_{\mathbb{Q}(\sqrt{5})}ij$ in $B$, so by Proposition \ref{construction for Type 2}, there exists a $4$-dimenional simple abelian variety whose endomorphism ring contains $\mathcal{O}$.
    Thus, the automorphism corresponding to $i\in\mathcal{O}$ corresponds to $\alpha=\sqrt{-\frac{1+\sqrt{5}}{2}}$ and so the proof is concluded.
   \end{proof}
  \begin{flushleft}{\bf{Type 4}}\end{flushleft}\par
   Let $X$ be a $4$-dimensional simple abelian variety of Type 4 and $f\in\mathrm{End}(X)$ be an automorphism of $X$ which corresponds to $\alpha\in B$.\\
   $\bullet$ If $d=1$, $[K:\mathbb{Q}]=2$ or $4$ or $8$ by the restriction in Table \ref{table1} and $\alpha\in U_K$.\par
    Now $[K_0:\mathbb{Q}]=1$ or $2$ or $4$, respectively.
    Denote the minimal polynomial of $\alpha$ over $\mathbb{Q}$ by $P(x)$.\par
    If $\alpha\in K_0$, the minimum value of the square of the maximal absolute value of the roots of $P(x)$ is $\left(\frac{1+\sqrt{5}}{2}\right)^2$ except $1$, by the deduction in Type 1.\par
    If $\alpha\in K\setminus K_0$, by denoting $a=\abs{\alpha}^2\in U_{K_0}$, as in Section \ref{The minimum value of the first dynamical degrees of automorphisms}, $\mathbb{Q}(\sqrt{a})$ is a totally real number field.
    If $a=1$, then the first dynamical degree would be $1$ and so we assume $a\neq1$.
    Also, by assuming that the maximal absolute value of the conjugates of $\sqrt{a}$ is less than $2$, it can be written as $\sqrt{a}=\zeta_N^m+\frac{1}{\zeta_N^m}$ ($m,N\in\mathbb{Z}_{>0}$) with $m<N$, $\mathrm{gcd}(N,m)=1$ and we can assume $N\geq3$.
    Thus, $a=\left(\zeta_N^m+\frac{1}{\zeta_N^m}\right)^2=\zeta_N^{2m}+\frac{1}{\zeta_N^{2m}}+2$ and so 
    \begin{align*}
     [\mathbb{Q}(a):\mathbb{Q}]=\left\{
                                 \begin{array}{ll}
                                  \frac{1}{2}\varphi(\frac{N}{2}) & (N:\text{even})\\ [+5pt]
                                  \frac{1}{2}\varphi(N) & (N:\text{odd})
                                 \end{array}.
                                \right.
    \end{align*} 
    Now $K_0\supset\mathbb{Q}(a)\supset\mathbb{Q}$ and so $[\mathbb{Q}(a):\mathbb{Q}]=1$ or $2$ or $4$, and this implies $N=5,8,10,12$, $15,16,20,24,30,32,40,48,60$.
    By comparing with Proposition \ref{constant term of minimal polynomial}, since $P(x)$ has constant term $\pm1$, $N=5,10,15,24,30,40,48,60$ are only possible.
    Thus, the minimum value of the maximal absolute value of the roots of $P(x)$ is $2\mathrm{cos}\left(\frac{\pi}{5}\right)$ except $1$.\par
    Thus, $4\mathrm{cos}^2\left(\frac{\pi}{5}\right)$ is the minimum value of the first dynamical degrees except $1$ for this case.\\
    $\bullet$ If $d=2$, $[K:\mathbb{Q}]=2$ and denote $K=\mathbb{Q}(\sqrt{-D})$ for some square-free integer $D>0$.\par
    Denote the minimal polynomial of $\alpha$ over $K$ by $p(x)\in\mathcal{O}_{K}[x]$ and the minimal polynomial of $\alpha$ over $\mathbb{Q}$ by $P(x)$.\par
    If the degree of $p(x)$ is $1$, then $\alpha\in U_K$ and so the maximal absolute value of the conjugates is always $1$.\par
    For the case that $p(x)$ is $2$, it can be written as $p(x)=x^2+sx+t\in\mathcal{O}_{K}[x]$ with $t\in U_{K}$.
    By Dirichlet's unit theorem, $U_K$ is generated by some cyclotomic element $\zeta$ which has the smallest positive angle and denote $t=\zeta^n$.\par
    If the degree of $P(x)$ is $2$, then it can be written as $x^2+sx\pm1\in\mathbb{Z}[x]$ and the minimum value of the maximal absolte value of the roots of $P(x)$ is $\frac{1+\sqrt{5}}{2}$ except $1$.\par
    Otherwise, $P(x)$ has degree $4$ of the form $(x^2+sx+\zeta^n)(x^2+\overline{s}x+\zeta^{-n})\in\mathbb{Z}[x]$.
    The roots of this polynomial are $\alpha,\frac{\zeta^n}{\alpha},\overline{\alpha},\frac{\zeta^{-n}}{\overline{\alpha}}$ and so the maximal absolute value of these roots is $\mathrm{max}\{\abs{\alpha},\frac{1}{\abs{\alpha}}\}$.
    Assuming that $\mathrm{max}\{\abs{\alpha},\frac{1}{\abs{\alpha}}\}<1.5$, then $s\in\mathcal{O}_K$ satisfies $\abs{s}<1.5+\frac{1}{1.5}=\frac{13}{6}$.\par
    For the case $D\neq1,3$, $U_K=\{\pm1\}$ and so $t$ is $\pm1$ and this implies $s\notin\mathbb{Z}$.\\
    $\bullet$ If $-D\equiv1$ $(\mathrm{mod}\ 4)$, then $\mathcal{O}_{\mathbb{Q}(\sqrt{-D})}=\mathbb{Z}\left[\frac{1+\sqrt{-D}}{2}\right]$.\par
     By denoting $s=p+q\left(\frac{1+\sqrt{-D}}{2}\right)\in\mathcal{O}_{\mathbb{Q}(\sqrt{-D})}$ ($p,q\in\mathbb{Z}$), 
     \begin{align*}
      \abs{s}^2=\left(p+\frac{q}{2}\right)^2+\frac{q^2D}{4}=\frac{(2p+q)^2+q^2D}{4}<\frac{169}{36}.
     \end{align*}\par
     Thus, if $D\neq3$, then $q\neq0$ and so $D=7,11,15$.\\
    $\bullet$ If $-D\equiv2,3$ $(\mathrm{mod}\ 4)$, then $\mathcal{O}_{\mathbb{Q}(\sqrt{-D})}=\mathbb{Z}\left[\sqrt{-D}\right]$.\par
      By denoting $s=p+q\sqrt{-D}\in\mathcal{O}_{\mathbb{Q}(\sqrt{-D})}$ ($p,q\in\mathbb{Z}$), 
      \begin{align*}
       \abs{s}^2=p^2+q^2D<\frac{169}{36}.
      \end{align*}\par
      Thus, if $D\neq1$, then $q\neq0$ and so $D=2$.\\
    For any integer $m$, the absolute values of the roots of $(x^2+sx+t)(x^2+\overline{s}x+\overline{t})$ and $(x^2+s\zeta^mx+t\zeta^{2m})(x^2+\overline{s}\zeta^{-m}x+\overline{t}\zeta^{-2m})$ are the same and so the maximal absolute value is the same too.
    Also, $s\neq0$ by assuming that the maximal absolute value of the roots is larger than $1$.\par
    Thus, only the cases $0\leq\mathrm{arg}(s)<\mathrm{arg}(\zeta)$ are worth to consider and they are enumerated in Table \ref{appendix 5} and the minimum of the square of the maximal absolute value of the roots of $P(x)$ is $1.3122\cdots^2=1.7220\cdots$ except $1$.\\
   \\
   Therefore, the minimum value of the first dynamical degrees of automorphisms of $4$-dimensional simple abelian varieties is $\frac{1+\sqrt{5}}{2}=1.6180\cdots$ except $1$, and realized on a simple abelian variety of Type 2.

 \subsection{dimension $g=6$}
  Let $X$ be a $6$-dimensional simple abelian variety and define $B$, $K$, $K_0$, $d$, $e$ and $e_0$ as in Section \ref{Endomorphism algebras of simple abelian varieties}.
  \begin{flushleft}{\bf{Type 1}}\end{flushleft}\par
   Let $X$ be a $6$-dimensional simple abelian variety of Type 1 and $f\in\mathrm{End}(X)$ be an automorphism of $X$ which corresponds to $\alpha\in U_K$ as in Section \ref{The minimum value of the first dynamical degrees of automorphisms}.
   Now $[K:\mathbb{Q}]=1$ or $2$ or $3$ or $6$ by the restriction in Table \ref{table1}.\par
   Assuming that the first dynamical degree is less than $4$, by using similar deduction as in the $4$-dimensional case, the minimum value of the maximal absolute value of the roots of $P(x)$ is $2\mathrm{cos}\left(\frac{\pi}{5}\right)$ except $1$.
   The realizability is concerned in Section \ref{The minimum value of the first dynamical degrees of automorphisms}.
   Thus, the minimum value of the first dynamical degrees is $4\mathrm{cos}^2\left(\frac{\pi}{5}\right)=\left(\frac{1+\sqrt{5}}{2}\right)^2$ except $1$ for this type.
  \begin{flushleft}{\bf{Type 2, Type 3}}\end{flushleft}\par
   Let $X$ be a $6$-dimensional simple abelian variety of Type 2 or 3 and $f\in\mathrm{End}(X)$ be an automorphism of $X$ which corresponds to $\alpha\in B$.
   Now $[K:\mathbb{Q}]=1$ or $3$ by the restriction in Table \ref{table1}.\par
   Denote the minimal polynomial of $\alpha$ over $K$ by $p(x)$ and the minimal polynomial of $\alpha$ over $\mathbb{Q}$ by $P(x)$.
   If the degree of $p(x)$ is $1$, then $\alpha\in U_K$ and so it is reduced to Type 1 of $3$-dimensional case.
   Thus, $4\mathrm{cos}^2\left(\frac{\pi}{7}\right)$ is a lower bound except $1$ for this case.\par
   If the degree of $p(x)$ is $2$, it can be written as $p(x)=x^2+ax+b\in\mathcal{O}_K[x](b\in U_K)$ and consider the following 3 cases as in Section \ref{The minimum value of the first dynamical degrees of automorphisms}.
   Let $\{\sigma_i\}$ be the set of all $\mathbb{Q}$-embeddings $K\hookrightarrow\mathbb{C}$.\\
   $\bullet$ there exists $1\leq i\leq e$ such that $\abs{\sigma_i(b)}>1$\par
    The maximal absolute value of all conjugates of $b\in U_K$ is at least $2\mathrm{cos}\left(\frac{\pi}{7}\right)$ and so the maximal absolute value of the roots of $P(x)$ is at least $\sqrt{2\mathrm{cos}\left(\frac{\pi}{7}\right)}$.
    This is achived by $p(x)=x^2-2\mathrm{cos}\left(\frac{2\pi}{7}\right)$ and $P(x)=(x^2-2\mathrm{cos}\left(\frac{2\pi}{7}\right))(x^2-2\mathrm{cos}\left(\frac{4\pi}{7}\right))(x^2-2\mathrm{cos}\left(\frac{6\pi}{7}\right))$.
    Thus, $2\mathrm{cos}\left(\frac{\pi}{7}\right)=1.8019\cdots$ is a lower bound except $1$ for this case.\\
   $\bullet$ $\sigma_i(b)=1$ for all $1\leq i\leq e$\par
    $\mathrm{max}\{\abs{\sigma_i(a)}\}$ must be larger than $2$ and close to $2$.
    By using a computer algebra, totally real algebraic integers of degree $3$ whose conjugates have modulus less than $2.25$ are enumerated in Table \ref{appendix 6}.
    $a=2.1149\cdots$ provides the minimum of the maximal absolute value.
    For this case, the maximal absolute value of the roots of $P(x)$ is $1.4012\cdots$ and so $1.4012\cdots^2=1.9635\cdots$ is a lower bound except $1$ for this case.\\
   $\bullet$ $\sigma_i(b)=-1$ for all $1\leq i\leq e$\par
    $a=1$ provides the minimum value of the first dynamical degrees as in Section \ref{The minimum value of the first dynamical degrees of automorphisms}.
    Thus, $\left(\frac{1+\sqrt{5}}{2}\right)^2$ is a lower bound except $1$ for this case.\par
   Thus, $2\mathrm{cos}\left(\frac{\pi}{7}\right)=1.8019\cdots$ is the minimum value of the first dynamical degrees except $1$ for these types by the next theorem.
   \begin{theorem}
    Define $\alpha:=\sqrt{\zeta_7+\frac{1}{\zeta_7}}$.
    Theorem \ref{divisional} implies that there exists a prime\nolinebreak\ number $p$ where $B:=\left(\frac{\zeta_7+\frac{1}{\zeta_7},\ p}{\mathbb{Q}(\zeta_7+\frac{1}{\zeta_7})}\right)$ is a divisional totally indefinite quaternion algebra.
    On this condition, there is a positive anti-involotion on $B$ over $\mathbb{Q}$.
    Moreover, there is an automorphism of a $6$-dimensional simple abelian vaeriety of \textnormal{Type 2 ($d=2$, $e=3$, $m=1$)} which correponds to $\alpha=\sqrt{\zeta_7+\frac{1}{\zeta_7}}$.
   \end{theorem}
   \begin{proof}
    By Theorem \ref{construction of positive anti-involution}, it is enough to find $c\in B\setminus\mathbb{Q}\left(\zeta_7+\frac{1}{\zeta_7}\right)$ such that $c^2\in\mathbb{Q}\left(\zeta_7+\frac{1}{\zeta_7}\right)$ is totally negative.\par
    Denote $c=xi+yij\in B$ where $1,i,j,ij$ is an $\mathbb{Q}\left(\zeta_7+\frac{1}{\zeta_7}\right)$-basis of $B$ with $i^2=\zeta_7+\frac{1}{\zeta_7}$ and $j^2=p$.
    Then $c^2=x^2\cdot\left(\zeta_7+\frac{1}{\zeta_7}\right)-y^2\cdot\left(\zeta_7+\frac{1}{\zeta_7}\right)\cdot p=\left(\zeta_7+\frac{1}{\zeta_7}\right)(x^2-y^2p)$ by calculation.
    By substituting $x=a+b\left(\zeta_7+\frac{1}{\zeta_7}\right), y=1$ with $a,b\in\mathbb{Z}$, then,
    \begin{align*}
     c^2=\left(\zeta_7+\frac{1}{\zeta_7}\right)\left(\left(a+b\left(\zeta_7+\frac{1}{\zeta_7}\right)\right)^2-p\right).
    \end{align*}
    The other conjugates of $c^2\in\mathbb{Q}\left(\zeta_7+\frac{1}{\zeta_7}\right)$ are
    \begin{align*}
     \left(\zeta_7^2+\frac{1}{\zeta_7^2}\right)\left(\left(a+b\left(\zeta_7^2+\frac{1}{\zeta_7^2}\right)\right)^2-p\right)
    \end{align*}
    and
    \begin{align*}
     \left(\zeta_7^3+\frac{1}{\zeta_7^3}\right)\left(\left(a+b\left(\zeta_7^3+\frac{1}{\zeta_7^3}\right)\right)^2-p\right).
    \end{align*}
    Take $(a,b)$ as $a>0$, $b<0$, $0<a+b\left(\zeta_7+\frac{1}{\zeta_7}\right)<1$ and then by $a\to\infty$, $b\to-\infty$, it holds that $c^2<0$, $\abs[\big]{a+b\left(\zeta_7^2+\frac{1}{\zeta_7^2}\right)}\to\infty$ and $\abs[\big]{a+b\left(\zeta_7^3+\frac{1}{\zeta_7^3}\right)}\to\infty$.
    Thus, all the conjugates of $c^2$ can be negative and so it defines a positive anti-involution on $B$ over $\mathbb{Q}$,\par
    Now there is the order $\mathcal{O}=\mathcal{O}_{\mathbb{Q}\left(\zeta_7+\frac{1}{\zeta_7}\right)}\oplus\mathcal{O}_{\mathbb{Q}\left(\zeta_7+\frac{1}{\zeta_7}\right)}i\oplus\mathcal{O}_{\mathbb{Q}\left(\zeta_7+\frac{1}{\zeta_7}\right)}j\oplus\mathcal{O}_{\mathbb{Q}\left(\zeta_7+\frac{1}{\zeta_7}\right)}ij$ in $B$, so by Proposition \ref{construction for Type 2}, there exists a $6$-dimenional simple abelian variety whose endomorphism ring contains $\mathcal{O}$.
    Thus, the automorphism corresponding to $i\in\mathcal{O}$ corresponds to $\alpha=\sqrt{\zeta_7+\frac{1}{\zeta_7}}$ and so the proof is concluded.
   \end{proof}
  \begin{flushleft}{\bf{Type 4}}\end{flushleft}\par
   Let $X$ be a $6$-dimensional simple abelian variety of Type 4 and $f\in\mathrm{End}(X)$ be an automorphism of $X$ which corresponds to $\alpha\in B$.\par
   Now $d=1$ and $[K:\mathbb{Q}]=2$ or $4$ or $6$ or $12$ by the restriction in Table \ref{table1} and $\alpha\in U_K$.\par
   Now $[K_0:\mathbb{Q}]=1$ or $2$ or $3$ or $6$, respectively.
   Denote the minimal polynomial of $\alpha$ over $\mathbb{Q}$ by $P(x)$.\par
   If $\alpha\in K_0$, the minimum value of the square of the maximal absolute value of the roots of $P(x)$ is $4\mathrm{cos}^2\left(\frac{\pi}{7}\right)$ except $1$, by the deduction in Type 1.\par
   If $\alpha\in K\setminus K_0$, by denoting $a=\abs{\alpha}^2\in U_{K_0}$, as in Section \ref{The minimum value of the first dynamical degrees of automorphisms}, $\mathbb{Q}(\sqrt{a})$ is a totally real number field.
   By using similar deduction as in the $4$-dimensional case, the minimum value of the square of the maximal absolute value of the roots of $P(x)$ is $4\mathrm{cos}^2\left(\frac{\pi}{5}\right)$ except $1$.\\
   \\
  Therefore, the minimum value of the first dynamical degrees of automorphisms of $6$-dimensional simple abelian varieties is $2\mathrm{cos}\left(\frac{\pi}{7}\right)$ except $1$, and realized on a simple abelian variety of Type 2.

 \subsection{dimension $g=7$}
  By using a computer algebra,
  \begin{align*}
   \mathrm{min}\left\{
    \begin{array}{ll}
     \text{the maximal absolute value of the conjugates of an}\\
     \text{algebraic unit, whose conjugates are all real and positive, of degree $p$}
    \end{array}
   \right\}
  \end{align*}
  is $4.0333\cdots$, the maximum root of the polynonial $x^7-14x^6+77x^5-211x^4+301x^3-210x^2+56x-1$ by Table \ref{appendix 7}.
  By Proposition \ref{case otherwise}, $m(p)=4.0333\cdots$ and it is realized on a simple abelian variety of Type 4.

 \subsection{dimension $g=8$}
  Let $X$ be an $8$-dimensional simple abelian variety and define $B$, $K$, $K_0$, $d$, $e$ and $e_0$ as in Section \ref{Endomorphism algebras of simple abelian varieties}.
  \begin{flushleft}{\bf{Type 1}}\end{flushleft}\par
   Let $X$ be an $8$-dimensional simple abelian variety of Type 1 and $f\in\mathrm{End}(X)$ be an automorphism of $X$ which corresponds to $\alpha\in U_K$ as in Section \ref{The minimum value of the first dynamical degrees of automorphisms}.
   Now $[K:\mathbb{Q}]=1$ or $2$ or $4$ or $8$ by the restriction in Table \ref{table1}.\par
   Assuming that the first dynamical degree is less than $4$, by using similar deduction as in the $4$-dimensional case, the minimum value of the maximal absolute value of the roots of $P(x)$ is $2\mathrm{cos}\left(\frac{\pi}{5}\right)$ except $1$.
   The realizability is concerned in Section \ref{The minimum value of the first dynamical degrees of automorphisms}.
   Thus, the minimum value of the first dynamical degrees is $4\mathrm{cos}^2\left(\frac{\pi}{5}\right)=\left(\frac{1+\sqrt{5}}{2}\right)^2$ except $1$ for this type.
  \begin{flushleft}{\bf{Type 2, Type 3}}\end{flushleft}\par
   Let $X$ be an $8$-dimensional simple abelian variety of Type 2 or 3 and $f\in\mathrm{End}(X)$ be an automorphism of $X$ which corresponds to $\alpha\in B$.
   Now $[K:\mathbb{Q}]=1$ or $2$ or $4$ by the restriction in Table \ref{table1}.\par
   Denote the minimal polynomial of $\alpha$ over $K$ by $p(x)$ and the minimal polynomial of $\alpha$ over $\mathbb{Q}$ by $P(x)$.
   If the degree of $p(x)$ is $1$, then $\alpha\in U_K$ and so it is reduced to Type 1 of $4$-dimensional case.
   Thus, the minimum value is $\frac{3+\sqrt{5}}{2}$ for this case.\par
   If the degree of $p(x)$ is $2$, it can be written as $p(x)=x^2+ax+b\in\mathcal{O}_K[x](b\in U_K)$ and consider the following 3 cases as in Section \ref{The minimum value of the first dynamical degrees of automorphisms}.
   Let $\{\sigma_i\}$ be the set of all $\mathbb{Q}$-embeddings $K\hookrightarrow\mathbb{C}$.\\
   $\bullet$ there exists $1\leq i\leq e$ such that $\abs{\sigma_i(b)}>1$\par
    The maximal absolute value of all conjugates of $b\in U_K$ is at least $2\mathrm{cos}\left(\frac{\pi}{5}\right)$ and so the maximal absolute value of the roots of $P(x)$ is at least $\sqrt{2\mathrm{cos}\left(\frac{\pi}{5}\right)}$.
    This is achived as in $4$-dimensional case. 
    Thus, $2\mathrm{cos}\left(\frac{\pi}{5}\right)=1.6180\cdots$ is a lower bound except $1$ for this case.\par
    This is realizable for the case of Type 2 ($d=2$, $e=2$, $m=2$) as in Theorem \ref{4-dimensional construction}.\\
   $\bullet$ $\sigma_i(b)=1$ for all $1\leq i\leq e$\par
    $\mathrm{max}\{\abs{\sigma_i(a)}\}$ must be larger than $2$ and close to $2$.
    By using a computer algebra, totally real algebraic integers of degree $4$ whose conjugates have modulus less than $2.1$ are enumerated in Table \ref{appendix 8}.
    $a=2.0614\cdots$ provides the minimum of the maximal absolute value.
    For this case, the maximal absolute value of the roots of $P(x)$ is $1.2806\cdots$ and so $1.2806\cdots^2=1.6400\cdots$ is a lower bound except $1$ for this case.\\
   $\bullet$ $\sigma_i(b)=-1$ for all $1\leq i\leq e$\par
    $a=1$ provides the minimum value of the first dynamical degrees as in Section \ref{The minimum value of the first dynamical degrees of automorphisms}.
    Thus, $\left(\frac{1+\sqrt{5}}{2}\right)^2$ is a lower bound except $1$ for this case.\par
   Therefore, $2\mathrm{cos}\left(\frac{\pi}{5}\right)$ is the minimum value of the first dynamical degrees for these types.
  \begin{flushleft}{\bf{Type 4}}\end{flushleft}\par
   Let $X$ be an $8$-dimensional simple abelian variety of Type 4 and $f\in\mathrm{End}(X)$ be an automorphism of $X$ which corresponds to $\alpha\in B$.\\
   $\bullet$ If $d=1$, $[K:\mathbb{Q}]=2$ or $4$ or $8$ or $16$ by the restriction in Table \ref{table1} and $\alpha\in U_K$.\par
    Now $[K_0:\mathbb{Q}]=1$ or $2$ or $4$ or $8$, respectively.
    Denote the minimal polynomial of $\alpha$ over $\mathbb{Q}$ by $P(x)$.\par
    If $\alpha\in K_0$, the minimum value of the square of the maximal absolute value of the roots of $P(x)$ is $\left(\frac{1+\sqrt{5}}{2}\right)^2$ except $1$, by the deduction in Type 1.\par
    If $\alpha\in K\setminus K_0$, by denoting $a=\abs{\alpha}^2\in U_{K_0}$, as in Section \ref{The minimum value of the first dynamical degrees of automorphisms}, $\mathbb{Q}(\sqrt{a})$ is a totally real number field.
    By using similar deduction as in the $4$-dimensional case, the minimum value of the maximal absolute value of the roots of $P(x)$ is $2\mathrm{cos}\left(\frac{\pi}{5}\right)$ except $1$.\par
    Thus, $4\mathrm{cos}^2\left(\frac{\pi}{5}\right)$ is the minimum value of the first dynamical degrees except $1$ for this case.\\
   $\bullet$ If $d=2$, $[K:\mathbb{Q}]=2$ or $4$ is only possible by the restriction in Table \ref{table1}.\par
    If $[K:\mathbb{Q}]=2$, then it is reduced to the case of $4$-dimensional, and $1.3122\cdots^2=1.7220\cdots$ is a lower bound of the first dynamical degrees except $1$.\par
    For the case $[K:\mathbb{Q}]=4$, let $p(x)$ be the minimal polynomial of $\alpha$ over $K$ and let $P(x)$ be the minimal polynomial of $\alpha$ over $\mathbb{Q}$.
    If the degree of $p(x)$ is $1$, then $\alpha\in U_K$ and so $\frac{3+\sqrt{5}}{2}$ is a lower bound except $1$ as the case of $2$-dimensional.\par
    For the case that the degree of $p(x)$ is $2$, it can be written as $p(x)=x^2+sx+t\in\mathcal{O}_{K}[x]$ with $t\in U_{K}$.
    If $\mathrm{max}\{\text{modulus of conjugates of }t\}>1$, then the maximal absolute value of the roots of $P(x)$ is at least $\sqrt{\mathrm{max}\{\text{modulus of conjugates of }t\}}$ and so $\sqrt{\frac{1+\sqrt{5}}{2}}$ is a lower bound.
    Thus, $\frac{1+\sqrt{5}}{2}$ is a lower bound of the first dynamical degrees except $1$ for this case.\par
    For the case $\mathrm{max}\{\text{modulus of conjugates of }t\}=1$, by $[K:\mathbb{Q}]=4$,
    \begin{align*}
     t\in\{1,-1,\zeta_3,\zeta_3^2,i,-i,\zeta_5,\zeta_5^2,\zeta_5^3,\zeta_5^4,\zeta_6,\zeta_6^5,\zeta_8,\zeta_8^3,\zeta_8^5,\zeta_8^7,\zeta_{10},\zeta_{10}^3,\zeta_{10}^7,\zeta_{10}^9, \zeta_{12},\zeta_{12}^5,\zeta_{12}^7,\zeta_{12}^{11}\}
    \end{align*}
    For considering the absolute value of the roots of $P(x)$, by the correspondence
    \begin{align*}
     &(x^2+sx+\zeta_3)\leftrightarrow(x^2+s\zeta_3x+1)\leftrightarrow(x^2+s\zeta_3^2x+\zeta_3^2),\\
     &(x^2+sx+\zeta_5)\leftrightarrow(x^2+s\zeta_5x+\zeta_5^3)\leftrightarrow(x^2+s\zeta_5^2x+1)\\
     &\hspace{5cm}\leftrightarrow(x^2+s\zeta_5^3x+\zeta_5^2)\leftrightarrow(x^2+s\zeta_5^4x+\zeta_5^4),\\
     &(x^2+sx+\zeta_6)\leftrightarrow(x^2+s\zeta_6x-1)\leftrightarrow(x^2+s\zeta_6^2x+\zeta_6^5),\\
     &(x^2+sx+\zeta_{10})\leftrightarrow(x^2+s\zeta_{10}x+\zeta_{10}^3)\leftrightarrow(x^2+s\zeta_{10}^2x-1)\\
     &\hspace{5cm}\leftrightarrow(x^2+s\zeta_{10}^3x+\zeta_{10}^7)\leftrightarrow(x^2+s\zeta_{10}^4x+\zeta_{10}^9),\\
     &(x^2+sx+i)\leftrightarrow(x^2+six-i),\\
     &(x^2+sx+\zeta_{8})\leftrightarrow(x^2+s\zeta_{8}x+\zeta_{8}^3)\leftrightarrow(x^2+s\zeta_{8}^2x+\zeta_{8}^5)\leftrightarrow(x^2+s\zeta_{8}^3x+\zeta_{8}^7),\\
     &(x^2+sx+\zeta_{12})\leftrightarrow(x^2+s\zeta_{12}^2x+\zeta_{12}^5)\leftrightarrow(x^2+s\zeta_{12}^3x+\zeta_{12}^7)\leftrightarrow(x^2+s\zeta_{12}^5x+\zeta_{12}^{11}),
    \end{align*}
    it suffices to consider the cases $t=1,-1,i,\zeta_8,\zeta_{12}$.
    Denote $K_0=\mathbb{Q}(\sqrt{D})$ where $D\in\mathbb{Z}_{>0}$ is a square-free integer.
    $s\in\mathcal{O}_K$ is of degree at most $2$ over $K_0$.
    \begin{enumerate}
     \item{Case $t=1$}\\
      It can be written as $P(x)=(x^2+sx+1)(x^2+s_1x+1)(x^2+s_2x+1)(x^2+s_3x+1)$ where $s=s_0,s_1,s_2,s_3$ are conjugates over $\mathbb{Q}$.
      Let
      \begin{align}\label{equation1}
       R(X)&=(X^2+(a+b\sqrt{D})X+a'+b'\sqrt{D})(X^2+(a-b\sqrt{D})X+a'-b'\sqrt{D})\nonumber\\
           &=X^4+2aX^3+(a^2-b^2D+2a')X^2+(2aa'-2bb'D)X+a'^2-b'^2D
      \end{align}
      be a polynomial, with $a+b\sqrt{D},a'+b'\sqrt{D}\in\mathcal{O}_{\mathbb{Q}(\sqrt{D})}$, which has $s=s_0,s_1,s_2,s_3$ as the roots.
      For $N=1.28$, all roots of $P(x)$ have modulus less than $N$ if and only if $s=s_0,s_1,s_2,s_3$ are all inside the domain
      \footnotesize
      \begin{align*}
       E:=\left\{z\in\mathbb{C}\;\middle|\;\mathrm{Re}(z)^2\left(N-\frac{1}{N}\right)^2+\mathrm{Im}(z)^2\left(N+\frac{1}{N}\right)^2<\left(N+\frac{1}{N}\right)^2\left(N-\frac{1}{N}\right)^2\right\}
      \end{align*}
      \normalsize
      The quartic polynomials, whose roots are all in $E$, are showed in Table \ref{appendix 9} and considering the polynomial is of the form (\ref{equation1}), or not.
      Thus, $1.2720\cdots=\sqrt{\frac{1+\sqrt{5}}{2}}$ is the minimum value and so $\frac{1+\sqrt{5}}{2}$ is a lower bound except $1$ for this case.
     \item{Case $t=-1$}\\
      It can be written as $P(x)=(x^2+sx-1)(x^2+s_1x-1)(x^2+s_2x-1)(x^2+s_3x-1)$ where $s=s_0,s_1,s_2,s_3$ are conjugates over $\mathbb{Q}$.\par
      For $N=1.28$, all roots of $P(x)$ have modulus less than $N$ if and only if $s=s_0,s_1,s_2,s_3$ are all inside the domain
      \footnotesize
      \begin{align*}
       E_1:=\left\{z\in\mathbb{C}\;\middle|\;\mathrm{Re}(z)^2\left(N+\frac{1}{N}\right)^2+\mathrm{Im}(z)^2\left(N-\frac{1}{N}\right)^2<\left(N+\frac{1}{N}\right)^2\left(N-\frac{1}{N}\right)^2\right\}
      \end{align*}
      \normalsize
      The quartic polynomials, whose roots are all in $E_1$, are showed in Table \ref{appendix 10} and considering the polynomial is of the form (\ref{equation1}), or not.
      Thus, $1.2720\cdots=\sqrt{\frac{1+\sqrt{5}}{2}}$ is the minimum value and so $\frac{1+\sqrt{5}}{2}$ is a lower bound except $1$ for this case.
     \item{Case $t=i$}\\
      It can be written as $P(x)=(x^2+sx+i)(x^2+tx+i)(x^2+\overline{s}x-i)(x^2+\overline{t}x-i)$ where $s,t,\overline{s},\overline{t}$ are conjugates over $\mathbb{Q}$.
      Especially, $s,t$ are conjugates over $\mathbb{Q}(i)$.\par
      $z=s,t$ is the roots of
      \begin{align}\label{equation2}
       R_1(z)=z^2+(a+bi)z+(c+di)\in\mathbb{Z}[i][z].
      \end{align}
      For $N=1.28$, all roots of $P(x)$ have modulus less than $N$ if and only if $s,t$ are inside the domain
      \begin{align*}
       E_2:=\left\{z\in\mathbb{C}\mid z\cdot\zeta_8\in E_1\right\}
      \end{align*}
      The quadratic polynomials of the form (\ref{equation2}), whose roots are all in $E_2$, are showed in Table \ref{appendix 11}.
      Thus, $1.2720\cdots=\sqrt{\frac{1+\sqrt{5}}{2}}$ is the minimum value except $1$ and so $\frac{1+\sqrt{5}}{2}$ is a lower bound except $1$ for this case.
     \item{Case $t=\zeta_8$}\\
      $K=\mathbb{Q}(\zeta_8)$ and $K_0=\mathbb{Q}\left(\zeta_8+\frac{1}{\zeta_8}\right)=\mathbb{Q}(\sqrt{2})$.
      It can be written as $P(x)=(x^2+sx+\zeta_8)(x^2+tx+\zeta_8^3)(x^2+\overline{t}x+\zeta_8^5)(x^2+\overline{s}x+\zeta_8^7)$ where $s,t,\overline{t},\overline{s}$ are conjugates over $\mathbb{Q}$.
      Especially, $s,t$ are conjugates over $\mathbb{Q}(\sqrt{2}i)$.\par
      $z=s,t$ is the roots of
      \begin{align}\label{equation3}
       R_2(z)=z^2+(a+b\sqrt{2}i)z+(c+d\sqrt{2}i)\in\mathbb{Z}[\sqrt{2}i][z].
      \end{align}
      For $N=1.28$, all roots of $P(x)$ have modulus less than $N$ if and only if $s$ is inside the domain
      \begin{align*}
       E_3:=\left\{z\in\mathbb{C}\mid z\cdot\zeta_{16}^3\in E_1\right\}
      \end{align*}
      and $t$ is inside the domain
      \begin{align*}
       E_4:=\left\{z\in\mathbb{C}\mid z\cdot\zeta_{16}\in E_1\right\}.
      \end{align*}
      The quadratic polynomials of the form (\ref{equation3}), with the roots $s,t\in\mathbb{Q}(\zeta_8)$ such that $s\in E_3$, $t\in E_4$ and $(x^2+sx+\zeta_8)(x^2+tx+\zeta_8^3)\in\mathbb{Z}[\sqrt{2}i][z]$, are showed in Table \ref{appendix 12}.
      Thus, $1.28^2$ is a lower bound except $1$ for this case.
     \item{Case $t=\zeta_{12}$}\\
      $K=\mathbb{Q}(\zeta_{12})$ and $K_0=\mathbb{Q}\left(\zeta_{12}+\frac{1}{\zeta_{12}}\right)=\mathbb{Q}(\sqrt{3})$.
      It can be written as $P(x)=(x^2+sx+\zeta_{12})(x^2+tx+\zeta_{12}^5)(x^2+\overline{t}x+\zeta_{12}^7)(x^2+\overline{s}x+\zeta_{12}^{11})$ where $s,t,\overline{t},\overline{s}$ are conjugates over $\mathbb{Q}$.
      Especially, $s,t$ are conjugates over $\mathbb{Q}(i)$.\par
      $z=s,t$ is the roots of
      \begin{align}\label{equation4}
       R_3(z)=z^2+(a+bi)z+(c+di)\in\mathbb{Z}[i][z].
      \end{align}
      For $N=1.28$, all roots of $P(x)$ have modulus less than $N$ if and only if $s$ is inside the domain
      \begin{align*}
       E_5:=\left\{z\in\mathbb{C}\mid z\cdot\zeta_{24}^5\in E_1\right\}
      \end{align*}
      and $t$ is inside the domain
      \begin{align*}
       E_6:=\left\{z\in\mathbb{C}\mid z\cdot\zeta_{24}\in E_1\right\}.
      \end{align*}
      The quadratic polynomials of the form (\ref{equation4}), with the roots $s,t\in\mathbb{Q}(\zeta_{12})$ such that $s\in E_5$, $t\in E_6$ and $(x^2+sx+\zeta_{12})(x^2+tx+\zeta_{12}^5)\in\mathbb{Z}[i][z]$, are showed in Table \ref{appendix 13}.
      Thus, $1.28^2$ is a lower bound except $1$ for this case.
    \end{enumerate}\par
    \noindent Therefore, the minimum value of the first dynamical degrees of automorphisms of $8$-dimensional simple abelian varieties is $\frac{1+\sqrt{5}}{2}$ except $1$, and realized on a simple abelian variety of Type 2.

 \subsection{dimension $g=9$}
  Let $X$ be a $9$-dimensional simple abelian variety and define $B$, $K$, $K_0$, $d$, $e$ and $e_0$ as in Section \ref{Endomorphism algebras of simple abelian varieties}.
  \begin{flushleft}{\bf{Type 1}}\end{flushleft}\par
   Let $X$ be a $9$-dimensional simple abelian variety of Type 1 and $f\in\mathrm{End}(X)$ be an automorphism of $X$ which corresponds to $\alpha\in U_K$ as in Section \ref{The minimum value of the first dynamical degrees of automorphisms}.
   Now $[K:\mathbb{Q}]=1$ or $3$ or $9$ by the restriction in Table \ref{table1}.\par
   Assuming that the first dynamical degree is less than $4$, by using similar deduction as in the $4$-dimensional case, the minimum value of the maximal absolute value of the roots of $P(x)$ is $2\mathrm{cos}\left(\frac{\pi}{7}\right)$ except $1$.
   The realizability is concerned in Section \ref{The minimum value of the first dynamical degrees of automorphisms}.
   Thus, the minimum value of the first dynamical degrees is $4\mathrm{cos}^2\left(\frac{\pi}{7}\right)$ except $1$ for this type.
  \begin{flushleft}{\bf{Type 4}}\end{flushleft}\par
   Let $X$ be a $9$-dimensional simple abelian variety of Type 4 and $f\in\mathrm{End}(X)$ be an automorphism of $X$ which corresponds to $\alpha\in B$.\\
   $\bullet$ If $d=1$, $[K:\mathbb{Q}]=2$ or $6$ or $18$ by the restriction in Table \ref{table1} and $\alpha\in U_K$.\par
   Now $[K_0:\mathbb{Q}]=1$ or $3$ or $9$, respectively.
   Denote the minimal polynomial of $\alpha$ over $\mathbb{Q}$ by $P(x)$.\par
   If $\alpha\in K_0$, the minimum value of the square of the maximal absolute value of the roots of $P(x)$ is $4\mathrm{cos}^2\left(\frac{\pi}{7}\right)$ except $1$, by the deduction in Type 1.\par
   If $\alpha\in K\setminus K_0$, by denoting $a=\abs{\alpha}^2\in U_{K_0}$, as in Section \ref{The minimum value of the first dynamical degrees of automorphisms}, $\mathbb{Q}(\sqrt{a})$ is a totally real number field.
   By using similar deduction as in the $4$-dimensional case, the minimum value of the square of the maximal absolute value of the roots of $P(x)$ is $4\mathrm{cos}^2\left(\frac{\pi}{7}\right)$ except $1$.\\
   $\bullet$ If $d=3$, $[K:\mathbb{Q}]=2$ is only possible by the restriction in Table \ref{table1}.\par
    Denote the minimal polynomial of $\alpha$ over $K$ by $p(x)\in\mathcal{O}_{K}[x]$ and the minimal polynomial of $\alpha$ over $\mathbb{Q}$ by $P(x)$.\par
    If the degree of $p(x)$ is $1$, then $\alpha\in U_K$ and so the maximal absolute value of the conjugates is always $1$.\par
    If the degree of $p(x)$ is $3$, then it can be written as $p(x)=x^3+sx^2+tx+u\in\mathcal{O}_{K}[x]$ with $u\in U_{K}$.
    By Dirichlet's unit theorem, $U_K$ contains only cyclotomic elements and so by $[K:\mathbb{Q}]=2$,
    \begin{align*}
     u\in\{1,-1,i,-i,\zeta_6,\zeta_6^2,\zeta_6^4,\zeta_6^5\}
    \end{align*}
    For considering the absolute value of the roots of $P(x)$, by the correspondence
    \begin{align*}
     &(x^3+sx^2+tx-1)\leftrightarrow(x^3-sx^2+tx+1),\\
     &(x^3+sx^2+tx+\zeta_6)\leftrightarrow(x^3+s\zeta_6x^2+t\zeta_6^2x+\zeta_6^4)\leftrightarrow(x^3+\overline{s}\zeta_6^5x^2+\overline{t}\zeta_6^4x+\zeta_6^2)\\
     &\hspace{10cm}\leftrightarrow(x^3+\overline{s}x^2+\overline{t}x+\zeta_6^5),\\
     &(x^3+sx^2+tx+i)\leftrightarrow(x^3+six^2-tx+1)\leftrightarrow(x^3-sx^2+tx-i)\leftrightarrow(x^3-six^2-tx-1),
    \end{align*}
    it suffices to consider the cases $u=1,\zeta_6$.
    Denote $K=\mathbb{Q}(\sqrt{-D})$ where $D\in\mathbb{Z}_{>0}$ is a square-free integer.
    Write $p(x)=x^3+(a+b\sqrt{-D})x^2+(a'+b'\sqrt{-D})x+u\in\mathcal{O}_{K}[x]$ and then $P(x)=(x^3+(a+b\sqrt{-D})x^2+(a'+b'\sqrt{-D})x+u)(x^3+(a-b\sqrt{-D})x^2+(a'-b'\sqrt{-D})x+\overline{u})$.
    For $N=1.3$, all roots of $P(x)$ have modulus less than $N$ if and only if all roots of $p(x)$ have modulus less than $N$.
    \begin{enumerate}
     \item{Case $u=1$, $b=0$, $b'=0$}\\
      $p(x)=x^3+ax^2+a'x+1\in\mathbb{Z}[x]$ and by Table \ref{appendix 14}, the minimum value of the maximal absolute value of the conjugates is $1.1509\cdots$.
      Thus, $1.1509\cdots^2=1.3247\cdots$ is a lower bound of the first dynamcal degrees except $1$ for this case.
     \item{Case $u=1$, $(b,b')\neq(0,0)$}\\
      Assuming that all roots of $p(x)$ have modulus less than $N=1.3$, the minimum value of the maximal absolute value of the conjugates is $1.1509\cdots$ by Table \ref{appendix 15}.
      Thus, $1.1509\cdots^2=1.3247\cdots$ is a lower bound of the first dynamcal degrees except $1$ for this case.
     \item{Case $u=\zeta_6$}\\
      Assuming that all roots of $p(x)$ have modulus less than $N=1.3$, the minimum value of the maximal absolute value of the conjugates is $1.2167\cdots$ by Table \ref{appendix 16}.
      Thus, $1.2167\cdots^2=1.4803\cdots$ is a lower bound of the first dynamcal degrees except $1$ for this case.
    \end{enumerate}\par
   Therefore, $1.3247\cdots$ is a lower bound of the first dynamical degrees except $1$.\par
   It remains to show that this is the minimum value.\par
   Denote the real root of $x^3-x^2+1\in\mathbb{Z}[x]$ by $\lambda$ and the other complex roots by $z,\overline{z}$.
   Define $E:=\mathbb{Q}(\lambda,z)$, $F:=\mathbb{Q}(\sqrt{-23})$, and $K:=\mathbb{Q}(\lambda)$.
   \begin{lemma}
    $E$ contains $F$ as a subfield.
    Moreover, $E/F$ is a Galois extension of degree $3$ and it is cyclic. 
   \end{lemma}
   \begin{proof}
    $E\supset K\supset\mathbb{Q}$ is a sequence of field extensions and it is clear that $E/\mathbb{Q}$ is a Galois extension, $[K:\mathbb{Q}]=3$ and $[E:K]\leq2$.
    $E\not\subset\mathbb{R}$ and $K\subset\mathbb{R}$ implies $E\neq K$ and so $[E:K]=2$ and $[E:\mathbb{Q}]=6$.\par
    The equation
    \begin{align*}
     x^3-x^2+1=(x-\lambda)(x^2+(\lambda-1)x+\lambda^2-\lambda)
    \end{align*}
    implies
    \begin{align*}
     z, \overline{z}=\frac{1-\lambda\pm\sqrt{-3\lambda^2+2\lambda+1}}{2}
    \end{align*}
    and by
    \begin{align*}
     (3\lambda^2-2\lambda)\sqrt{-3\lambda^2+2\lambda+1}&=\sqrt{(3\lambda^2-2\lambda)^2(-3\lambda^2+2\lambda+1)}\\
                                                       &=\sqrt{-23},
    \end{align*}
    $\mathbb{Q}(\lambda,\sqrt{-23})=\mathbb{Q}(\lambda,z)=E$ and so $E\supset F$.\par
    Now $[F:\mathbb{Q}]=2$ and so $E/F$ is a Galois extension of degree $3$.
    The Galois group $\mathrm{Gal}(E/F)$ is a group of order $3$ and so it is cyclic.
   \end{proof}
   Define
    \begin{align*}
     \begin{array}{cccc}
      \sigma\colon & E & \rightarrow & E\\
      & \lambda & \mapsto & z \\
      & z & \mapsto & \overline{z}\\
      & \overline{z} & \mapsto & \lambda
     \end{array}
    \end{align*}
   over $\mathbb{Q}$ and then $\sigma$ is a generator of $\mathrm{Gal}(E/F)$.
   \begin{lemma}
    There exists $u\in F^\times$ such that the order of $u$ in $F^\times/N_{E/F}(E^\times)$ is $3$.
   \end{lemma}
   \begin{proof}
    First consider the prime factorization of the prime ideal $3\mathbb{Z}$ in $E$ and $F$.\par
    For the field extension $F/\mathbb{Q}$ and the primitive element $\theta=\sqrt{-23}$, the conductor $\mathfrak{J}$ of $\mathbb{Z}[\sqrt{-23}]$ in $\mathcal{O}_F=\mathbb{Z}\left[\frac{1+\sqrt{-23}}{2}\right]$ contains $2$ and so $3\mathbb{Z}$ and $\mathfrak{J}$ are relatively prime.
    By using Remark \ref{construction of prime ideal factorization}, there is a prime ideal factorization
    \begin{align*}
     3\mathcal{O}_F=\mathfrak{p}_1\mathfrak{p}_2\quad(\mathfrak{p}_1,\ \mathfrak{p}_2\subset\mathcal{O}_F\text{ are prime ideals}),
    \end{align*}
    which corresponds to
    \begin{align*}
     X^2+23=(X+1)(X+2)\in(\mathbb{Z}/3\mathbb{Z})[X].
    \end{align*}
    Now $F/\mathbb{Q}$ is a Galois extension and so $[\mathcal{O}_F/\mathfrak{p}_1:\mathbb{Z}/3\mathbb{Z}]=1$. 
    For the field extension $E/F$ and the primitive element $\theta'=\lambda$, by using Remerk \ref{conductor and discriminant}, the conductor $\mathfrak{J}'$ of $\mathcal{O}_F[\lambda]$ in $\mathcal{O}_E$ contains
    \begin{align*}
     \delta=\Delta_{E/F}(1,\lambda,\lambda^2)=\mathrm{det}\begin{pmatrix}
                                               1      & \lambda      & \lambda^2 \\
                                               1      & z            & z^2 \\
                                               1      & \overline{z} & \overline{z}^2
                                              \end{pmatrix}^2\\
                                             =(\lambda-z)^2(\lambda-\overline{z})^2(z-\overline{z})^2
                                             =-23
    \end{align*}
    by using the equation of Vandermonde's determinant and the discriminant of a polynomial for $x^3-x^2+1\in\mathbb{Z}[x]$.
    This implies $\mathfrak{p}_1$ is relatively prime to $\mathfrak{J}'$ and so by using Remark \ref{construction of prime ideal factorization},
    \begin{align*}
     \mathfrak{p}_1\mathcal{O}_E=\mathfrak{q}\quad(\mathfrak{q}\subset\mathcal{O}_F\text{ is a prime ideal}),
    \end{align*}
    which corresponds to
    \begin{align*}
     X^3-X^2+1\text{ is irreducible in }(\mathcal{O}_F/\mathfrak{p}_1)[X]=(\mathbb{Z}/3\mathbb{Z})[X].
    \end{align*}
    The pair $(p=3,\mathfrak{p}=\mathfrak{p}_1,\mathfrak{q})$ satisfies the following conditions.
    \begin{itemize}
     \item $\mathfrak{p}$ is over $p\mathbb{Z}$ and $\mathfrak{q}$ is over $\mathfrak{p}$.
     \item $\sigma(\mathfrak{q})=\mathfrak{q}$ for all $\sigma\in \mathrm{Gal}(E/F)$.
     \item the multiplicity of $\mathfrak{q}$ of the prime ideal factorization of the ideal $p\mathcal{O}_E$ is not a multiple of $3$.
    \end{itemize}
    Thus, by Lemma \ref{multiplicity}, there is no $\alpha\in E^\times\setminus U_E$ such that $N_{E/F}(\alpha)=p$.
    Moreover, since $p$ is not an unit of $\mathbb{Z}$, there is no $\alpha\in U_E$ such that $N_{E/F}(\alpha)=p$.
    Thus, $p\notin N_{E/F}(E^\times)$ and the order of $p$ in $F^\times/N_{E/F}(E^\times)$ is $3$.
   \end{proof}
   \begin{remark}\label{division algebra}
    Take a symbol $v$ with $v^3=u$ and define $B=\oplus_{i=0}^{2} v^{i}E$ as in Theorem \ref{construct division algebra}, and then $B$ is a division ring and also a central simple algebra over $F$ with $[B:F]=9$.\par
    Also, the order $\mathcal{O}:=\oplus_{i=0}^{2} v^{i}\mathcal{O}_E\subset B$ contains both $\lambda$ and $\frac{1}{\lambda}$.
   \end{remark}
   \begin{theorem}
    For the division algebra $B$ in Remark \ref{division algebra}, there is a positive anti-involution of the second kind on $B$ over $\mathbb{Q}$.
   \end{theorem}
   \begin{proof}
    Define
    \begin{align*}
     \begin{array}{ccccc}
      \phi\colon & B & \rightarrow & B & \\
      & v^i\cdot a & \mapsto & v^i\cdot(\sigma^i\circ\tau)(a) & (a\in E)
     \end{array}
    \end{align*}
    over $\mathbb{Q}$, where $\tau$ is the complex conjugation map.\par
    By using the equation $\sigma\circ\tau\circ\sigma=\tau$,
    \begin{align*}
     &\phi(1)=1,\\
     &\phi(\phi(v^i\cdot a))=\phi(v^i\cdot\sigma^i\circ\tau(a))=v^i\cdot(\sigma^i\circ\tau\circ\sigma^i\circ\tau)(a)=v^i\cdot a,
    \end{align*}
    \begin{align*}
     \phi(v^i\cdot a\cdot v^{i'}\cdot a')=\phi(v^i v^{i'}\cdot\sigma^{i'}(a)a')&=\phi(v^{i+i'}\cdot\sigma^{i'}(a)a')\\
                                                                               &=v^{i+i'}\cdot(\sigma^{i+i'}\circ\tau)(\sigma^{i'}(a)a')\\
                                                                               &=v^{i+i'}\cdot(\sigma^i\circ\tau)(a)\cdot(\sigma^{i+i'}\circ\tau)(a')
    \end{align*}
    and
    \begin{align*}
     \phi(v^{i'}\cdot a')\phi(v^i\cdot a)&=v^{i'}\cdot(\sigma^{i'}\circ\tau)(a')\cdot v^i\cdot(\sigma^i\circ\tau)(a)\\
                                         &=v^{i'} v^i\cdot(\sigma^{i+i'}\circ\tau)(a')\cdot(\sigma^i\circ\tau)(a)\\
                                         &=u^{i+i'}\cdot(\sigma^i\circ\tau)(a)\cdot(\sigma^{i+i'}\circ\tau)(a')
    \end{align*}
    hold and so $\phi$ is the anti-involution map over $\mathbb{Q}$.
    Also, the restriction map $\phi\lvert_F$ is the complex conjugation map and so $\phi$ is of the second kind.
    Thus, by composing with Theorem \ref{existence of positive anti-involution}, there exists a positive anti-involution map over $\mathbb{Q}$ of the second kind.
   \end{proof}
   By Proposition \ref{construction for Type 4}, there is a $9$-dimensional simple abelian variety with an automorphism corresponding to $\lambda$ for the case $d=3$, $e=2$, $e_0=1$ and $m=1$.\par
   The first dynamical degree of the automorphism is the square of the maximal absolute value of the conjugates of $\lambda$ and it is $1.1509\cdots^2=1.3247\cdots$.\\
   \\
  Therefore, the minimum value of the first dynamical degrees of automorphisms of $9$-dimensional simple abelian varieties is $1.3247\cdots$ except $1$, and realized on a simple abelian variety of Type 4.
  \begin{remark}
   By some deduction, we can check that $1.3247\cdots$ is the real root of $x^3-x-1\in\mathbb{Z}[x]$ and it is the smallest Pisot number (cf.\ \cite[Theorem 7.2.1]{BDGPS92}).
  \end{remark}

 \subsection{dimension $g=10$}
  Let $X$ be a $10$-dimensional simple abelian variety and define $B$, $K$, $K_0$, $d$, $e$ and $e_0$ as in Section \ref{Endomorphism algebras of simple abelian varieties}.
  \begin{flushleft}{\bf{Type 1}}\end{flushleft}\par
   Let $X$ be a $10$-dimensional simple abelian variety of Type 1 and $f\in\mathrm{End}(X)$ be an automorphism of $X$ which corresponds to $\alpha\in U_K$ as in Section \ref{The minimum value of the first dynamical degrees of automorphisms}.
   Now $[K:\mathbb{Q}]=1$ or $2$ or $5$ or $10$ by the restriction in Table \ref{table1}.\par
   Assuming that the first dynamical degree is less than $4$, by using similar deduction as in the $4$-dimensional case, the minimum value of the maximal absolute value of the roots of $P(x)$ is $2\mathrm{cos}\left(\frac{\pi}{5}\right)$ except $1$.
   The realizability is concerned in Section \ref{The minimum value of the first dynamical degrees of automorphisms}.
   Thus, the minimum value of the first dynamical degrees is $4\mathrm{cos}^2\left(\frac{\pi}{5}\right)=\left(\frac{1+\sqrt{5}}{2}\right)^2$ except $1$ for this type.
  \begin{flushleft}{\bf{Type 2, Type 3}}\end{flushleft}\par
   Let $X$ be a $10$-dimensional simple abelian variety of Type 2 or 3 and $f\in\mathrm{End}(X)$ be an automorphism of $X$ which corresponds to $\alpha\in B$.
   Now $[K:\mathbb{Q}]=1$ or $5$ by the restriction in Table \ref{table1}.\par
   Denote the minimal polynomial of $\alpha$ over $K$ by $p(x)$ and the minimal polynomial of $\alpha$ over $\mathbb{Q}$ by $P(x)$.
   If the degree of $p(x)$ is $1$, then $\alpha\in U_K$ and so it is reduced to Type 1 of $5$-dimensional case.
   Thus, $4\mathrm{cos}^2\left(\frac{\pi}{11}\right)$ is a lower bound except $1$ for this case.\par
   If the degree of $p(x)$ is $2$, it can be written as $p(x)=x^2+ax+b\in\mathcal{O}_K[x](b\in U_K)$ and consider the following 3 cases as in Section \ref{The minimum value of the first dynamical degrees of automorphisms}.
   Let $\{\sigma_i\}$ be the set of all $\mathbb{Q}$-embeddings $K\hookrightarrow\mathbb{C}$.\\
   $\bullet$ there exists $1\leq i\leq e$ such that $\abs{\sigma_i(b)}>1$\par
    The maximal absolute value of all conjugates of $b\in U_K$ is at least $2\mathrm{cos}\left(\frac{\pi}{11}\right)$ and so the maximal absolute value of the roots of $P(x)$ is at least $\sqrt{2\mathrm{cos}\left(\frac{\pi}{11}\right)}$.
    Thus, $2\mathrm{cos}\left(\frac{\pi}{11}\right)=1.9189\cdots$ is a lower bound except $1$ for this case.\\
   $\bullet$ $\sigma_i(b)=1$ for all $1\leq i\leq e$\par
    $\mathrm{max}\{\abs{\sigma_i(a)}\}$ must be larger than $2$ and close to $2$.
    By using a computer algebra, totally real algebraic integers of degree $5$ whose conjugates have modulus less than $2.1$ are enumerated in Table \ref{appendix 17}.
    $a=2.0264\cdots$ provides the minimum of the maximal absolute value.
    For this case, the maximal absolute value of the roots of $P(x)$ is the Lehmer number $1.1762\cdots$ and so $1.1762\cdots^2=1.3836\cdots$ is a lower bound except $1$ for this case.\\
   $\bullet$ $\sigma_i(b)=-1$ for all $1\leq i\leq e$\par
    $a=1$ provides the minimum value of the first dynamical degrees as in Section \ref{The minimum value of the first dynamical degrees of automorphisms}.
    Thus, $\left(\frac{1+\sqrt{5}}{2}\right)^2$ is a lower bound except $1$ for this case.\par
   Thus, $1.1762\cdots^2=1.3836\cdots$ is a lower bound of the first dynamical degrees except $1$ for these types.
   The realizability of $1.1762\cdots^2=1.3836\cdots$ as the first dynamical degree is due to \cite[Theorem 2.1 (i)]{DH22}.
  \begin{flushleft}{\bf{Type 4}}\end{flushleft}\par
   Let $X$ be a $10$-dimensional simple abelian variety and $f\in\mathrm{End}(X)$ be an automorphism of $X$ which corresponds to $\alpha\in B$.\par
   Now $d=1$ and $[K:\mathbb{Q}]=2$ or $4$ or $10$ or $20$ by the restriction in Table \ref{table1} and $\alpha\in U_K$.\par
   Now $[K_0:\mathbb{Q}]=1$ or $2$ or $5$ or $10$, respectively.
   Denote the minimal polynomial of $\alpha$ over $\mathbb{Q}$ by $P(x)$.\par
   If $\alpha\in K_0$, the minimum value of the square of the maximal absolute value of the roots of $P(x)$ is $4\mathrm{cos}^2\left(\frac{\pi}{5}\right)$ except $1$, by the deduction in Type 1.\par
   If $\alpha\in K\setminus K_0$, by denoting $a=\abs{\alpha}^2\in U_{K_0}$, as in Section \ref{The minimum value of the first dynamical degrees of automorphisms}, $\mathbb{Q}(\sqrt{a})$ is a totally real number field.
   By using similar deduction as in the $4$-dimensional case, the minimum value of the square of the maximal absolute value of the roots of $P(x)$ is $4\mathrm{cos}^2\left(\frac{\pi}{5}\right)$ except $1$.\\
   \\
  Therefore, the minimum value of the first dynamical degrees of automorphisms of $10$-dimensional simple abelian varieties is $1.3836\cdots$ except $1$, and realized on a simple abelian variety of Type 2.

\appendix
\setcounter{table}{0}
\renewcommand{\thetable}{\Alph{section}.\arabic{table}}
\section{Tables}\label{appendix a}
 The tables we use in this paper are lined up in this section.
 On creating from Table \ref{appendix 6} to Table \ref{appendix 17}, we use Python as a programming language.
 The program is available at \url{https://github.com/sugi0000/maximal-absolute-value}. 
\begin{table}[hbtp]
  \caption{The degree of $S_n(x)$ for small $n$}
  \label{appendix 1}
  \small
  \begin{tabular}{|c|c|c|}
   \hline
   degree of $P_n(x)$ & cyclotomic factor of $P_n(x)$ & degree of $S_n(x)$ \\
   \hline
   10  &             & 10\\
   \hline
   11  & $\Phi_2(x)$ & 10\\
   \hline
   12  & $\Phi_3(x)$ & 10\\
   \hline
   13  & $\Phi_2(x)$, $\Phi_8(x)$ & 8\\
   \hline
   14  & $\Phi_5(x)$ & 10\\
   \hline
   15  & $\Phi_2(x)$, $\Phi_3(x)$  & 12\\
   \hline
   16  &             & 16\\
   \hline
   17  & $\Phi_2(x)$ & 16\\
   \hline
   18  & $\Phi_3(x)$, $\Phi_{12}(x)$ & 12\\
   \hline
   19  & $\Phi_2(x)$, $\Phi_5(x)$ & 14\\
   \hline
   20  &             & 20\\
   \hline
   21  & $\Phi_2(x)$, $\Phi_3(x)$, $\Phi_8(x)$ & 14\\
   \hline
   22  &             & 22\\
   \hline
   23  & $\Phi_2(x)$ & 22\\
   \hline
   24  & $\Phi_3(x)$, $\Phi_5(x)$ & 18\\
   \hline
   25  & $\Phi_2(x)$, $\Phi_{18}(x)$ & 18\\
   \hline
   26  &             & 26\\
   \hline
   27  & $\Phi_2(x)$, $\Phi_3(x)$ & 24\\
   \hline
   28  &             & 28\\
   \hline
   29  & $\Phi_2(x)$, $\Phi_5(x)$, $\Phi_8(x)$ & 20\\
   \hline
   30  & $\Phi_3(x)$, $\Phi_{12}(x)$ & 24\\
   \hline
   31  & $\Phi_2(x)$ & 30\\
   \hline
   32  &             & 32\\
   \hline
   33  & $\Phi_2(x)$, $\Phi_3(x)$ & 30\\
   \hline
  \end{tabular}
 \end{table}

\clearpage
\begin{table}[hbtp]
  \caption{The degree of $S'_m(x)$ for small $m$}
  \label{appendix 2}
  \small
  \begin{tabular}{|c|c|c|}
   \hline
   degree of $Q_m(x)$ & cyclotomic factor of $Q_m(x)$ & degree of $S'_m(x)$ \\
   \hline
   4  &             & 4\\
   \hline
   5  & $\Phi_2(x)$ & 4\\
   \hline
   6  &             & 6\\
   \hline
   7  & $\Phi_2(x)$ & 6\\
   \hline
   8  &             & 8\\
   \hline
   9  & $\Phi_2(x)$ & 8\\
   \hline
   10  & $\Phi_8(x)$ & 6\\
   \hline
   11  & $\Phi_2(x)$ & 10\\
   \hline
   12  &             & 12\\
   \hline
   13  & $\Phi_2(x)$, $\Phi_{12}(x)$ & 8\\
   \hline
   14  &             & 14\\
   \hline
   15  & $\Phi_2(x)$ & 14\\
   \hline
   16  &             & 16\\
   \hline
   17  & $\Phi_2(x)$, $\Phi_{18}(x)$ & 10\\
   \hline
   18  & $\Phi_8(x)$ & 14\\
   \hline
   19  & $\Phi_2(x)$ & 18\\
   \hline
   20  &             & 20\\
   \hline
   21  & $\Phi_2(x)$ & 20\\
   \hline
   22  &             & 22\\
   \hline
   23  & $\Phi_2(x)$ & 22\\
   \hline
   24  & $\Phi_{30}(x)$ & 16\\
   \hline
   25  & $\Phi_2(x)$, $\Phi_{12}(x)$ & 20\\
   \hline
   26  & $\Phi_8(x)$ & 22\\
   \hline
   27  & $\Phi_2(x)$ & 26\\
   \hline
   28  &             & 28\\
   \hline
   29  & $\Phi_2(x)$ & 28\\
   \hline
   30  &             & 30\\
   \hline
   31  & $\Phi_2(x)$ & 30\\
   \hline
  \end{tabular}
 \end{table}

\clearpage
\begin{table}[hbtp]
  \caption{The degree of $S_n(x)$}
  \label{appendix 3}
  \tiny
  \setlength{\tabcolsep}{3pt}
  \begin{tabular}{|c|c|c|c|c|c|c|c|c|c|}
   \hline
   \begin{tabular}{c}degree of\\ $P_n(x)$\end{tabular} & \begin{tabular}{c}degree of\\ $S_n(x)$\end{tabular} & \begin{tabular}{c}degree of\\ $P_n(x)$\end{tabular} & \begin{tabular}{c}degree of\\ $S_n(x)$\end{tabular} & \begin{tabular}{c}degree of\\ $P_n(x)$\end{tabular} & \begin{tabular}{c}degree of\\ $S_n(x)$\end{tabular} & \begin{tabular}{c}degree of\\ $P_n(x)$\end{tabular} & \begin{tabular}{c}degree of\\ $S_n(x)$\end{tabular} & \begin{tabular}{c}degree of\\ $P_n(x)$\end{tabular} & \begin{tabular}{c}degree of\\ $S_n(x)$\end{tabular}\\
   \hline
   11  & 10 & 90 & 84 & 168 & 166 & 247 & 246 & 327 & 324 \\
   \hline
   12  & 10 & 91 & 90 & 171 & 168 & 248 & 240 & 328 & 328\\
   \hline
   13  & 8 & 92 & 92 & 172 & 172 & 251 & 250 & 330 & 324\\
   \hline
   15  & 12 & 93 & 86 & 173 & 168 & 252 & 250 & 331 & 324\\
   \hline
   16  & 16 & 95 & 94 & 174 & 164 & 253 & 248 & 332 & 332\\
   \hline
   18  & 12 & 96 & 94 & 175 & 174 & 255 & 252 & 333 & 326\\
   \hline
   19  & 14 & 99 & 92 & 176 & 176 & 256 & 256 & 335 & 334\\
   \hline
   20  & 20 & 100 & 100 & 179 & 174 & 258 & 252 & 336 & 334\\
   \hline
   21  & 14 & 101 & 96 & 180 & 178 & 259 & 248 & 339 & 332\\
   \hline
   23  & 22 & 102 & 96 & 181 & 176 & 260 & 260 & 340 & 340\\
   \hline
   24  & 18 & 103 & 102 & 183 & 180 & 261 & 254 & 341 & 336\\
   \hline
   27  & 24 & 104 & 100 & 184 & 180 & 263 & 262 & 342 & 336\\
   \hline
   28  & 28 & 107 & 106 & 186 & 180 & 264 & 258 & 343 & 342\\
   \hline
   29  & 20 & 108 & 106 & 187 & 180 & 267 & 264 & 344 & 340\\
   \hline
   30  & 24 & 109 & 100 & 188 & 180 & 268 & 268 & 347 & 346\\
   \hline
   31  & 30 & 111 & 108 & 189 & 178 & 269 & 260 & 348 & 346\\
   \hline
   32  & 32 & 112 & 112 & 191 & 190 & 270 & 264 & 349 & 334\\
   \hline
   35  & 34 & 114 & 104 & 192 & 190 & 271 & 270 & 351 & 348\\
   \hline
   36  & 34 & 115 & 108 & 195 & 192 & 272 & 272 & 352 & 352\\
   \hline
   37  & 32 & 116 & 116 & 196 & 196 & 275 & 274 & 354 & 344\\
   \hline
   39  & 32 & 117 & 110 & 197 & 192 & 276 & 274 & 355 & 354\\
   \hline
   40  & 40 & 119 & 114 & 198 & 192 & 277 & 266 & 356 & 356\\
   \hline
   42  & 36 & 120 & 118 & 199 & 194 & 279 & 272 & 357 & 350\\
   \hline
   43  & 36 & 123 & 120 & 200 & 200 & 280 & 280 & 359 & 354\\
   \hline
   44  & 40 & 124 & 120 & 203 & 202 & 282 & 276 & 360 & 358\\
   \hline
   45  & 38 & 125 & 120 & 204 & 198 & 283 & 282 & 363 & 360\\
   \hline
   47  & 46 & 126 & 120 & 205 & 194 & 284 & 280 & 364 & 360\\
   \hline
   48  & 46 & 127 & 126 & 207 & 204 & 285 & 278 & 365 & 360\\
   \hline
   51  & 48 & 128 & 120 & 208 & 208 & 287 & 286 & 366 & 360\\
   \hline
   52  & 52 & 131 & 130 & 210 & 204 & 288 & 286 & 367 & 360\\
   \hline
   53  & 48 & 132 & 130 & 211 & 210 & 291 & 288 & 368 & 360\\
   \hline
   54  & 44 & 133 & 122 & 212 & 212 & 292 & 292 & 371 & 370\\
   \hline
   55  & 54 & 135 & 132 & 213 & 206 & 293 & 288 & 372 & 370\\
   \hline
   56  & 56 & 136 & 136 & 215 & 214 & 294 & 284 & 373 & 368\\
   \hline
   59  & 54 & 138 & 132 & 216 & 214 & 295 & 288 &  & \\
   \hline
   60  & 58 & 139 & 134 & 219 & 212 & 296 & 296 &  & \\
   \hline
   61  & 50 & 140 & 140 & 220 & 220 & 299 & 294 &  & \\
   \hline
   63  & 60 & 141 & 134 & 221 & 216 & 300 & 298 &  & \\
   \hline
   64  & 60 & 143 & 142 & 222 & 216 & 301 & 296 &  & \\
   \hline
   66  & 60 & 144 & 138 & 223 & 216 & 303 & 300 &  & \\
   \hline
   67  & 66 & 147 & 144 & 224 & 220 & 304 & 300 &  & \\
   \hline
   68  & 60 & 148 & 148 & 227 & 226 & 306 & 300 &  & \\
   \hline
   69  & 58 & 149 & 140 & 228 & 226 & 307 & 306 &  & \\
   \hline
   71  & 70 & 150 & 144 & 229 & 220 & 308 & 300 &  & \\
   \hline
   72  & 70 & 151 & 144 & 231 & 228 & 309 & 298 &  & \\
   \hline
   75  & 72 & 152 & 152 & 232 & 232 & 311 & 310 &  & \\
   \hline
   76  & 76 & 155 & 154 & 234 & 224 & 312 & 310 &  & \\
   \hline
   77  & 72 & 156 & 154 & 235 & 234 & 315 & 312 &  & \\
   \hline
   78  & 72 & 157 & 152 & 236 & 236 & 316 & 316 &  & \\
   \hline
   79  & 68 & 159 & 152 & 237 & 230 & 317 & 312 &  & \\
   \hline
   80  & 80 & 160 & 160 & 239 & 234 & 318 & 312 &  & \\
   \hline
   83  & 82 & 162 & 156 & 240 & 238 & 319 & 314 &  & \\
   \hline
   84  & 78 & 163 & 162 & 243 & 240 & 320 & 320 &  & \\
   \hline
   85  & 80 & 164 & 160 & 244 & 240 & 323 & 322 &  & \\
   \hline
   87  & 84 & 165 & 158 & 245 & 240 & 324 & 318 &  & \\
   \hline
   88  & 88 & 167 & 166 & 246 & 240 & 325 & 320 &  & \\
   \hline
  \end{tabular}
 \end{table}

\begin{table}[hbtp]
  \caption{The degree of $S'_m(x)$}
  \label{appendix 4}
  \tiny
  \setlength{\tabcolsep}{3pt}
  \begin{tabular}{|c|c|c|c|c|c|c|c|c|c|}
   \hline
   \begin{tabular}{c}degree of\\ $Q_m(x)$\end{tabular} & \begin{tabular}{c}degree of\\ $S'_m(x)$\end{tabular} & \begin{tabular}{c}degree of\\ $Q_m(x)$\end{tabular} & \begin{tabular}{c}degree of\\ $S'_m(x)$\end{tabular} & \begin{tabular}{c}degree of\\ $Q_m(x)$\end{tabular} & \begin{tabular}{c}degree of\\ $S'_m(x)$\end{tabular} & \begin{tabular}{c}degree of\\ $Q_m(x)$\end{tabular} & \begin{tabular}{c}degree of\\ $S'_m(x)$\end{tabular} &\begin{tabular}{c}degree of\\ $Q_m(x)$\end{tabular} & \begin{tabular}{c}degree of\\ $S'_m(x)$\end{tabular} \\
   \hline
   7  & 6 & 95 & 94 & 181 & 176 & 267 & 266 & 355 & 354\\
   \hline
   10  & 6 & 97 & 92 & 183 & 182 & 271 & 270 & 359 & 352\\
   \hline
   11  & 10 & 98 & 94 & 186 & 182 & 274 & 270 & 361 & 356\\
   \hline
   13  & 8 & 99 & 98 & 187 & 186 & 275 & 274 & 362 & 358\\
   \hline
   15  & 14 & 103 & 102 & 191 & 190 & 277 & 272 & 363 & 362\\
   \hline
   18  & 14 & 106 & 102 & 193 & 188 & 279 & 278 & 367 & 366\\
   \hline
   19  & 18 & 107 & 100 & 194 & 190 & 282 & 278 & 370 & 366\\
   \hline
   23  & 22 & 109 & 104 & 195 & 194 & 283 & 282 & 371 & 370\\
   \hline
   25  & 20 & 111 & 110 & 199 & 198 & 287 & 280 & 373 & 368\\
   \hline
   26  & 22 & 114 & 102 & 202 & 198 & 289 & 284 & 375 & 374\\
   \hline
   27  & 26 & 115 & 114 & 203 & 202 & 290 & 286 & 378 & 374\\
   \hline
   31  & 30 & 119 & 118 & 205 & 200 & 291 & 290 &  & \\
   \hline
   34  & 30 & 121 & 116 & 207 & 206 & 295 & 294 &  & \\
   \hline
   35  & 28 & 122 & 118 & 210 & 206 & 298 & 294 &  & \\
   \hline
   37  & 32 & 123 & 122 & 211 & 210 & 299 & 298 &  & \\
   \hline
   39  & 38 & 127 & 126 & 215 & 208 & 301 & 296 &  & \\
   \hline
   42  & 38 & 130 & 126 & 217 & 212 & 303 & 302 &  & \\
   \hline
   43  & 42 & 131 & 130 & 218 & 214 & 306 & 302 &  & \\
   \hline
   47  & 46 & 133 & 128 & 219 & 218 & 307 & 306 &  & \\
   \hline
   49  & 44 & 135 & 134 & 223 & 222 & 311 & 310 &  & \\
   \hline
   50  & 46 & 138 & 134 & 226 & 222 & 313 & 308 &  & \\
   \hline
   51  & 50 & 139 & 138 & 227 & 226 & 314 & 310 &  & \\
   \hline
   55  & 54 & 143 & 136 & 229 & 224 & 315 & 314 &  & \\
   \hline
   58  & 54 & 145 & 140 & 231 & 230 & 319 & 318 &  & \\
   \hline
   59  & 58 & 146 & 142 & 234 & 222 & 322 & 318 &  & \\
   \hline
   61  & 56 & 147 & 146 & 235 & 234 & 323 & 316 &  & \\
   \hline
   63  & 62 & 151 & 150 & 239 & 238 & 325 & 320 &  & \\
   \hline
   66  & 62 & 154 & 150 & 241 & 236 & 327 & 326 &  & \\
   \hline
   67  & 66 & 155 & 154 & 242 & 238 & 330 & 326 &  & \\
   \hline
   71  & 64 & 157 & 152 & 243 & 242 & 331 & 330 &  & \\
   \hline
   73  & 68 & 159 & 158 & 247 & 246 & 335 & 334 &  & \\
   \hline
   74  & 70 & 162 & 158 & 250 & 246 & 337 & 332 &  & \\
   \hline
   75  & 74 & 163 & 162 & 251 & 244 & 338 & 334 &  & \\
   \hline
   79  & 78 & 167 & 166 & 253 & 248 & 339 & 338 &  & \\
   \hline
   82  & 78 & 169 & 164 & 255 & 254 & 343 & 342 &  & \\
   \hline
   83  & 82 & 170 & 166 & 258 & 254 & 346 & 342 &  & \\
   \hline
   85  & 80 & 171 & 170 & 259 & 258 & 347 & 346 &  & \\
   \hline
   87  & 86 & 175 & 174 & 263 & 262 & 349 & 344 &  & \\
   \hline
   90  & 86 & 178 & 174 & 265 & 260 & 351 & 350 &  & \\
   \hline
   91  & 90 & 179 & 172 & 266 & 262 & 354 & 342 &  & \\
   \hline
  \end{tabular}
 \end{table}

\clearpage
\begin{table}[hbtp]
   \begin{threeparttable}[h]
   \caption{Calculation of the maximal absolute value of the roots for the case of Type 4 of $4$-dimenisonal}
   \label{appendix 5}
   \footnotesize
   \begin{tabular}{|c|c|c|c|c|}
    \hline
    $D$ & $(p,q)$ & $t$ & $(x^2+sx+t)(x^2+\overline{s}x+\overline{t})$ & maximal absolute value \\
    \hline\hline
    $7$ & $(-2,1)$   & 1  & $x^4-3x^3+6x^2-3x+1$ & $2.1022\cdots$\\
    $7$ & $(-2,1)$   & -1 & $x^4-3x^3+2x^2+3x+1$ & $2.1889\cdots$\\
    $7$ & $(-1,1)$   & 1  & $x^4-x^3+4x^2-x+1$   & $1.8832\cdots$\\
    $7$ & $(-1,1)$   & -1 & $x^4-x^3+x+1$        & $1.3722\cdots$\\
    $7$ & $(0,1)$    & 1  & $x^4+x^3+4x^2+x+1$   & $1.8832\cdots$\\
    $7$ & $(0,1)$    & -1 & $x^4+x^3-x+1$        & $1.3722\cdots$\\
    $7$ & $(1,1)$    & 1  & $x^4+3x^3+6x^2+3x+1$ & $2.1022\cdots$\\
    $7$ & $(1,1)$    & -1 & $x^4+3x^3+2x^2-3x+1$ & $2.1889\cdots$\\
    $11$ & $(-1,1)$  & 1  & $x^4-x^3+5x^2-x+1$   & $2.1537\cdots$\\
    $11$ & $(-1,1)$  & -1 & $x^4-x^3+x^2+x+1$    & $1.4675\cdots$\\
    $11$ & $(0,1)$   & 1  & $x^4+x^3+5x^2+x+1$   & $2.1537\cdots$\\
    $11$ & $(0,1)$   & -1 & $x^4+x^3+x^2-x+1$    & $1.4675\cdots$\\
    $15$ & $(-1,1)$  & 1  & $x^4-x^3+6x^2-x+1$   & $2.3869\cdots$\\
    $15$ & $(-1,1)$  & -1 & $x^4-x^3+2x^2+x+1$   & $1.6180\cdots$\\
    $15$ & $(0,1)$   & 1  & $x^4+x^3+6x^2+x+1$   & $2.3869\cdots$\\
    $15$ & $(0,1)$   & -1 & $x^4+x^3+2x^2-x+1$   & $1.6180\cdots$\\
    $2$ & $(-1,1)$   & 1  & $x^4-2x^3+5x^2-2x+1$ & $2.0322\cdots$\\
    $2$ & $(-1,1)$   & -1 & $x^4-2x^3+x^2+2x+1$  & $1.8039\cdots$\\
    $2$ & $(0,1)$    & 1  & $x^4+4x^2+1$         & $1.9318\cdots$\\
    $2$ & $(0,1)$    & -1 & $x^4+1$              & $1$\\
    $2$ & $(1,1)$    & 1  & $x^4+2x^3+5x^2+2x+1$ & $2.0322\cdots$\\
    $2$ & $(-1,1)$   & -1 & $x^4+2x^3+x^2-2x+1$  & $1.8039\cdots$\\
    $3$ & $(1,1)$  & $1$            & $x^4+3x^3+5x^2+3x+1$ & $1.7220\cdots$\\
    $3$ & $(1,1)$  & $\omega$\tnote{*1}     & $x^4+3x^3+4x^2+3x+1$ & $1$\\
    $3$ & $(1,1)$  & $\omega^2$     & $x^4+3x^3+2x^2+1$    & $1.7220\cdots$\\
    $3$ & $(1,1)$  & $\omega^3=-1$  & $x^4+3x^3+x^2-3x+1$  & $2.0758\cdots$\\
    $3$ & $(1,1)$  & $\omega^4$     & $x^4+3x^3+2x^2-3x+1$ & $2.1889\cdots$\\
    $3$ & $(1,1)$  & $\omega^5$     & $x^4+3x^3+4x^2+1$    & $2.0758\cdots$\\
    $3$ & $(1,0)$  & $\omega$       & $x^4+2x^3+2x^2+x+1$  & $1.3122\cdots$\\
    $3$ & $(1,0)$  & $\omega^2$     & $x^4+2x^3-x+1$       & $1.5392\cdots$\\
    $3$ & $(1,0)$  & $\omega^4$     & $x^4+2x^3-x+1$       & $1.5392\cdots$\\
    $3$ & $(1,0)$  & $\omega^5$     & $x^4+2x^3+2x^2+x+1$  & $1.3122\cdots$\\
    $3$ & $(2,0)$  & $\omega$       & $x^4+4x^3+5x^2+2x+1$ & $1.9318\cdots$\\
    $3$ & $(2,0)$  & $\omega^2$     & $x^4+4x^3+3x^2-2x+1$ & $2.2966\cdots$\\
    $3$ & $(2,0)$  & $\omega^4$     & $x^4+4x^3+3x^2-2x+1$ & $2.2966\cdots$\\
    $3$ & $(2,0)$  & $\omega^5$     & $x^4+4x^3+5x^2+2x+1$ & $1.9318\cdots$\\
    $1$ & $(1,1)$  & $1$            & $x^4+2x^3+4x^2+2x+1$ & $1.7000\cdots$\\
    $1$ & $(1,1)$  & $i$            & $x^4+2x^3+2x^2+2x+1$ & $1$\\
    $1$ & $(1,1)$  & $-1$           & $x^4+2x^3-2x+1$      & $1.7000\cdots$\\
    $1$ & $(1,1)$  & $-i$           & $x^4+2x^3+2x^2-2x+1$ & $1.9318\cdots$\\
    $1$ & $(1,0)$  & $i$            & $x^4+2x^3+x^2+1$     & $1.4425\cdots$\\
    $1$ & $(1,0)$  & $-i$           & $x^4+2x^3+x^2+1$      & $1.4425\cdots$\\
    $1$ & $(2,0)$  & $i$            & $x^4+4x^3+4x^2+1$    & $2.1474\cdots$\\
    $1$ & $(2,0)$  & $-i$           & $x^4+4x^3+4x^2+1$    & $2.1474\cdots$\\
   \hline
   \end{tabular}
  \begin{tablenotes}
  \item{*1} Denote $\omega=\frac{1+\sqrt{3}i}{2}$
  \end{tablenotes}
  \end{threeparttable}
  \end{table}

\begin{table}[hbtp]
   \caption{totally real, cubic algebraic integer, which is larger than $2$, whose conjugates have modulus$<2.25$}
   \label{appendix 6}
   \small
   \begin{tabular}{|c|c|}
    \hline
     algebraic integer & polynomial \\
    \hline\hline
    $2.1149\cdots$ & $x^3-4x-1$\\
    $2.1700\cdots$ & $x^3-x^2-3x+1$\\
    $2.1986\cdots$ & $x^3-x^2-4x+3$\\
    $2.2143\cdots$ & $x^3-4x-2$\\
    $2.2469\cdots$ & $x^3-2x^2-x+1$\\
   \hline
   \end{tabular}
  \end{table}

\begin{table}[hbtp]
   \caption{totally positive, septic algebraic unit, which is larger than $4$, whose conjugates are in $(0,4.04)$}
   \label{appendix 7}
   \small
   \begin{tabular}{|c|c|}
    \hline
     algebraic unit & polynomial \\
    \hline\hline
    $4.0333\cdots$ & $x^7-14x^6+77x^5-211x^4+301x^3-210x^2+56x-1$\\
    $4.0341\cdots$ & $x^7-14x^6+76x^5-200x^4+259x^3-146x^2+24x-1$\\
   \hline
   \end{tabular}
  \end{table}

\begin{table}[hbtp]
   \caption{totally real, quartic algebraic integer, which is larger than $2$, whose conjugates have modulus$<2.1$}
   \label{appendix 8}
   \small
   \begin{tabular}{|c|c|}
    \hline
     algebraic integer & minimal polynomial \\
    \hline\hline
    $2.0614\cdots$ & $x^4-4x^2-x+1$\\
    $2.0743\cdots$ & $x^4-5x^2+3$\\
    $2.0952\cdots$ & $x^4-x^3-3x^2+x+1$\\
   \hline
   \end{tabular}
  \end{table}

\begin{table}[hbtp]
   \caption{irreducible quartic polynomial, all roots are in $E$, one root is not inside $[-2,2]$, identifying $R(X)$ with $R(-X)$}
   \label{appendix 9}
   \footnotesize
   \begin{tabular}{|c|c|c|c|}
    \hline
    \begin{tabular}{c} condition\\ (\ref{equation1})\end{tabular} & polynomial $R(X)$ & polynomial $x^4R(x+\frac{1}{x})$ & \begin{tabular}{c}maximal\\ absolute value\end{tabular} \\
    \hline\hline
    N & $X^4-2X^3-2X^2+5X-1$ & \begin{tabular}{c}$x^8-2x^7+2x^6-x^5+x^4$\\ $-x^3+2x^2-2x+1$\end{tabular} & $1.2150\cdots$ \\
    N & $X^4-X^3-4X^2+2X+5$ & \begin{tabular}{c}$x^8-x^7-x^5+3x^4-x^3-x+1$\end{tabular} & $1.1837\cdots$ \\
    N & $X^4-X^3-4X^2+4X+2$ & \begin{tabular}{c}$x^8-x^7+x^5+x^3-x+1$\end{tabular} & $1.2744\cdots$ \\
    N & $X^4-X^3-3X^2+X+3$ & \begin{tabular}{c}$x^8-x^7+x^6-2x^5+3x^4$\\ $-2x^3+x^2-x+1$\end{tabular} & $1.2522\cdots$ \\
    N & $X^4-X^3-3X^2+3X-1$ & \begin{tabular}{c}$x^8-x^7+x^6-x^4+x^2-x+1$\end{tabular} & $1.1705\cdots$ \\
    N & $X^4-X^3-3X^2+3X+1$ & \begin{tabular}{c}$x^8-x^7+x^6+x^4+x^2-x+1$\end{tabular} & $1.2196\cdots$ \\
    N & $X^4-X^3-2X^2+1$ & \begin{tabular}{c}$x^8-x^7+2x^6-3x^5+3x^4$\\ $-3x^3+2x^2-x+1$\end{tabular} & $1.2408\cdots$ \\
    N & $X^4-X^3-2X^2+X+2$ & \begin{tabular}{c}$x^8-x^7+2x^6-2x^5+4x^4$\\ $-2x^3+2x^2-x+1$\end{tabular} & $1.2474\cdots$ \\
    N & $X^4-X^3-2X^2+2X-1$ & \begin{tabular}{c}$x^8-x^7+2x^6-x^5+x^4$\\ $-x^3+2x^2-x+1$\end{tabular} & $1.2722\cdots$ \\
    N & $X^4-5X^2+7$ & \begin{tabular}{c}$x^8-x^6+3x^4-x^2+1$\end{tabular} & $1.2406\cdots$ \\
    N & $X^4-4X^2-X+3$ & \begin{tabular}{c}$x^8-x^5+x^4-x^3+1$\end{tabular} & $1.1692\cdots$ \\
    Y & $X^4-4X^2-1$ & \begin{tabular}{c}$x^8-3x^4+1$\end{tabular} & $1.2720\cdots$ \\
    N & $X^4-4X^2+5$ & \begin{tabular}{c}$x^8+3x^4+1$\end{tabular} & $1.2720\cdots$ \\
    N & $X^4-3X^2+3$ & \begin{tabular}{c}$x^8+x^6+3x^4+x^2+1$\end{tabular} & $1.2406\cdots$ \\
   \hline
   \end{tabular}
  \end{table}

\clearpage
\begin{table}[hbtp]
   \caption{irreducible quartic polynomial, all roots are in $E_1$, one root is not written as $si$ with $s\in[-2,2]$, identifying $R(X)$ with $R(-X)$}
   \label{appendix 10}
   \footnotesize
   \begin{tabular}{|c|c|c|c|}
    \hline
    \begin{tabular}{c} condition\\ (\ref{equation1})\end{tabular} & polynomial $R(X)$ & polynomial $x^4R(x-\frac{1}{x})$ & \begin{tabular}{c}maximal\\ absolute value\end{tabular} \\
    \hline\hline
    N & $X^4-X^3+3X^2-2X+1$ & \begin{tabular}{c}$x^8-x^7-x^6+x^5+x^4$\\ $-x^3-x^2+x+1$\end{tabular} & $1.2245\cdots$ \\
    N & $X^4-X^3+4X^2-3X+1$ & \begin{tabular}{c}$x^8-x^7-x^4+x+1$\end{tabular} & $1.2306\cdots$ \\
    N & $X^4-X^3+4X^2-3X+2$ & \begin{tabular}{c}$x^8-x^7+x+1$\end{tabular} & $1.2612\cdots$ \\
    N & $X^4-X^3+4X^2-2X+2$ & \begin{tabular}{c}$x^8-x^7+x^5-x^3+x+1$\end{tabular} & $1.2397\cdots$ \\
    N & $X^4-X^3+4X^2-2X+3$ & \begin{tabular}{c}$x^8-x^7+x^5+x^4-x^3+x+1$\end{tabular} & $1.1837\cdots$ \\
    N & $X^4-X^3+5X^2-4X+3$ & \begin{tabular}{c}$x^8-x^7+x^6-x^5+x^4$\\ $+x^3+x^2+x+1$\end{tabular} & $1.2788\cdots$ \\
    N & $X^4-X^3+5X^2-3X+4$ & \begin{tabular}{c}$x^8-x^7+x^6+x^2+x+1$\end{tabular} & $1.2272\cdots$ \\
    N & $X^4-X^3+5X^2-3X+5$ & \begin{tabular}{c}$x^8-x^7+x^6+x^4+x^2+x+1$\end{tabular} & $1.2734\cdots$ \\
    N & $X^4-X^3+6X^2-3X+8$ & \begin{tabular}{c}$x^8-x^7+2x^6+2x^4+2x^2+x+1$\end{tabular} & $1.2474\cdots$ \\
    N & $X^4+2X^2-X+1$ & \begin{tabular}{c}$x^8-2x^6-x^5+3x^4+x^3-2x^2+1$\end{tabular} & $1.2553\cdots$ \\
    N & $X^4+3X^2-X+1$ & \begin{tabular}{c}$x^8-x^6-x^5+x^4+x^3-x^2+1$\end{tabular} & $1.1932\cdots$ \\
    N & $X^4+3X^2-X+2$ & \begin{tabular}{c}$x^8-x^6-x^5+2x^4+x^3-x^2+1$\end{tabular} & $1.2461\cdots$ \\
    N & $X^4+3X^2+3$ & \begin{tabular}{c}$x^8-x^6+3x^4-x^2+1$\end{tabular} & $1.2406\cdots$ \\
    N & $X^4+4X^2-X+1$ & \begin{tabular}{c}$x^8-x^5-x^4+x^3+1$\end{tabular} & $1.2512\cdots$ \\
    N & $X^4+4X^2-X+2$ & \begin{tabular}{c}$x^8-x^5+x^3+1$\end{tabular} & $1.2331\cdots$ \\
    N & $X^4+4X^2-X+3$ & \begin{tabular}{c}$x^8-x^5+x^4+x^3+1$\end{tabular} & $1.2450\cdots$ \\
    Y & $X^4+4X^2-1$ & \begin{tabular}{c}$x^8-3x^4+1$\end{tabular} & $1.2720\cdots$ \\
    N & $X^4+4X^2+5$ & \begin{tabular}{c}$x^8+3x^4+1$\end{tabular} & $1.2720\cdots$ \\
    N & $X^4+5X^2+7$ & \begin{tabular}{c}$x^8+x^6+3x^4+x^2+1$\end{tabular} & $1.2406\cdots$ \\
   \hline
   \end{tabular}
  \end{table}

\begin{table}[hbtp]
   \caption{quadratic polynomial of the form (\ref{equation2}), all roots are in $E_2$, identifying $R_1(z)$ with $R_1(-z)$}
   \label{appendix 11}
   \small
   \begin{tabular}{|c|c|c|c|}
    \hline
    polynomial $R_1(z)$ & polynomial $x^2R_1(x+\frac{i}{x})$ & \begin{tabular}{c}maximal\\ absolute value\end{tabular} \\
    \hline\hline
    $z^2-(2+2i)z+2i$ & \begin{tabular}{c}$x^4-(2+2i)x^3+4ix^2+(2-2i)x-1$\end{tabular} & $1$ \\
    $z^2-(1+i)z-i$ & \begin{tabular}{c}$x^4-(1+i)x^3+ix^2+(1-i)x-1$\end{tabular} & $1$ \\
    $z^2+(-1-i)z$ & \begin{tabular}{c}$x^4-(1+i)x^3+2ix^2+(1-i)x-1$\end{tabular} & $1$ \\
    $z^2-1-2i$ & \begin{tabular}{c}$x^4-x^2-1$\end{tabular} & $1.2720\cdots$ \\
    $z^2-4i$ & \begin{tabular}{c}$x^4-2ix^2-1$\end{tabular} & $1$ \\
    $z^2-3i$ & \begin{tabular}{c}$x^4-ix^2-1$\end{tabular} & $1$ \\
    $z^2-2i$ & \begin{tabular}{c}$x^4-1$\end{tabular} & $1$ \\
    $z^2-i$ & \begin{tabular}{c}$x^4+ix^2-1$\end{tabular} & $1$ \\
    $z^2$ & \begin{tabular}{c}$x^4+2ix^2-1$\end{tabular} & $1$ \\
    $z^2+1-2i$ & \begin{tabular}{c}$x^4+x^2-1$\end{tabular} & $1.2720\cdots$ \\
   \hline
   \end{tabular}
  \end{table}

\clearpage
\begin{table}[hbtp]
   \caption{quadratic polynomial of the form (\ref{equation3}), with the roots $s,t\in\mathbb{Q}(\zeta_8)$ satisfy $s\in E_3$, $t\in E_4$ and $(x^2+sx+\zeta_8)(x^2+tx+\zeta_8^3)\in\mathbb{Z}[\sqrt{2}i][x]$, identifying $R_2(z)$ with $R_2(-z)$}
   \label{appendix 12}
   \small
   \begin{tabular}{|c|c|c|c|}
    \hline
    polynomial $R_2(z)$ & \begin{tabular}{c}polynomial\\ $(x^2+sx+\zeta_8)(x^2+tx+\zeta_8^3)$\end{tabular} & \begin{tabular}{c}maximal\\ absolute value\end{tabular} \\
    \hline\hline
    $z^2-(2+\sqrt{2}i)z+\sqrt{2}i$ & \begin{tabular}{c}$x^4+(2+\sqrt{2}i)x^3+2\sqrt{2}ix^2+(-2+\sqrt{2}i)x-1$\end{tabular} & $1$ \\
    $z^2-\sqrt{2}iz-\sqrt{2}i$ & \begin{tabular}{c}$x^4+\sqrt{2}ix^3+\sqrt{2}ix-1$\end{tabular} & $1$ \\
    $z^2$ & \begin{tabular}{c}$x^4+\sqrt{2}ix^2-1$\end{tabular} & $1$ \\
   \hline
   \end{tabular}
  \end{table}

\begin{table}[hbtp]
   \caption{quadratic polynomial of the form (\ref{equation4}), with the roots $s,t\in\mathbb{Q}(\zeta_{12})$ satisfy $s\in E_5$ and $t\in E_6$ and $(x^2+sx+\zeta_{12})(x^2+tx+\zeta_{12}^5)\in\mathbb{Z}[i][x]$, identifying $R_3(z)$ with $R_3(-z)$}
   \label{appendix 13}
   \small
   \begin{tabular}{|c|c|c|c|}
    \hline
    polynomial $R_3(z)$ & \begin{tabular}{c}polynomial\\ $(x^2+sx+\zeta_{12})(x^2+tx+\zeta_{12}^5)$\end{tabular} & \begin{tabular}{c}maximal\\ absolute value\end{tabular} \\
    \hline\hline
    $z^2-(2+i)z+i$ & \begin{tabular}{c}$x^4+(2+i)x^3+2ix^2-(2-i)x-1$\end{tabular} & $1$ \\
    $z^2-(1+2i)z+i$ & \begin{tabular}{c}$x^4+(1+2i)x^3+2ix^2-(1-2i)x-1$\end{tabular} & $1$ \\
    $z^2-(1-i)z-2i$ & \begin{tabular}{c}$x^4+(1-i)x^3-ix^2-(1+i)x-1$\end{tabular} & $1$ \\
    $z^2$ & \begin{tabular}{c}$x^4+ix^2-1$\end{tabular} & $1$ \\
   \hline
   \end{tabular}
  \end{table}

\begin{table}[hbtp]
   \caption{cubic polynomial with constant term $1$, with coefficients in $\mathbb{Z}$, whose roots have modulus less than $1.3$, at least one root has modulus$>1$}
   \label{appendix 14}
   \small
   \begin{tabular}{|c|c|}
    \hline
    polynomial $p(x)$ & maximal absolute value \\
    \hline\hline
    $x^3-x^2+1$ & $1.1509\cdots$ \\
    $x^3+x+1$ & $1.2106\cdots$ \\
   \hline
   \end{tabular}
  \end{table}

\begin{table}[hbtp]
   \caption{cubic polynomial with constant term $1$, with coefficients in $\mathcal{O}_{\mathbb{Q}(\sqrt{-D})}$, whose roots have modulus less than $1.3$, at least one root has modulus$>1$, identifying $p(x)$ with $\overline{p(x)}$}
   \label{appendix 15}
   \small
   \begin{tabular}{|c|c|}
    \hline
    polynomial $p(x)$ & maximal absolute value \\
    \hline\hline
    $x^3+\left(\frac{1}{2}-\frac{\sqrt{3}}{2}i\right)x^2+1$ & $1.1509\cdots$ \\
    $x^3-\left(\frac{1}{2}+\frac{\sqrt{3}}{2}i\right)x+1$ & $1.2106\cdots$ \\
    $x^3+x^2+(1-i)x+1$ & $1.2328\cdots$ \\
    $x^3-ix^2-x+1$ & $1.2878\cdots$ \\
    $x^3-ix+1$ & $1.2966\cdots$ \\
    $x^3-(1+i)x^2+ix+1$ & $1.2969\cdots$ \\
   \hline
   \end{tabular}
  \end{table}

\clearpage
\begin{table}[hbtp]
   \caption{cubic polynomial with constant term $\frac{1}{2}+\frac{\sqrt{3}}{2}i$, with coefficients in $\mathcal{O}_{\mathbb{Q}(\sqrt{-3})}$, whose roots have modulus less than $1.3$, at least one root has modulus$>1$}
   \label{appendix 16}
   \small
   \begin{tabular}{|c|c|}
    \hline
    polynomial $p(x)$ & maximal absolute value \\
    \hline\hline
    $x^3-\left(\frac{1}{2}+\frac{\sqrt{3}}{2}i\right)x^2+\left(-\frac{1}{2}+\frac{\sqrt{3}}{2}i\right)x+\left(\frac{1}{2}+\frac{\sqrt{3}}{2}i\right)$ & $1.2167\cdots$ \\
    $x^3-\left(\frac{1}{2}-\frac{\sqrt{3}}{2}i\right)x^2+\left(-\frac{1}{2}-\frac{\sqrt{3}}{2}i\right)x+\left(\frac{1}{2}+\frac{\sqrt{3}}{2}i\right)$ & $1.2167\cdots$ \\
    $x^3+x^2+x+\left(\frac{1}{2}+\frac{\sqrt{3}}{2}i\right)$ & $1.2167\cdots$ \\
    $x^3-x+\left(\frac{1}{2}+\frac{\sqrt{3}}{2}i\right)$ & $1.2746\cdots$ \\
    $x^3+\left(\frac{1}{2}-\frac{\sqrt{3}}{2}i\right)x+\left(\frac{1}{2}+\frac{\sqrt{3}}{2}i\right)$ & $1.2746\cdots$ \\
    $x^3+\left(\frac{1}{2}+\frac{\sqrt{3}}{2}i\right)x+\left(\frac{1}{2}+\frac{\sqrt{3}}{2}i\right)$ & $1.2746\cdots$ \\
   \hline
   \end{tabular}
  \end{table}

\begin{table}[hbtp]
   \caption{totally real, quintic algebraic integer, which is larger than $2$, whose conjugates have modulus$<2.1$}
   \label{appendix 17}
   \small
   \begin{tabular}{|c|c|}
    \hline
     algebraic integer & minimal polynomial \\
    \hline\hline
    $2.0264\cdots$ & $x^5+x^4-5x^3-5x^2+4x+3$\\
    $2.0384\cdots$ & $x^5-5x^3+4x-1$\\
    $2.0431\cdots$ & $x^5-5x^3-x^2+5x+1$\\
    $2.0541\cdots$ & $x^5-6x^3+8x-1$\\
    $2.0665\cdots$ & $x^5-6x^3-x^2+8x+3$\\
    $2.0850\cdots$ & $x^5-x^4-5x^3+4x^2+5x-3$\\
    $2.0911\cdots$ & $x^5-x^4-5x^3+4x^2+4x-1$\\
   \hline
   \end{tabular}
  \end{table}

\section{the proof of the fact in Remark \ref{Remark}}\label{appendix b}
 In Remark \ref{Remark}, we presented that for $0<t<2$, the equation
 \begin{align*}
  \frac{2\mathrm{sin}(x)\mathrm{cos}(2x)}{1+2\mathrm{sin}(x)\mathrm{sin}(2x)}=\mathrm{tan}(tx)
 \end{align*}
 has at most one root in the interval $\left(0, \frac{\pi}{4}\right)$.
 We prove this fact in this section.\par
 Denote
 \begin{align*}
  f(x)=1+2\mathrm{sin}(x)\mathrm{sin}(2x),\ g(x)=2\mathrm{sin}(x)\mathrm{cos}(2x).
 \end{align*}
 Now
 \begin{align*}
  \frac{g(x)}{f(x)}>0\text{ and }\mathrm{tan}(tx)>0
 \end{align*}
 for $0<x<\frac{\pi}{4}$.
 In the $xy$ coordinate plane, the curve $C:y=\frac{g(x)}{f(x)}$ and $D_t:y=\mathrm{tan}(tx)$ both through $(x,y)=(0,0)$.
 Assume $C$ and $D_t$ intersect at least one point for $0<x<\frac{\pi}{4}$.
 Let $(x,y)=(x_1,y_1)$ be an intersection point on the very right of $C$ and $D_t$ with $0<x<\frac{\pi}{4}$.
 Denote the line passing through $(0,0)$ and $(x_1,y_1)$ by $L$.
 First,
 \begin{align*}
  \left(\frac{d}{dx}\right)^2(\mathrm{tan}(tx))=\frac{2t^2\mathrm{sin}(tx)\mathrm{cos}(tx)}{\mathrm{cos}^4(tx)}>0
 \end{align*}
 for $0<x<\frac{\pi}{4}$ and so $\mathrm{tan}(tx)$ is a convex downward function on the interval $\left(0, \frac{\pi}{4}\right)$.
 Thus, $L$ is over $D_t$ for $0<x<x_1$ and $L$ is under $D_t$ for $x_1<x<\frac{\pi}{4}$.
 Next we show that $L$ intersects with $C$ at only $(x_1,y_1)$ transversally.
 By some calculation,
 \begin{align*}
  \left(\frac{d}{dx}\right)(g'(x)f(x)-f'(x)g(x))&=g''(x)f(x)-f''(x)g(x)\\
                                                &=-10\mathrm{sin}(x)\mathrm{cos}(2x)-8\mathrm{cos}(x)\mathrm{sin}(2x)\\
                                                &\quad\quad-16\mathrm{sin}(x)\mathrm{cos}(x)\mathrm{sin}^2(2x)-16\mathrm{sin}(x)\mathrm{cos}(x)\mathrm{cos}^2(2x)\\
                                                &<0
 \end{align*}
 for $0<x<\frac{\pi}{4}$ and so 
 \begin{align*}
  g'(x)f(x)-f'(x)g(x)
 \end{align*}
 is strictly decreasing on the interval $\left(0, \frac{\pi}{4}\right)$.
 Now by 
 \begin{align*}
  f'(x)=2\mathrm{cos}(x)\mathrm{sin}(2x)+4\mathrm{sin}(x)\mathrm{cos}(2x)>0,\ f(x)>0
 \end{align*}
 for $0<x<\frac{\pi}{4}$, $f(x)$ is strictly increasing on the interval $\left(0, \frac{\pi}{4}\right)$ and so
 \begin{align*}
  \left(\frac{d}{dx}\right)\left(\frac{g(x)}{f(x)}\right)=\frac{g'(x)f(x)-f'(x)g(x)}{f(x)^2}
 \end{align*}
 is strictly decreasing on $\left\{x\in\left(0, \frac{\pi}{4}\right)\mid g'(x)f(x)-f'(x)g(x)>0\right\}$.\par
 Thus, by $\frac{y_1}{x_1}>0$, there exists at most one root of the equation
 \begin{align*}
  \left(\frac{d}{dx}\right)\left(\frac{g(x)}{f(x)}\right)=\frac{y_1}{x_1}
 \end{align*}
 for $0<x<\frac{\pi}{4}$.
 By the mean value theorem, $L$ and $C$ intersects at only $(x,y)=(x_1,y_1)$ for $0<x<\frac{\pi}{4}$ and it is transversally.
 Also by $g'(0)f(0)-f'(0)g(0)>0$,
 \begin{align*}
  \left(\frac{d^2}{{dx}^2}\right)\left(\frac{g(x)}{f(x)}\right)<0,
 \end{align*}
 on $(0,\epsilon)$ for some $\epsilon>0$ and so $C:y=\frac{g(x)}{f(x)}$ is a convex upward function on $(0,\epsilon)$.
 Thus, $C$ is over $L$ for $0<x<x_1$ and $C$ is under $L$ for $x_1<x<\frac{\pi}{4}$.\par
 Therefore, $C$ and $D_t$ intersects at only $(x_1,y_1)$ and so the proof is concluded.

\renewcommand{\refname}{References}

\end{document}